\documentclass[a4paper]{article}

\usepackage{amsthm,amsfonts,amssymb,amsmath,amsxtra}
\usepackage[all]{xy}
\usepackage{xr-hyper}
\usepackage[colorlinks=true]{hyperref}
\usepackage{makeidx}\makeindex
\usepackage{mathrsfs}

\theoremstyle{plain}
\newtheorem{thm}{Theorem}[section]
\newtheorem{lem}[thm]{Lemma}
\newtheorem{cor}[thm]{Corollary}
\newtheorem{prop}[thm]{Proposition}

\theoremstyle{definition}
\newtheorem{defn}[thm]{Definition}

\newtheorem{conj}[thm]{Conjecture}

\theoremstyle{remark}
\newtheorem{remark}[thm]{Remark}

\numberwithin{equation}{section}

    \newcommand{\FH}{{\mathbf{H}}} \newcommand{\FG}{{\mathbf{G}}}
    \newcommand{\FT}{{\mathbf{T}}}
    \newcommand{\FU}{{\mathbf{U}}}\newcommand{\FK}{{\mathbf{K}}}
    \newcommand{\FP}{{\mathbf{P}}}\newcommand{\FQ}{{\mathbf{Q}}}
    \renewcommand{\FT}{{\mathbf{T}}}\newcommand{\FM}{{\mathbf{M}}}
    \newcommand{\FA}{{\mathbf{A}}}
    \newcommand{\FZ}{{\mathbf{Z}}}
    
    \newcommand{\FL}{{\mathbf{L}}}\newcommand{\FX}{{\mathbf{X}}}

    \renewcommand{\1}{{\mathbbold{1}}}

    \newcommand{\sC}{{\mathscr{C}}}

    \newcommand{\BA}{{\mathbb {A}}} 
    \newcommand{\BC}{{\mathbb {C}}} \newcommand{\BD}{{\mathbb {D}}}

     \newcommand{\BN}{{\mathbb {N}}}
     
     \newcommand{\BR}{{\mathbb {R}}}

     \newcommand{\BZ}{{\mathbb {Z}}}

     \newcommand{\CB}{{\mathcal {B}}}
    \newcommand{\CC}{{\mathcal {C}}} \newcommand{\CD}{{\mathcal {D}}}
     \newcommand{\CF}{{\mathcal {F}}}
    \newcommand{\CG}{{\mathcal {G}}} \newcommand{\CH}{{\mathcal {H}}}
     \newcommand{\CJ}{{\mathcal {J}}}
     
     \newcommand{\CN}{{\mathcal {N}}}
    \newcommand{\CO}{{\mathcal {O}}} \newcommand{\CP}{{\mathcal {P}}}
    \newcommand{\CQ}{{\mathcal {Q}}} 
    \newcommand{\CS}{{\mathcal {S}}} 
    \newcommand{\CU}{{\mathcal {U}}} \newcommand{\CV}{{\mathcal {V}}}
     
     \newcommand{\CZ}{{\mathcal {Z}}}

    \newcommand{\gl}{{\mathfrak{gl}}}
     
    \newcommand{\fc}{{\mathfrak{c}}} \newcommand{\fd}{{\mathfrak{d}}}
    \newcommand{\fe}{{\mathfrak{e}}} \newcommand{\ff}{{\mathfrak{f}}}
    \newcommand{\fg}{{\mathfrak{g}}} \newcommand{\fh}{{\mathfrak{h}}}
     
     \newcommand{\fl}{{\mathfrak{l}}}
    \newcommand{\fm}{{\mathfrak{m}}} \newcommand{\fn}{{\mathfrak{n}}}
     \newcommand{\fp}{{\mathfrak{p}}}
    \newcommand{\fq}{{\mathfrak{q}}} \newcommand{\fr}{{\mathfrak{r}}}
    \newcommand{\fs}{{\mathfrak{s}}} \newcommand{\ft}{{\mathfrak{t}}}
    \newcommand{\fu}{{\mathfrak{u}}}

    \renewcommand{\L}{{\mathrm{L}}}

    \newcommand{\B}{{\mathrm{B}}}
    
    \newcommand{\ad}{{\mathrm{ad}}}
    \newcommand{\Ad}{{\mathrm{Ad}}}\newcommand{\A}{{\mathrm{A}}}

     \newcommand{\bc}{{\mathrm{bc}}}

    \newcommand{\cl}{{\mathrm{cl}}}

    \newcommand{\diag}{{\mathrm{diag}}}\renewcommand{\d}{{\mathrm{d}}}
    \newcommand{\D}{{\mathrm{D}}}

    \newcommand{\el}{{\mathrm{ell}}}

    \newcommand{\f}{{\mathrm{f}}}

    \newcommand{\Gal}{{\mathrm{Gal}}} 
    \newcommand{\GL}{{\mathrm{GL}}}
    \newcommand{\Hom}{{\mathrm{Hom}}}

    \newcommand{\JL}{{\mathrm{JL}}}

    \newcommand{\Lie}{{\mathrm{Lie}}}

     \newcommand{\N}{{\mathrm{N}}}

    \renewcommand{\P}{{\mathrm{P}}}

    \renewcommand{\mod}{\ \mathrm{mod}\ }

    \newcommand{\reg}{{\mathrm{reg}}}\newcommand{\Res}{{\mathrm{Res}}}
    \newcommand{\res}{{\mathrm{res}}}\newcommand{\rs}{{\mathrm{rs}}}

    \renewcommand{\ss}{{\mathrm{ss}}}
    \newcommand{\SL}{{\mathrm{SL}}}
    \newcommand{\Spec}{{\mathrm{Spec}}}

    \newcommand{\Supp}{{\mathrm{Supp}}}

    \newcommand{\tr}{{\mathrm{tr}}}
    \newcommand{\RTr}{{\mathrm{Tr}}}

    \newcommand{\vol}{{\mathrm{vol}}}

    \newcommand{\wt}{\widetilde}
    \newcommand{\wh}{\widehat}
    
    \newcommand{\pair}[1]{\langle {#1} \rangle}

    \newcommand{\incl}{\hookrightarrow}
    \newcommand{\sk}{\medskip}
    \newcommand{\lra}{\longrightarrow}
    \newcommand{\ra}{\rightarrow} 
    
    \newcommand{\bs}{\backslash}

    \newcommand{\s}{\sk\noindent}

    \newcommand{\abs}[1]{\lvert#1\rvert}

\begin{document}

\title{On the smooth transfer for Guo-Jacquet relative trace formulae}
\date{}
\author{Chong Zhang}

\maketitle

\begin{abstract}
We establish the existence of smooth transfer for Guo-Jacquet
relative trace formulae in $p$-adic case. This kind of smooth
transfer is a key step towards a generalization of Waldspurger's
result on central values of L-functions of $\GL_2$.
\end{abstract}

\section{Introduction}\label{sec: introduction}
\paragraph{History}
The periods of automorphic forms play an important role in the study
of automorphic representations and related number theoretic
problems. For example, people believe that periods of automoprhic
forms can characterize the Langlands functoriality of automorphic
representations. Recently, Y. Sakellaridis and A. Venkatesh
\cite{sv} developed an ambitious program, the so-called relative
Langlands program, on this aspect. There are several powerful tools
to study periods. The theory of relative trace formula is one of
them, which was first studied by H. Jacquet. In \cite{ja}, Jacquet
reproved a remarkable result of J.-L. Waldspurger \cite{wa1} on
central values of L-functions of $\GL_2$ by comparing relative trace
formulae on different groups. In \cite{guo}, J. Guo and Jacquet made
a conjecture (see Conjecture \ref{conj}) generalizing Waldspurger's
result to higher rank cases.

To be precise, let $k$ be a number field, $\BA$ its ring of adeles.
Consider $\FG=\GL_{2n}$ and $\FH=\GL_n\times\GL_n$ embedded into
$\FG$ diagonally, which are reductive groups over $k$. Let $k'$ be a
quadratic field extension of $k$, $\eta$ the quadratic character of
$\BA^\times/k^\times$ attached to $k'$ by class field theory. Let
$\FZ$ be the center of $\FG$. When we say a cuspidal representation
$\pi$, we always mean that $\pi$ is irreducible and automorphic. For
a cuspidal representation $\pi$ of $\FG(\BA)$, we consider the
linear forms $\ell_\FH$ and $\ell_{\FH,\eta}$ on $\pi$ defined by
periods:
$$\ell_\FH(\phi):=\int_{\FH(k)\FZ(\BA)\bs\FH(\BA)}\phi(h)\ \d h,
\quad
\ell_{\FH,\eta}(\phi):=\int_{\FH(k)\FZ(\BA)\bs\FH(\BA)}\phi(h)\eta(h)\
\d h,$$ where $\phi\in\pi$ and $\eta(h):=\eta(\det h)$. We say that
$\pi$ is $\FH$-distinguished (resp. $(H,\eta)$-distinguished) if
$\ell_\FH\neq0$ (resp. $\ell_{\FH,\eta}\neq0$). On the other hand,
for a quaternion algebra $\D$ over $k$ containing $k'$, let
$\FG'=\FG'_\D=\GL_n(\D)$ and $\FH'=\GL_n(k')$, both viewed as
reductive groups defined over $k$. View $\FH'$ as a subgroup of
$\FG'$ in the natural way and identify the center of $\FG'$ with
$\FZ$. For a cuspidal representation $\pi'$ of $\FG'(\BA)$, consider
the linear form $\ell_{\FH'}$ on $\pi'$ defined by
$$\ell_{\FH'}(\phi):=\int_{\FH'(k)\FZ(\BA)\bs\FH'(\BA)}\phi(h)\ \d
h,\quad \phi\in\pi'.$$ We say that $\pi'$ is $\FH'$-distinguished if
$\ell_{\FH'}\neq0$.

Denote by $X(k',k)$ the set of quaternion algebras $\D$ over $k$
containing $k'$. For a cuspidal representation $\pi$ of $\FG(\BA)$,
denote by $X(k',k;\pi)$ the subset of $X(k',k)$ such that the
Jacquet-Langlands correspondence $\pi'_\D:=\JL(\pi)$ of $\pi$ exists
as a cuspidal representation of $\FG'_\D(\BA)$.

Motivated by Waldspurger's result in the case $n=1$, the following
conjecture was made in \cite{guo}.

\begin{conj}[(Guo-Jacquet)]\label{conj}
Fix a cuspidal representation $\pi$ of $\FG(\BA)$.
\begin{enumerate}
\item
Fix a quaternion algebra $\D$ in $X(k',k;\pi)$. Suppose that
$\pi'_\D$ is $\FH'$-distinguished. Then $\pi$ is both
$\FH$-distinguished and $(\FH,\eta)$-distinguished.
\item Suppose that $n$ is odd and $\pi$ is both $\FH$-distinguished
and $(\FH,\eta)$-distinguished. Then there exists $\D\in
X(k',k;\pi)$ such that $\pi'_\D$ is $\FH'$-distinguished.
\end{enumerate}
\end{conj}

Moreover, when $n$ is even, with more restriction, the direction
(ii) of Conjecture \ref{conj} should also hold. We refer the reader
to \cite[Conjecture 3]{fm} and \cite[Conjecture 1.5]{fmw} for more
information.

The periods defined above can be used to study the central value
$L(\frac{1}{2},\pi_{k'})
=L(\frac{1}{2},\pi)L(\frac{1}{2},\pi\otimes\eta)$ where $\pi_{k'}$
is the base change of $\pi$ to $\FG(\BA_{k'})$. It was shown in
\cite{fj} that if $\pi$ is both $\FH$-distinguished and
$(\FH,\eta)$-distinguished then $L(\frac{1}{2},\pi_{k'})\neq0$. One
also expects that there exists a relation between this L-value and
the period $\ell_{\FH'}$ on $\pi'$.

In \cite{guo}, a relative trace formula approach called Guo-Jacquet
relative trace formulae today, which is a natural extension of
Jacquet's method in \cite{ja}, was proposed to attack the above
conjecture. The first step, that is, the fundamental lemma for unit
Hecke functions, has also been established by \cite{guo}. The smooth
transfer can be viewed as the second step on the geometric side of
Guo-Jacquet relative trace formulae. Since we only focus on the
smooth transfer, which is a local issue, we will not recall the
precise form of Guo-Jacquet relative trace formulae, which is a
global issue. We refer the reader to \cite{guo} or \cite{fmw} for
more details.

Very recently, B. Feigon, K. Martin and D. Whitehouse \cite{fmw}
obtained some partial results on Conjecture \ref{conj}, by using a
simple form of Guo-Jacquet trace formulae. They showed the existence
of smooth transfer for Bruhat-Schwartz functions satisfying certain
specific properties. Of course, one has to show the existence of
smooth transfer for the full space of Bruhat-Schwartz functions, if
one aims to prove Conjecture \ref{conj} completely. Due to our
result, one can remove some conditions of the results in \cite{fmw},
as \cite[Remark 6.2]{fmw} states.

There is also a generalization of Waldspurger's result in another
direction: the so called Gan-Gross-Prasad conjecture \cite{ggp} and
the refined version of it by Ichino-Ikeda \cite{ii} in the case of
orthogonal groups and by N. Harris \cite{ha} in the case of unitary
groups. Recently, W. Zhang (\cite{zh},\cite{zh2}) has made a great
advance towards the global Gan-Gross-Prasad conjecture for unitary
groups by using the relative trace formula developed by Jacquet and
S. Rallis. One of his achievements is his proof of the smooth
transfer conjecture in $p$-adic case for the Jacquet-Rallis relative
trace formula. His method is close to that of \cite{ja2}. The
several remarkable successes on the Gan-Gross-Prasad conjecture,
both in local and global directions, will shed some light on the
problem considered here.

\paragraph{Results of this article}
In this article, we establish the existence of smooth transfer in
$p$-adic case for Guo-Jacquet relative trace formulae. Let us
briefly explain what the smooth transfer means. From now on, let $F$
be a $p$-adic field, which is a completion of $k$ at a finite place.
Let $E$ be a quadratic field extension of $F$ and $\D$ a quaternion
algebra over $F$ containing $E$. Notice that such quaternion
algebras are parameterized by $F^\times/\N E^\times$, where $\N$ is
the norm map from $E^\times$ to $F^\times$. When we want to
emphasize the dependence of $\D$ on $\epsilon\in F^\times/\N
E^\times$, we write $\D_\epsilon$. Let $\eta$ be the quadratic
character of $F^\times$ associated to $E/F$. We define $(\FG,\FH)$
and $(\FG',\FH')$ over $F$ in the same way as the global situation.
Write $G=\FG(F),\ H=\FH(F),\ G'=\FG'(F)$ and $H'=\FH'(F)$.

The group $H\times H$ (resp. $H'\times H'$) acts on $G$ (resp. $G'$)
by left and right translations. With respect to this action, we can
talk about the notion of $H\times H$- or $H'\times H'$-regular
semisimple (cf. \S3.1) elements in $G$ or $G'$ respectively. Denote
by $G_\rs$ and $G'_\rs$ the set of the regular and semisimple
elements in $G$ and $G'$ respectively. Then there is a natural
injection (cf. Proposition \ref{prop. matching of orbits})
$$[G'_\rs]\incl [G_\rs]$$ from the set of $H'\times H'$-orbits
in $G'_\rs$ to the set of $H\times H$-orbits in $G_\rs$. We say that
$x\in G_\rs$ matches $y\in G'_\rs$ and write $x\leftrightarrow y$ if
the orbit of $y$ goes to that of $x$ under this injection. We say
that $x\in G_\rs$ comes from $G'_\rs$ if there exists $y\in G'_\rs$
such that $x\leftrightarrow y$. If $x\leftrightarrow y$, their
stabilizers denoted by $(H\times H)_x$ and $(H'\times H')_y$ are
isomorphic. Fix a Haar measure on $H$ and a Haar measure on
$(H\times H)_x$ for each $x\in G_\rs$. Note that $\eta|_{(H\times
H)_x}=1$. For each $f\in\CC_c^\infty(G)$, define the orbital
integral of $f$ at $x$ to be
$$O^\eta(x,f)=\int_{(H\times H)_x\bs H\times H}
f(h_1^{-1}xh_2)\eta(\det h_2)\ \d h_1\ \d h_2.$$ We can define a
transfer factor $\kappa$ (cf. Defintion \ref{defn. transfer factor})
which is a function on $G_\rs$ so that
$\kappa(\cdot)O^\eta(\cdot,f)$ only depends on the $H\times
H$-orbits in $G_\rs$. Similarly, fix a Haar measure on $H'$. We fix
the Haar measure on $(H'\times H')_y$ for each $y\in G'_\rs$ so that
it is compatible with that on $(H\times H)_x$ if $x\leftrightarrow
y$. For each $f'\in\CC_c^\infty(G')$, define the orbital integral of
$f'$ at $y$ to be
$$O(y,f')=\int_{(H'\times H')_y\bs H'\times H'}f'(h_1^{-1}yh_2)
\ \d h_1 \d h_2.$$ For $f\in\CC_c^\infty(G)$ and
$f'\in\CC_c^\infty(G')$, we say that $f$ and $f'$ are smooth
transfer of each other if
$$\begin{array}{lll}\kappa(x)O^\eta(x,f)=\left\{\begin{array}{ll}
\begin{aligned}O(y,f'),&\quad \textrm{if there exists }
y\in G'_\rs\textrm{ such that }x\leftrightarrow y,\\
0,&\quad \textrm{otherwise}.
\end{aligned}
\end{array}\right.
\end{array}$$
Denote by $\CC_c^\infty(G)_0$ the subspace of elements $f$ in
$\CC_c^\infty(G)$ satisfying that $O^\eta(x,f)=0$ for any $x\in
G_\rs$ that does not come from $G'_\rs$.

Our main result is the following theorem.
\begin{thm}\label{thm. partial smooth transfer groups}
For each $f'\in\CC_c^\infty(G')$, there exists $f\in\CC_c^\infty(G)$
that is a smooth transfer of $f'$. Conversely, for each
$f\in\CC_c^\infty(G)_0$, there exists $f'\in\CC_c^\infty(G')$ that
is a smooth transfer of $f$.
\end{thm}

Now we explain how to reduce the existence of smooth transfer for
functions on groups to the existence of smooth transfer for
functions on symmetric spaces. This reduction is a standard
procedure.

There is an involution $\theta$ on $\FG$ such that $\FH=\FG^\theta$
is the subgroup of $\FG$ fixed by $\theta$. Let $S:=G/H$ be the
$p$-adic symmetric space associated to $(\FG,\FH)$. The group $H$
acts on $S$ by the conjugate action. There is a symmetrization map
$s:\FG\ra\FG^\iota$, where $\iota$ is the anti-involution on $\FG$
defined by $\iota(g)=\theta(g^{-1})$ and $\FG^\iota$ is the subgroup
fixed by $\iota$. The symmetrization map is given by
$s(g)=g\iota(g)$. Via the map $s$, we view $S$ as a subset of
$\FG^\iota(F)$. An element $g\in G$ is $H\times H$-regular
semisimple if and only if $x=s(g)\in S$ is $H$-regular semisimple.
Denote by $S_\rs$ the subset of regular semisimple elements in $S$.
Let $q:\CC_c^\infty(G)\ra\CC_c^\infty(S)$ be the natural surjection
map defined by $$(qf)(x)=\int_H f(gh)\ \d h$$ if $x=s(g)$. Let
$x=s(g)\in S$ be regular semisimple. Then its stabilizer $H_x$ is
isomorphic to $(H\times H)_g$. We choose the same Haar measure on
$H$ as before and the Haar measure on $H_x$ compatible with that on
$(H\times H)_g$. For $\wt{f}\in\CC_c^\infty(S)$, define the orbital
integral of $\wt{f}$ at $x$ to be $$O^\eta(x,\wt{f})=\int_{H_x\bs
H}\wt{f}(h^{-1}xh)\eta(\det h)\ \d h.$$ We define a transfer factor
on $S_\rs$ so that $\kappa(x)=\kappa(g)$ if $x=s(g)$. Then, by a
routine computation, we have
$$\kappa(g)O^\eta(g,f)=\kappa(x)O^\eta(x,\wt{f})$$
for each $f\in\CC_c^\infty(G)$, $\wt{f}=qf\in\CC_c^\infty(S)$ and
$x=s(g)\in S_\rs$. Thus, the study of orbital integrals for
$\CC_c^\infty(G)$ with respect to $H\times H$-action is equivalent
to that of orbital integrals for $\CC_c^\infty(S)$ with respect to
$H$-action. Similarly, the study of orbital integrals for
$\CC_c^\infty(G')$ with respect to $H'\times H'$-action is
equivalent to that of orbital integrals for $\CC_c^\infty(S')$ with
respect to $H'$-action, where $S':=G'/H'$ is the $p$-adic symmetric
space associated to $(\FG',\FH')$.

There is a natural injection (cf. Proposition \ref{prop. matching of
orbits}) $$[S'_\rs]\incl [S_\rs]$$  from the set of $H'$-orbits in
$S'_\rs$ to the set of $H$-orbits in $S_\rs$. We say that $x\in
S_\rs$ matches $y\in S'_\rs$ and write $x\leftrightarrow y$ if the
orbit of $y$ goes to that of $x$. Similarly, for
$f\in\CC_c^\infty(S)$ and $f'\in\CC_c^\infty(S')$, we can define the
notion of smooth transfer for them (see \S5.1 for more details).
Then we immediately see that Theorem \ref{thm. partial smooth
transfer groups} is equivalent to Theorem \ref{thm. partial smooth
transfer 2} which claims the existence of smooth transfer at the
level of symmetric spaces.

There is also the notion of smooth transfer at the level of Lie
algebras, called the Lie algebra version of smooth transfer. Here,
for Lie algebras, we mean the tangent spaces $\fs$ and $\fs'$ of
$\FG/\FH$ and $\FG'/\FH'$ at the identity respectively. The notion
of smooth transfer in this version is determined by the orbital
integrals with respect to adjoint actions of $H$ and $H'$ on
$\fs(F)$ and $\fs'(F)$ respectively. We refer the reader to
\S\ref{section. smooth transfer} for more details.

Our method to showing the existence of smooth transfer is mainly
inspired by Zhang's work \cite{zh} on the smooth transfer for the
Jacquet-Rallis relative trace formula and Waldspurger's work
\cite{wa97} on the endoscopic smooth transfer for Arthur's stable
trace formula. First, we reduce Theorem \ref{thm. partial smooth
transfer 2} to Theorem \ref{thm. partial smooth transfer 3} which
claims the existence of smooth transfer at the level of Lie
algebras. Next we reduce Theorem \ref{thm. partial smooth transfer
3} to Theorem \ref{thm. fourier preserve transfer} which asserts
that the Fourier transform preserves smooth transfer up to a nonzero
scalar. The several reduction steps here almostly follow those of
\cite{zh}. To prove Theorem \ref{thm. fourier preserve transfer},
since the absence of a suitable partial Fourier transform, we could
not adapt the inductive argument in \cite[\S4.4]{zh} any more. Our
approach is more close to that of \cite{wa97} where a global
argument emerged. However there are still some differences between
our method and that of \cite{wa97}. These differences are caused by
the following facts. The first fact is that
$$[\fs_\rs(F)]\supseteqq\bigcup_{\epsilon\in F^\times/\N
E^\times}[\fs'_{\epsilon,\rs}(F)],$$ where $\fs'_\epsilon$ is the
Lie algebra associated to $(\FG_\epsilon'=\GL_n(\D_\epsilon),\FH')$
and $[\fs'_{\epsilon,\rs}(F)]$ (resp. $[\fs_\rs(F)]$) is the set of
$H'$- (resp. $H$)-regular semisimple orbits. The above two sets are
equal if and only if $n=1$. Even worse, the elliptic parts of the
above two sets are equal if and only $n$ is odd. These phenomenons
are unlike other cases of relative trace formulae. Now suppose that
we are in the global setting. The second fact is that if $X_0$ is a
global element in $\fs_\rs(k)$ which does not come from
$\fs'_\rs(k)$ then there exist at least two places $v_1,v_2$ such
that $X_0$ does not come from $\fs'_\rs(k_{v_1})$ or
$\fs'_\rs(k_{v_2})$. This is unlike the case of endoscopic transfer
and prevents us to use global method to prove Theorem \ref{thm. half
fourier 1} which asserts that the orbital integral
$O^\eta(X,\wh{f})=0$ for $X\in\fs_\rs(F)$ not coming from
$\fs'_\rs(F)$ where $f\in\CC_c^\infty(\fs(F))$ is a smooth transfer
of some element in $\CC_c^\infty(\fs'(F))$ and $\wh{f}$ is its
Fourier transform. Instead we will use a pure local argument, which
is due to the referee, to show Theorem \ref{thm. half fourier 1}.

To prove Theorem \ref{thm. fourier preserve transfer}, we have to
show the representability of the Fourier transform of orbital
integrals as distributions (see Theorem \ref{thm.
representability}), exhibit ``limit formulae'' for the kernel
functions (see Proposition \ref{prop. i(X,Y)}) as Waldspurger did in
\cite{wa95}, and also prove analogues of some results (see
Proposition \ref{prop. local prop} and Theorem \ref{thm.
convergence}) in \cite{wa97}. These results, which are on harmonic
analysis on certain $p$-adic symmetric spaces, maybe appear in the
literature for the first time. We expect that the techniques
developed in this paper should be probably generalized to treat some
other similar open questions concerning relative trace formulae for
symmetric pairs. Actually, we do successfully generalize this method
to prove the existence of smooth transfer for other relative trace
formula in \cite{zhc}. Here we mention some cases of symmetric pairs
where our results in \S\ref{sec: representability} and \S\ref{sec:
local calculations} should hold. Still let $E$ be a quadratic field
extension of a $p$-adic field $F$. The first class of symmetric
pairs are ``inner forms'' of $(\FG,\FH)$ or $(\FG',\FH')$. Now let
$\D$ be a central division algebra over $F$. Let $\FG=\GL_{2m}(\D)$
and $\FH=\GL_m(\D)\times\GL_m(\D)$. Then $(\FG,\FH)$ is the
symmetric pair considered in \cite{zhc}. We can also consider the
symmetric pair $(\FG,\FH)=(\GL_{2m}(\D),\GL_m(\D\otimes_FE))$, or,
more generally, the symmetric pair
$(\FG,\FH)=(\GL_m(\D),\GL_m(\D'))$ where $\D$ is a central simple
algebra over $F$ containing $E$ and $\D'$ is the centralizer of $E$
in $\D$. The second class of symmetric pairs are Galois symmetric
pairs. Now let $\FH$ be a connected reductive group over $F$, and
$\FG=\Res_{E/F}(\FH_E)$ the Weil restriction of the base change of
$\FH$ to $E$. Then $(\FG,\FH)$ is called a Galois symmetric pair.

\paragraph{Structure of this article}
In \S\ref{sec: Notations}, we introduce some notations and
conventions that are frequently used in the paper.

In \S\ref{sec: symmetric pairs 1}, since $(\FG,\FH)$ and
$(\FG',\FH')$ are symmetric pairs, we collect some basic notions and
results on symmetric pairs. In particular, we recall the analytic
Luna Slice Theorem which plays an pivotal role on the reduction
steps of the smooth transfer.

In \S\ref{sec: symmetric pairs 2}, we study our specific symmetric
pairs $(\FG,\FH)$ and $(\FG',\FH')$ more concretely. We give a
complete description of all the descendants of the corresponding
symmetric spaces and their Lie algebras. We also prove Propositions
\ref{prop. inequality for nilpotent} and \ref{prop. inequality for
nilpotent 2}, which are about two inequalities. These inequalities
are crucial for bounding the orbital integrals later (see Theorem
\ref{thm. upper bound}).

In \S\ref{section. smooth transfer}, we introduce the main issue of
this article, that is, the smooth transfer at the level of symmetric
spaces and its Lie algebra version. We explain why Theorem \ref{thm.
fourier preserve transfer} implies Theorem \ref{thm. partial smooth
transfer 2}. We also prove the fundamental lemma in the Lie algebra
version, which is crucial for our global approach to prove Theorem
\ref{thm. fourier preserve transfer}.

In \S\ref{sec: representability}, to prove Theorem \ref{thm. fourier
preserve transfer}, we pay more effort on studying the Fourier
transform of orbital integrals. One of the most important question
is to show the representability, that is, the Fourier transform of
an orbital integral considered as a distribution can be represented
by a locally integrable kernel function. We deal with this issue in
this section. The representability itself is also a fundamental
question in harmonic analysis on $p$-adic symmetric spaces.

\S\ref{sec: local calculations} is devoted to showing limit formulae
for the kernel functions of the Fourier transform, which is an
analogue of \cite[Section VIII]{wa95}. We also construct certain
good test functions which are smooth transfer of each other and
whose Fourier transforms are also smooth transfer of each other up
to a scalar. This construction is an analogue of \cite[Proposition
8.2]{wa97}. Such test functions are used in the later construction
of certain global Schwartz functions.

Finally, in \S\ref{sec: proof of the thm}, we finish the proof of
Theorem \ref{thm. fourier preserve transfer}, basing on the results
of \S7 and the fundamental lemma.

\section{Notations and conventions}\label{sec: Notations}
We now introduce some notations and conventions, which are
frequently used in \S3--\S7.

\paragraph{Fields}
Let $F$ be a non-archimedean local field of characteristic $0$, with
finite residue field. Fix an algebraic closure $\bar{F}$, and denote
by $\Gamma_F=\Gal(\bar{F}/F)$ the absolute Galois group. We denote
by $|\cdot|_F$ (resp. $v_F$) the absolute value (resp. the
valuation) of $F$, and extend them to $\bar{F}$ in the usual way.
Let $\CO_F$ be the integer ring of $F$ and fix a uniformizer
$\varpi$ of $\CO_F$. For a finite extension field $L$ of $F$, denote
by $\N_{L/F}$ and $\RTr_{L/F}$ the norm and trace maps respectively.
Throughout this article, we fix a nontrivial additive unitary
character $\psi:F\ra\BC^\times$.

\paragraph{Varieties and groups}
All the algebraic varieties and algebraic groups that we consider
are defined over $F$ except in \S8. We always use bold letter to
denote an algebraic group, italic letter to denote its $F$-rational
points, and Fraktur letter to denote its Lie algebra. For example,
let $\FG$ be a reductive group. We write $G=\FG(F)$ and denote by
$\fg$ the Lie algebra of $\FG$. By a subgroup of $\FG$, we mean a
closed $F$-subgroup. We write $N_G(\cdot)$ for the normalizer and
$Z_G(\cdot)$ for the centralizer of a certain set in $G$, and write
$\FZ$ for the center of $\FG$. For an algebraic variety $\FX,$
$X=\FX(F)$ is equipped with the natural topology induced from $F$.
Thus, $X$ is a locally compact totally disconnected topological
space. Sometimes we treat finite dimensional vector spaces defined
over $F$ as algebraic varieties over $F$.

\paragraph{Heights}
Let $\FG$ be a reductive group and $G=\FG(F)$. Following
Harish-Chandra, we define a height function $\|\cdot\|$ on $G$
valued in $\BR_{\geq1}$. If $\FT$ is a sub-torus of $\FG$ and
$T=\FT(F)$, denote by $\|\cdot\|_{T\bs G}$ the induced height
function on $G$. The precise definitions and some important
properties of height functions are well discussed in
\cite[\S18]{ko}.

\paragraph{$\ell$-spaces}
For a group $H$ acting on a topological $X$ and for a subset
$\omega\subset X$, we denote by $\omega^H$ the set $\{h\cdot x:
x\in\omega,h\in H\}$, and by $\cl(\omega)$ the closure of $\omega$
in $X$. For an element $x\in X$, we denote by $H_x$ the stabilizer
of $x$ in $H$.

For a locally compact totally disconnected topological space $X$, we
denote by $\CC_c^\infty(X)$ the space of locally constant and
compactly supported $\BC$-valued functions, and by $\CD(X)$ the
space of distributions on $X$. For $f\in\CC_c^\infty(X)$, we denote
by $\Supp(f)$ its support. Suppose that $H$ (an $\ell$-group) acts
on $X$. Then $H$ acts on $\CC_c^\infty(X)$ by
$$(h\cdot f)(x)=f(h^{-1}\cdot x),\quad\mathrm{where}\ h\in
H,\ f\in\CC_c^\infty(X),\ x\in X,$$ and acts on $\CD(X)$ by
$$\pair{h\cdot T,f}=\pair{T,h\cdot f},\quad\mathrm{where}\
T\in\CD(X),\ f\in\CC_c^\infty(X).$$ For a locally constant character
$\eta:H\ra\BC^\times$, we say that a distribution $T\in\CD(X)$ is
$(H,\eta)$-invariant if $h\cdot T=\eta(h)T$ for each $h\in H$. We
denote by $\CD(X)^{H,\eta}$ the space of $(H,\eta)$-invariant
distributions on $X$. If $X$ is a finite dimensional space and the
Fourier transform $f\mapsto\wh{f}$ on $\CC_c^\infty(X)$ has already
been defined, for $T\in\CD(X)$, we denote by $\wh{T}$ its Fourier
transform, which is a distribution on $X$ defined by
$\wh{T}(f)=T(\wh{f})$.

\paragraph{Fourier transforms}
Let $\FG$ be a reductive group, $\fg$ its Lie algebra. Fix a
nondegenerate symmetric bilinear form $\pair{\ ,\ }$ on $\fg(F)$,
which is invariant under conjugation. For each subspace $\ff$ of
$\fg(F)$ on which the restriction of $\pair{\ ,\ }$ is
nondegenerate, we always equip this subspace with the self-dual Haar
measure with respect to the bi-character $\psi\left(\pair{\ ,\
}\right)$. Define the Fourier transform $f\mapsto\wh{f}$ on
$\CC_c^\infty(\ff)$ by
$$\wh{f}(X)=\int_\ff f(Y)\psi\left(\pair{X,Y}\right)\ \d Y.$$ Then
$\hat{\hat{f}}(X)=f(-X)$.

\paragraph{Weil index}
At last, we recall the definition of Weil index $\gamma_\psi$
associated to a quadratic space. Let $q$ be a nondegenerate
quadratic form on a finite dimensional vector space $V$ over $F$. If
$L\subset V$ is an $\CO_F$-lattice, set $i(L)=\int_L\psi(q(v)/2)\ \d
v$ and $\wt{L}=\{v\in V:\forall\ \ell\in L,\psi(q(v,\ell))=1\}$. It
is well known that, if $\wt{L}\subset 2L$, then
$\abs{i(L)}=\vol(L)^{\frac{1}{2}}\vol(\wt{L})^{\frac{1}{2}}$, and
$i(L)|i(L)|^{-1}$ is independent of $L$. We denote by
$\gamma_\psi(q)$ the value $i(L)|i(L)|^{-1}$, assuming
$\wt{L}\subset 2L$. Recall that $\gamma_\psi(q)$ is an 8th root of
unity.

\section{Symmetric pairs I: general cases}\label{sec: symmetric pairs
1}

In this section, we recall some basic theory and necessary results
for general symmetric pairs. We refer the reader to \cite{ag} and
\cite{rr} for most of the contents.

\subsection{Actions of reductive groups}
Fix a reductive group $\FH$ and an affine variety $\FX$ with an
action by $\FH$, both defined over $F$. Write $H=\FH(F)$ and
$X=\FX(F)$. Then the categorical quotient $\FX/\FH$ of $\FX$ by
$\FH$ exists. In fact, $\FX/\FH=\Spec(\CO(\FX)^\FH)$. Let $\pi$
denote the natural maps $\FX\ra\FX/\FH$ and $X\ra(\FX/\FH)(F)$.

Let $x\in X$. We say that $x$ is
\begin{itemize}
\item $\FH$-semisimple or $H$-semisimple
if $\FH x$ is Zariski closed in $\FX$ (or equivalently, $Hx$ is
closed in $X$ for the analytic topology),
\item $\FH$-regular or $H$-regular
if the stabilizer $\FH_x$ has minimal dimension.
\end{itemize}
We usually say semisimple or regular without mentioning $H$ if there
is no confusion. Denote by $X_\rs$ (resp. $X_\ss$) the set of
regular semisimple (resp. semisimple) elements in $X$.

If $\FX$ is an $F$-rational finite dimensional representation of
$\FH$, say a point $x\in X$ nilpotent if $0\in\cl(Hx)$. Let $\CN$
denote the set of nilpotent elements in $X$, which is called the
null-cone of $\FX$. Note that $\CN=\pi^{-1}(\pi(0))$.

An open subset $U\subset X$ is called saturated if there exists an
open subset $V\subset(\FX/\FH)(F)$ such that $U=\pi^{-1}(V)$.

For $x\in X$ a semisimple element, we denote by $N^{X}_{Hx,x}$ the
normal space of $Hx$ at $x$. Then the stabilizer $H_x$ acts
naturally on the vector space $N^{X}_{Hx,x}$. We call
$(H_x,N^{X}_{Hx,x})$ the sliced representation at $x$, or the
descendent of $(H,X)$ at $x$. Then we have the following analytic
Luna Slice Theorem (cf. \cite[Theorem 2.3.17]{ag}): there exist
\begin{itemize}
\item an open $H$-invariant neighborhood $U_x$ of $Hx$ in $X$ with an
$H$-equivariant retract $p:U_x\ra Hx$,
\item and an $H_x$-equivariant embedding $\psi:p^{-1}(x)\incl
N^{X}_{Hx,x}$ with an open saturated image such that $\psi(x)=0$.
\end{itemize}
Write $Z_x=p^{-1}(x)$ and $N_x=N^{X}_{Hx,x}$. We call
$(U_x,p,\psi,Z_x,N_x)$ an analytic Luna slice at $x$. Let $y\in
p^{-1}(x)$ and $z:=\psi(y)$. Then we have (cf. \cite[Corollary
2.3.19]{ag}):
\begin{itemize}
\item $(H_x)_z=H_y$,
\item $N^X_{Hy,y}\simeq N^{N_x}_{H_xz,z}$ as $H_y$ spaces,
\item $y$ is $H$-semisimple if and only if $z$ is $H_x$-semisimple.
\end{itemize}

\subsection{Symmetric pairs}
A symmetric pair is a triple $(\FG,\FH,\theta)$ where
$\FH\subset\FG$ are reductive groups, and $\theta$ is an involution
of $\FG$ such that $\FH=\FG^\theta$ is the subgroup of fixed points.
For a symmetric pair $(\FG,\FH,\theta)$ we define an anti-involution
$\iota:\FG\ra\FG$ by $\iota(g)=\theta(g^{-1})$. Set
$\FG^\iota=\{g\in\FG;\ \iota(g)=g\}$ and define a symmetrization map
$$s:\FG\ra\FG^\iota,\quad s(g)=g\iota(g).$$ By this symmetrization
map we can view the symmetric space $S:=G/H$ as a subset of
$\FG^\iota(F)$. We consider the action of $\FH\times\FH$ on $\FG$ by
left and right translation and the action of $\FH$ on $\FG^\iota$ by
conjugation.

Let $\theta$ act by its differential on $\fg=\Lie(\FG)$. Write
$\fh=\Lie(\FH)$. Thus, $$\fh=\{X\in\fg:\ \theta(X)=X\}.$$ Put
$$\fs=\{X\in\fg:\ \theta(X)=-X\},$$ on which $\FH$ acts by adjoint
action. We also call $\fs$ the Lie algebra of $S$ for simplicity,
though, in fact $\fs$ is not a Lie algebra. We always write
$X^h=h^{-1}\cdot X=\Ad(h^{-1})X$ for $h\in\FH$ and $X\in\fs$. There
exists a $\FG$-invariant $\theta$-invariant nondegenerate symmetric
bilinear form $\pair{\ ,\ }$ on $\fg$. In particular,
$\fg=\fh\oplus\fs$ is an orthogonal direct sum with respect to
$\pair{\ ,\ }$.

Let $(\FG,\FH,\theta)$ be a symmetric pair. Let $g\in G$ be $H\times
H$-semisimple, and $x=s(g)$. Then the triple
$(\FG_x,\FH_x,\theta|_{\FG_x})$ is still a symmetric pair, and we
have (cf. \cite[Proposition 7.2.1]{ag})
\begin{itemize}
\item $x$ is semisimple (both as an element of $G$
and with respect to the $H$-action),
\item $H_x\simeq(H\times H)_g$ and $\fs_x\simeq N^G_{HgH,H}$ as
$H_x$-spaces, where $\fs_x$ is the centralizer of $x$ in $\fs(F)$.
\end{itemize}
A symmetric pair obtained in this way is called a descendant of
$(\FG,\FH,\theta)$. Note that $\fs_x$ can be identified with the Lie
algebra of $\FG_x/\FH_x$.

\paragraph{Weyl integration formula}
Let $(\FG,\FH,\theta)$ be a symmetric pair. Denote by $\fs_\rs$ the
regular and semisimple locus in $\fs$ with respect to the
$\FH$-action. We call a torus $\FT$ of $\FG$ $\theta$-split if
$\theta(t)=t^{-1}$ for all $t\in\FT$. Fix a Cartan subspace $\fc$ of
$\fs$, which by definition is a maximal abelian subspace of $\fs$
consisting of $\FH$-semisimple elements. We always assume that a
Cartan subspace is $F$-rational when we mention it. Then there is an
$F$-rational $\theta$-split torus denoted by $\FT^-$ whose Lie
algebra is $\fc$. Denote by $\fc_\reg$ the $\FH$-regular locus in
$\fc$. Let $\FT$ be the centralizer of $\fc$ in $\FH$, which is a
torus. Write $\ft=\Lie(\FT)$.

For $X\in\fc_\reg(F)$, we now introduce the factor $|D^\fs(X)|_F$.
Consider the morphism $$\beta:(\FT\bs\FH)\times\fc\lra\fs,\quad
(h,X)\mapsto X^h,$$ which is regular at $(1,X)$. The Jacobian of the
differential $\d\beta$ at $(1,X)$ is equal to
$$|D^\fs(X)|_F:=|\det(\ad(X);\fh/\ft\oplus\fs/\fc)|_F^{\frac{1}{2}}.$$

Denote by $S_\fc$ the set of roots of $\FT^-$ in $\fg(\bar{F})$. For
any $\alpha\in S_\fc$, since $\fc\subset\fs$, we have
$\theta(\alpha)=-\alpha$. Therefore $\theta$ interchanges the root
subspaces $\fg_{\alpha}$ and $\fg_{-\alpha}$. Fix a set of positive
abstract roots in $S_\fc$, and choose a basis $\{E_1,E_2,...,E_k\}$
of root vectors for the direct sum of $\fg_\alpha$ with $\alpha>0$.
Set $\fg_1=\oplus_{\alpha\in S_\fc}\fg_\alpha$ so that
$\fg=\ft\oplus\fc\oplus\fg_1$. Then over $\bar{F}$,
\begin{itemize}
\item $\{E_1,E_2,...,E_k\}\cup\{\theta(E_1),\theta(E_2),...,\theta(E_k)\}$
is a basis for $\fg_1$;
\item $\{E_1-\theta(E_1),E_2-\theta(E_2),...,E_k-\theta(E_k)\}$
is a basis for $\fs_1:=\fs\cap\fg_1$;
\item $\{E_1+\theta(E_1),E_2+\theta(E_2),...,E_k+\theta(E_k)\}$
is a basis for $\fh_1:=\fh\cap\fg_1$.
\end{itemize}
Under the adjoint action, elements of $\fc$ map $\fh_1$ to $\fs_1$
and vice versa. There is an involution $\varrho$ on $\fg_1$ whose
$+1$-eigenspace is $\oplus_{\alpha>0}\fg_\alpha$ and whose
$-1$-eigenspace is $\oplus_{\alpha<0}\fg_\alpha$. Then $\varrho$
interchanges $\fs_1$ and $\fh_1$, and $\varrho$ commutes with
$\ad(X)$ for $X$ in $\fc(F)$. Thus we have
$$|D^\fs(X)|_F=|\det(\varrho\circ\ad(X);\fh/\ft)|_F
=|\det(\varrho\circ\ad(X);\fs/\fc)|_F.$$

For a Cartan subspace $\fc$, let $M$ be its normalizer in $H$,
$W_\fc:=M/T$ be its Weyl group. The map $$(T\bs
H)\times\fc_\reg(F)\lra\fs_\rs(F)$$ obtained from $\beta$ by
restriction is a local isomorphism of $p$-adic manifolds and its
image, denoted by $\fs_\rs^\fc$, is open in $\fs(F)$. The fiber of
$\beta$ through $(h,X)\in(T\bs H)\times\fc_\reg(F)$ has $|W_\fc| $
elements. We have $$\fs_\rs(F)=\bigsqcup_\fc\fs^\fc_\rs,$$ where the
union runs over a (finite) set of representatives $\fc$ for the set
of $H$-conjugacy classes of $F$-rational Cartan subspaces in $\fs$.
Then, for $f\in\CC_c^\infty(\fs(F))$, we have the following Weyl
integration formula (cf. \cite[page 106]{rr})
$$\int_{\fs(F)} f(X)\ \d X=\sum_\fc\frac{1}{|W_\fc|}
\int_{\fc_\reg(F)}|D^\fs(X)|_F\int_{T\bs H} f(X^h)\ \d h\ \d X.$$

\paragraph{The null-cone}
Denote by $\CN$ the null-cone of $\fs(F)$ with respect to the
$H$-action. Then, by \cite[Theorem 7.3.8]{ag}, $\CN$ is also the set
of nilpotent elements (considered as elements in $\fg$) in $\fs(F)$.
It is known that $\CN$ consists of finitely many $H$-orbits. Denote
by $\CN_q$ the union of all $H$-orbits in $\CN$ of dimension $\leq
q$, which is closed in $\CN_{q+1}$.

Fix $X_0\neq0$ in $\CN$. Denote by $X_0^H$ the $H$-orbit of $X_0$,
and $\fh_{X_0}$ the centralizer of $X_0$ in $\fh(F)$. Write
$r=\dim\fh_{X_0}$. Then $X_0^H$ is of dimension $d-r$ where
$d=\dim\fh(F)$, and is open in $\CN_{d-r}$.

\begin{lem}\label{lem. sl2-triple}
There exists a group homomorphism $\phi:\SL_2(F)\ra G$ such that
$$\d\phi\left(\begin{pmatrix}0&1\\0&0\end{pmatrix}\right)=X_0,\
\d\phi\left(\begin{pmatrix}0&0\\1&0\end{pmatrix}\right)=:Y_0,\
\phi\left(\begin{pmatrix}t&0\\0&t^{-1}\end{pmatrix}\right)=:\D_t(X_0),$$
with $Y_0\in\fs(F)$ and $\D_t(X_0)\in H$.
\end{lem}
\begin{proof}
See \cite[Lemma 7.1.11]{ag}.
\end{proof}

We write
$\d(X_0)=\d\phi\left(\begin{pmatrix}1&0\\0&-1\end{pmatrix}\right)$,
which is in $\fh(F)$. Actually, we often write $\d=\d(X_0)$ when
there is no confusion. For any $X\in\fs(F)$, we denote by $\fs_X$
(resp. $\fg_X$) the centralizer of $X$ in $\fs(F)$ (resp. $\fg(F)$).
\begin{lem}\label{lem. nilp1} We have
$$\fs_{Y_0}\oplus[X_0,\fh(F)]=\fs(F),\quad \fs_{X_0}\oplus[Y_0,\fh(F)]
=\fs(F).$$
\end{lem}
\begin{proof}
We have the following decompositions (cf. \cite[page 73]{hc1})
$$\fg_{Y_0}\oplus[X_0,\fg(F)]=\fg_{X_0}\oplus[Y_0,\fg(F)]=\fg(F).$$
From the decomposition $$\fg(F)=\fh(F)\oplus\fs(F),$$ we see that
$$\fg_{X_0}=\fh_{X_0}\oplus\fs_{X_0},\quad \fg_{Y_0}=\fh_{Y_0}
\oplus\fs_{Y_0},$$ since
$$[X_0,\fh(F)]\subset\fs(F),\ [X_0,\fs(F)]\subset\fh(F),
\ [Y_0,\fh(F)]\subset\fs(F),\ [Y_0,\fs(F)]\subset\fh(F).$$ Thus we
have
$$\left(\fh_{Y_0}\oplus\fs_{Y_0}\right)
\bigoplus\left([X_0,\fs(F)]\oplus[X_0,\fh(F)]\right)
=\fh(F)\oplus\fs(F),$$ and
$$\left(\fh_{X_0}\oplus\fs_{X_0}\right)
\bigoplus\left([Y_0,\fs(F)]\oplus[Y_0,\fh(F)]\right)
=\fh(F)\oplus\fs(F).$$ Taking the $\fs$-parts of the above
identities, we prove the assertions of the lemma.
\end{proof}
Let $\Gamma$ be the Cartan subgroup of $H$ with the Lie algebra
$F\cdot\d(X_0)$. Let $\xi$ be the rational character of $\Gamma$
defined by
$$X_0^\gamma=\xi(\gamma)X_0,\quad Y_0^\gamma=\xi^{-1}(\gamma)Y_0,$$
which is not trivial. Let $r'=\dim\fs_{Y_0}$. The following lemma
essentially is a variant of \cite[Lemma 34]{hc1}, and the proof is
also similar to that of \cite[Lemma 34]{hc1}.

\begin{lem}\label{lem. eigen for nilp.}
We can choose a basis $Y_0=U_1,U_2,...,U_{r'}$ for $\fs_{Y_0}$ and
rational characters $\xi_1,\xi_2,...,\xi_{r'}$ of $\Gamma$ such that
\begin{enumerate}
\item $\xi_i^2=\xi^{\lambda_i},\lambda_i\geq0$,
\item $\ad(-\d)U_i=\lambda_i U_i$,
\item $U_i^\gamma=\xi_i^{-1}(\gamma)U_i$, for all $1\leq i\leq r'$.
\end{enumerate}
\end{lem}
Set $$m=\frac{1}{2}\left(\sum_{1\leq i\leq r'}\lambda_i\right)
=\frac{1}{2}\RTr\left(\ad(-\d)|_{\fs_{Y_0}}\right).$$

\section{Symmetric pairs II: specific cases}
\label{sec: symmetric pairs 2} Now we focus on the symmetric pairs
concerned in this article. The notations introduced here will be
used without mention from now on.

\subsection{$(\FG,\FH)$} Let $\FG=\GL_{2n}$ and
$\FH=\GL_n\times\GL_n$, both defined over $F$. $\FH$ is viewed as a
subgroup of $\FG$ by embedding it into $\FG$ diagonally. Let
$\epsilon=\begin{pmatrix}{\bf1}_n&0\\0&-{\bf1}_n\end{pmatrix}$ and
define an involution $\theta$ on $\FG$ by $\theta(g)=\epsilon
g\epsilon$. Then $\FH=\FG^\theta$, and the Lie algebra $\fs$
associated to $(\FG,\FH,\theta)$ is
$$\fs(F)=\left\{\begin{pmatrix}0&A\\B&0\end{pmatrix}:
\ A,B\in\fg\fl_n(F)\right\} \simeq\fg\fl_n(F)\oplus\fg\fl_n(F).$$ If
we identify $\fs(F)$ with $\fg\fl_n(F)\oplus\fg\fl_n(F)$, then $H$
acts on $\fs(F)$ by
$$(h_1,h_2)\cdot(A,B)=(h_1Ah_2^{-1},h_2Bh_1^{-1}).$$
Recall that we write $X^h=h^{-1}\cdot X$ for $h\in H,X\in\fs(F)$. We
fix a nondegenerate symmetric bilinear form $\pair{\ ,\ }$ on
$\fg(F)$ defined by
$$\pair{X,Y}=\tr(XY),\quad\mathrm{for}\ X,Y\in\fg(F).$$
Then $\pair{\ ,\ }$ is both $G$-invariant and $\theta$-invariant.

Since $H^1(F,\FH)$ is trivial, we have $S=\CS(F)$ where $S:=G/H$ and
$\CS:=\FG/\FH$. We identify $S$ with its image in $\FG^\iota(F)$ by
the the symmetrization map $s$. When we want to emphasize the index
$n$, we write $\FG_n,\FH_n,\theta_n$ and $\fs_n$.

\paragraph{Descendants}
Now we describe all the $H$-semisimple elements $x$ of $S$ and
$\fs(F)$ and the descendants $(H_x,\fs_x)$ at $x$. The results below
also hold when $F=k$ is a number field.

\begin{prop}\label{prop. descendant 1 group}
\begin{enumerate}
\item
Each semisimple element $x$ of $S$ is $H$-conjugate to an element of
the form
$$x(A,n_1,n_2):=\begin{pmatrix}A&0&0&A-{\bf1}_m&0&0\\
0&{\bf1}_{n_1}&0&0&0&0\\
0&0&-{\bf1}_{n_2}&0&0&0\\
A+{\bf1}_m&0&0&A&0&0\\
0&0&0&0&{\bf1}_{n_1}&0\\
0&0&0&0&0&-{\bf1}_{n_2}\end{pmatrix},$$ with $n=m+n_1+n_2$, $A\in
\fg\fl_m(F)$ being semisimple without eigenvalues $\pm1$ and unique
up to conjugation. Moreover, $x(A,n_1,n_2)$ is regular if and only
if $n_1=n_2=0$ and $A$ is regular in $\fg\fl_n(F)$.
\item Let $x=x(A,n_1,n_2)$ in $S$ be semisimple. Then the descendant
$(H_x,\fs_x)$ is isomorphic to the product (as a representation)
$$\left(\GL_m(F)_A,\fg\fl_m(F)_A\right)\times(H_{n_1},\fs_{n_1})
\times(H_{n_2},\fs_{n_2}).$$ Here $\GL_m(F)_A$ and $\fg\fl_m(F)_A$
are the centralizers of $A$ in $\GL_m(F)$ and $\fg\fl_m(F)$
respectively, and $\GL_m(F)_A$ acts on $\fg\fl_m(F)_A$ by
conjugation.
\end{enumerate}
\end{prop}
\begin{proof}
See \cite[Proposition 4.1]{jr} or \cite[Proposition 1.1]{guo} for
the first assertion. The second assertion can be proved by a direct
computation.
\end{proof}

\begin{prop}\label{prop. descendant 1 lie}
\begin{enumerate}
\item Each semisimple element $X$ of $\fs(F)$ is $H$-conjugate to an
element of the form
$$X(A)=\begin{pmatrix}0&0&{\bf1}_m&0\\0&0&0&0\\A&0&0&0\\0&0&0&0
\end{pmatrix}$$
with $A\in\GL_m(F)$ being semisimple and unique up to conjugation.
Moreover, $X(A)$ is regular if and only if $m=n$ and $A\in\GL_n(F)$
is regular.
\item Let $X=X(A)$ in $\fs(F)$ be semisimple.
Then the descendant $(H_X,\fs_X)$ is isomorphic to the product (as a
representation)
$$\left(\GL_m(F)_A,\fg\fl_m(F)_A\right)\times(H_{n-m},\fs_{n-m}).$$
\end{enumerate}
\end{prop}

\begin{proof}
See \cite[Propositions 2.1 and 2.2]{jr}.
\end{proof}

\paragraph{The null-cone}
Fix $X_0\neq 0$ in the null-cone $\CN$ of $\fs(F)$. Let
$(X_0,\d,Y_0)$ be an $\fs\fl_2$-triple as before. Recall
$\d=\d(X_0)$.

\begin{lem}\label{lem. nilp2}
We have $\dim\fs_{Y_0}=\dim\fh_{X_0}=r$.
\end{lem}
\begin{proof}
It follows from Lemma \ref{lem. nilp1} and the relation
$$\dim\fh_{X_0}+\dim[X_0,\fh(F)]=\dim\fh(F)=\dim\fs(F).$$
\end{proof}

In \cite[Lemma 3.1]{jr}, $\fh_{X_0}$ is well studied, and an upper
bound for $\RTr\left(\ad(\d)|_{\fh_{X_0}}\right)$ is given there. By
a minor modification of the discussion in \cite[\S3]{jr}, we study
the structure of $\fs_{Y_0}$. For our purpose, we want to compare
$r+m$ with $n^2+\frac{n}{2}$, where $r=\dim\fs_{Y_0}$ and
$m=\frac{1}{2}\RTr\left(\ad(-\d)|_{\fs_{Y_0}}\right)$. The following
inequalities will be used in \S\ref{subsec. majorize of orb}.

\begin{prop}\label{prop. inequality for nilpotent}
We have the relations
\begin{enumerate}
\item $r\geq n$,
\item $r+m> n^2+\frac{n}{2}$.
\end{enumerate}
\end{prop}

\begin{proof}
Write $Y=Y_0$ for short. Let $V=V_0\oplus V_1$, where
$V_i=F^{n},0\leq i\leq 1$. We identify $\fg(F)=\Hom(V,V)$,
$\fh(F)=\Hom(V_0,V_0)\oplus\Hom(V_1,V_1)$ and
$\fs(F)=\Hom(V_1,V_0)\oplus\Hom(V_0,V_1)$. Given $Y$, there is a
decomposition $V=W_1\oplus W_2\oplus\cdots\oplus W_k$, where each
$W_i$ is an indecomposable $F[Y]$-submodule. We can choose a
generator $z_i$ of $W_i$ such that $z_i$ is in either $V_0$ or
$V_1$. Define $\deg(z_i)=0$ if $z_i\in V_0$, otherwise
$\deg(z_i)=1$. Write $w_i=\dim W_i$. There is an isomorphism from
$\fs_{Y_0}$ to some space $$\CZ=\oplus_{1\leq i,j\leq k} S_{ij}.$$
Now we describe $S_{ij}$ precisely. An element $b_{ij}\in S_{ij}$ is
in $F[X]/(X^{w_j})$ of the form:
\begin{enumerate}
\item $b_{ij}(X)=\sum_{\max\{w_j-w_i,0\}\leq\ell<w_j}a_\ell^{ij}X^\ell$,
\item $a_\ell^{ij}=0$ when $\delta_i\delta_j=(-1)^\ell$, where
$\delta_i:=(-1)^{\deg(z_i)}$.
\end{enumerate}
We define an operator $\rho(\d):=X\frac{\d}{\d X}$ on $F[X]$, and an
endomorphism $\rho(\d)$ on $\CZ$ by restriction. Each $S_{ij}$ is an
invariant subspace of $\rho(\d)$. Set
$$r_{ii}=\dim S_{ii},\quad m_{ii}=\RTr\left(\rho(\d)|_{S_{ii}}\right)$$
for $1\leq i\leq k$, and
$$r_{ij}=\dim S_{ij}+\dim S_{ji},\quad m_{ij}
=\RTr\left(\rho(\d)|_{S_{ij}}\right)+\RTr\left(\rho(\d)|_{S_{ij}}\right)$$
for $1\leq i<j\leq k$. Then
$$r=\sum_{1\leq i\leq k}r_{ii}+\sum_{1\leq i<j\leq k}r_{ij},\quad
m=\sum_{1\leq i\leq k}m_{ii}+\sum_{1\leq i<j\leq k}m_{ij}.$$ The
following lemma gives a complete list of $r_{ii},r_{ij},m_{ii}$ and
$m_{ij}$. One can obtain it by the above description and a direct
computation.

\begin{lem}\label{lem. nilpotent table}
\begin{enumerate}
\item For $1\leq i\leq k$, if $w_i=2p_i$ or $2p_i+1$, we have
$$r_{ii}=p_i,\quad m_{ii}=p_i^2.$$
\item For $1\leq i<j\leq k$, we have the following table.
\begin{table}[h]
\centering
\begin{tabular}{| c | c | c | c |}
\hline $w_i,w_j$ & $\delta_i\delta_j$ & $m_{ij}$ & $r_{ij}$\\
\hline $w_i=2p_i,w_j=2p_j$ & $1$ & $2p_ip_j$ & $2\min(p_i,p_j)$\\
\hline $w_i=2p_i,w_j=2p_j$ & $-1$ & $2p_ip_j-2\min(p_i,p_j)$ &
$2\min(p_i,p_j)$\\
\hline $w_i=2p_i,w_j=2p_j+1,w_i<w_j$ & $\pm1$ & $2p_ip_j$ & $2p_i$\\
\hline $w_i=2p_i,w_j=2p_j+1,w_i>w_j$ & $1$ & $2p_ip_j+2(p_i-p_j)-1$
& $2p_j+1$\\
\hline $w_i=2p_i,w_j=2p_j+1,w_i>w_j$ & $-1$ & $2p_ip_j$ & $2p_j+1$\\
\hline $w_i=2p_i+1,w_j=2p_j+1$ & $1$ & $2p_ip_j$ &
$2\min(p_i,p_j)$\\
\hline $w_i=2p_i+1,w_j=2p_j+1$ & $-1$ & $2p_ip_j+2\sup(p_i,p_j)$ &
$2\min(p_i,p_j)+2$\\
\hline
\end{tabular}
\label{table. nilpotent}
\end{table}
\end{enumerate}
\end{lem}

Now we continue to prove the proposition.

(1) The first inequality of the proposition can be read off from the
above list. It is not hard to see that $r=n$ if and only if
$Y^{2n}=0$ and $Y^{2n-1}\neq 0$.

(2) For the second inequality, compare with the proof of \cite[Lemma
3.1]{jr}. We denote by $u$ the number of indices $i$ such that $w_i$
is odd and $\delta_i=1$, which is equal to the number of indices $j$
such that $w_j$ is odd and $\delta_j=-1$. Then
$$n=u+\sum_{1\leq i\leq k} p_i,$$ where $w_i=2p_i$ or $2p_i+1$. Thus
$$n^2+\frac{n}{2}=u^2+\frac{u}{2}+\left(2u+\frac{1}{2}\right)
\left(\sum_{1\leq i\leq k}p_i\right)+ \sum_{1\leq i\leq
k}p_i^2+2\sum_{1\leq i<j\leq k}p_ip_j.$$ On the other hand
$$r+m=\sum_{1\leq i\leq k}(r_{ii}+m_{ii})+\sum_{1\leq i<j\leq k}
(r_{ij}+m_{ij})$$ is determined by the data
$$(w_1,\delta_1,w_2,\delta_2,...,w_k,\delta_k).$$ We now induct on the
number of indices $i$ so that $w_i$ is even. First assume all the
integers $w_i$ are odd. Then it is not hard to see that
$$\begin{aligned}r+m=&\sum_{1\leq i\leq k}(p_i^2+p_i)+
2\sum_{1\leq i<j\leq k}p_ip_j +2\sum_{1\leq i<j\leq k,\
\delta_i\delta_j=1}\min(p_i,p_j)\\&+2\sum_{1\leq i<j\leq k,\
\delta_i\delta_j=-1}
\left(\sup(p_i,p_j)+\min(p_i,p_j)+1\right)\\
=&\sum_{1\leq i\leq k}(p_i^2+p_i)+2\sum_{1\leq i<j\leq k}p_ip_j
+2\sum_{1\leq i<j\leq k,\ \delta_i\delta_j=1}\min(p_i,p_j)\\&+2u
\sum_{1\leq i\leq k}p_i+2u^2\\
\geq&2u^2+\left(2u+1\right)\left(\sum_{1\leq i\leq k}p_i\right)+
\sum_{1\leq i\leq k}p_i^2+2\sum_{1\leq i<j\leq k}p_ip_j\\
>&n^2+\frac{n}{2}.\end{aligned}$$
Now we can arrange the data so that $w_k$ is even. If $k=1$, then
$w_1=2n$ and $r+m=n^2+n$ which is strictly greater than
$n^2+\frac{n}{2}$. By induction on the number of indices $i$ with
$w_i$ even, we may assume that the inequality has been proved for
the data $(w_1,\delta_1,...,w_{k-1},\delta_{k-1})$. By the induction
hypothesis, the contribution of the indices $(i,j)$ with $1\leq
i\leq j\leq k-1$ is strictly greater than $n'^2+\frac{n'}{2}$ where
$$n'=u+\sum_{1\leq i<k}p_i.$$ Therefore we have to show that the
sum of the contributions $r_{ik}+m_{ik}$ of the pairs $(i,k)$ with
$i\leq k$ is greater than or equal to
$$n^2+\frac{n}{2}-n'^2-\frac{n'}{2}
=p_k^2+\frac{p_k}{2}+2\sum_{1\leq i<k}p_ip_k+2up_k.$$ The
contribution of the pair $(k,k)$ is $p_k^2+p_k>p_k^2+\frac{p_k}{2}$.
Now consider the contribution of a pair $(i,k)$ with $i<k$. It is
always greater than or equal to $2p_ip_k$ when $w_i=2p_i$. When
$w_i=2p_i+1$, it is always greater than or equal to $2p_ip_k+2p_k$
(called good case) except when $w_k>w_i$ and $\delta_i\delta_k=-1$
(called bad case). It contributes at least $2p_ip_k$ in bad case.
However it does not matter when bad cases happen. Since if bad cases
happen $u'$ times with $\delta_i=-\delta_k$, good cases happen at
least $u'$ times with $w_j$ such that $w_j=2p_j+1$ and
$\delta_j=\delta_k$, which contribute at least
$2u'p_k+\sum_{j}2p_jp_k$. This concludes the proof of the
proposition.
\end{proof}

\subsection{$(\FG',\FH')$}
Let $E=F(\sqrt{\Delta})$ be a quadratic extension field of $F$, $\D$
a quaternion algebra over $F$ containing $E$. Let $\eta$ be the
quadratic character of $F^\times$ associated to $E$ by the local
class field theory. Denote by $\sigma$ the nontrivial element in
$\Gal(E/F)$. Sometimes we also write $x\mapsto\bar{x}$ instead of
$x\mapsto\sigma(x)$. Let $\FG'=\GL_n(\D)$, $\FH'=\GL_n(E)$, both
viewed as reductive groups defined over $F$. We can write $\FG'$ and
$\FH'$ in a more concrete form. There is a $\gamma\in F^\times$ such
that, if we denote by $\L_\gamma$ the algebra
$$\left\{\begin{pmatrix}A&\gamma B\\ \bar{B}&\bar{A}\end{pmatrix}:
\ A,B\in\fg\fl_n(E)\right\},$$ then $G'=\FG'(F)=\L_\gamma^\times$
and $H'=\FH'(F)$ consists of the ones with $B=0$. We will always
consider $\FG'$ and $\FH'$ in such a form. Note that if
$\gamma\in\N_{E/F}E^\times$, then $\FG'\simeq\GL_{2n}$. Fix a square
root $\delta$ of ${\Delta}$ in $E$. Let
$\epsilon'=\begin{pmatrix}\delta{\bf1}_n&0\\0&-\delta{\bf1}_n\end{pmatrix}$.
Define an involution $\theta$ on $\FG'$ by $\theta(g)=\epsilon'
g\epsilon'^{-1}$. Then $\FH'=\FG'^\theta$. Let $\fg'=\Lie(\FG')$,
$\fh'=\Lie(\FH')$, and $\fs'$ be the Lie algebra associated to the
symmetric pair $(\FG',\FH',\theta)$ so that $\fg'=\fh'\oplus\fs'$.
Thus, $\fs'(F)$ is
the space $$\left\{\begin{pmatrix}0&\gamma B\\
\bar{B}&0\end{pmatrix}: B\in\fg\fl_n(E)\right\}\simeq\fg\fl_n(E).$$
If we identify $\fs'(F)$ with $\fg\fl_n(E)$, then $H'=\GL_n(E)$ acts
on $\fs'(F)$ by $h\cdot X=hX\bar{h}^{-1}$, which is the
$\sigma$-twisted conjugation. We fix a nondegenerate symmetric
bilinear form $\pair{\ ,\ }$ on $\fg'(F)$ defined by
$$\pair{X,Y}=\tr(XY),\quad X,Y\in\fg'(F).$$
By definition, it is easy to see $\pair{X,Y}\in F$, and $\pair{\ ,\
}$ is both $G'$-invariant and $\theta$-invariant.

Since $H^1(F,\FH')$ is trivial, we have $S'=\CS'(F)$, where
$S':=G'/H'$ and $\CS':=\FG'/\FH'$. We identity $S'$ with its image
in $\FG'^\iota(F)$ by the symmetrization map $s$. When we want to
emphasize the index $n$, we write $\FG'_n,\FH'_n,\theta_n$ and
$\fs'_n$.

Before we continue, we recall some basic facts about the norm map in
theory of base change. If $x\in\GL_n(E)$, we write $\N(x)=x\bar{x}$,
which is called the norm of $x$. If $x\in\GL_n(E)$, $\N(x)$ is
conjugate in $\GL_n(E)$ to an element $y$ of $\GL_n(F)$, and $y$ is
uniquely determined modulo conjugation in $\GL_n(F)$. We denote by
$\N(\GL_n(E))$ the subset of elements $y$ in $\GL_n(F)$ satisfying
that there exists $x\in\GL_n(E)$ such that $y$ is conjugate to
$\N(x)$. In fact, if $y\in\N(\GL_n(E))$, there exists $x\in\GL_n(E)$
such that $y=x\bar{x}$.

\paragraph{Descendants}
Now we describe all the $H'$-semisimple elements $x$ of $S'$ and
$\fs'(F)$ and the descendants at $x$. The results below also hold
when $F=k$ is a number field.

\begin{prop}\label{prop. descendant 2 group}
\begin{enumerate}
\item Each semisimple elements $y$ of $S'$ is $H'$-conjugate to an
element of the form
$$y(A,n_1,n_2)=\begin{pmatrix}A&0&0&\gamma B&0&0\\
0&{\bf1}_{n_1}&0&0&0&0\\
0&0&-{\bf1}_{n_2}&0&0&0\\
\bar{B}&0&0&A&0&0\\
0&0&0&0&{\bf1}_{n_1}&0\\
0&0&0&0&0&-{\bf1}_{n_2}\end{pmatrix},$$ with $A\in\fg\fl_m(F)$ being
semisimple and unique up to conjugation such that
$A^2-{\bf1}_m\in\gamma\N(\GL_m(E))$ and $B\in\GL_m(E)$ is a matrix
unique up to twisted conjugation such that $A^2-{\bf1}_m=\gamma
B\bar{B}$, $AB=BA$, and $n=m+n_1+n_2$. Moreover, $y(A,n_1,n_2)$ is
regular if and only if $n_1=n_2=0$ and $A$ is regular in
$\fg\fl_n(F)$.
\item Let $y=y(A,n_1,n_2)$ in $S'$ be semisimple. Then the
descendant $(H'_y,\fs'_y)$ is isomorphic to the product (as a
representation)
$$\left(\GL_m(E)_A\cap\GL_{\sigma,m}(E)_B,
\fg\fl_m(E)_A\cap\fg\fl_m^\sigma(E)_B\right)
\times(H'_{n_1},\fs'_{n_1})\times (H'_{n_2},\fs'_{n_2}).$$ Here
$$\GL_{\sigma,m}(E)_B:=\left\{h\in\GL_m(E):\ hB=B\bar{h}\right\},$$
$$\fg\fl_m^\sigma(E)_B:=\left\{X\in\fg\fl_m(E):\ X\bar{B}=B\bar{X}\right\},$$
and $\GL_{\sigma,m}(E)_B$ acts on $\fg\fl_m^\sigma(E)_B$ by
$\sigma$-twisted conjugation.
\end{enumerate}
\end{prop}
\begin{proof}
See \cite[Proposition 1.2]{guo} for the first assertion. The second
assertion can be proved by a direct computation.
\end{proof}

\begin{prop}\label{prop. descendant 2 lie}
\begin{enumerate}
\item Each semisimple element $Y$ of $\fs'(F)$ is
$H'$-conjugate to an element of the form
$$Y(A)=\begin{pmatrix}0&0&\gamma B&0\\0&0&0&0\\
\bar{B}&0&0&0\\0&0&0&0\end{pmatrix}$$ where $A\in\GL_m(F)$ is
semisimple and unique up to conjugation such that
$A\in\gamma\N(\GL_m(E))$ and $B\in\GL_m(E)$ is a matrix unique up to
twisted conjugation such that $A=\gamma B\bar{B}$. Moreover, $Y(A)$
is regular if and only if $A\in\GL_n(F)$ is regular.
\item Let $Y=Y(A)$ in $\fs'(F)$ be semisimple. Then the descendant
$(H'_Y,\fs'_Y)$ is isomorphic to the product (as a representation)
$$\left(\GL_{\sigma,m}(E)_B,\fg\fl_m^\sigma(E)_B\right)
\times(H'_{n-m},\fs'_{n-m}).$$
\end{enumerate}
\end{prop}
\begin{proof}
See \cite[Lemma 2.1]{guo2} for the first assertion. The second
assertion can be proved by a direct computation.
\end{proof}

\paragraph{The null-cone}
Fix $X_0\neq0$ in the null-cone $\CN'$ of $\fs'(F)$. Let
$(X_0,\d,Y_0)$ be an $\fs\fl_2$-triple as before. By the same proof
as Lemma \ref{lem. nilp2}, we have $$\dim
\fs'_{Y_0}=\dim\fh'_{X_0}.$$ Write $r=\dim\fs'_{Y_0}$ and
$m=\frac{1}{2}\RTr\left(\ad(-\d)|_{\fs'_{Y_0}}\right)$. We still
want to compare $r+m$ with $n^2+\frac{n}{2}$, which is easier in
this case.

\begin{prop}\label{prop. inequality for nilpotent 2}
We have $r+m>n^2+\frac{n}{2}$ and $m'<n^2$ where
$m'=\frac{1}{2}\RTr\left(\ad(-\d)|_{\fh'_{Y_0}}\right)$.
\end{prop}
\begin{proof}
Write $Y_0=\begin{pmatrix}0&\gamma A\\ \bar{A}&0\end{pmatrix}$. If
we change $(X_0,\d,Y_0)$ to be any triple in the $H'$-orbit of
$(X_0,\d,Y_0)$, the numbers $r$ and $m$ are unchanged. By
\cite[Lemma 2.2]{guo2}, we can choose $A$ to be of the Jordan normal
form. At the same time, we can also choose $\d$ to be in
$\fg\fl_n(F)$. In such situation, it is easy to see that there is a
$\d$-equivariant isomorphism $\fs'_{Y_0}\simeq\fh'_{Y_0}$. Thus
$r=r'$ and $m=m'$, where $r'=\dim\fh'_{Y_0}$. Since
$\fg'_{Y_0}=\fh'_{Y_0}\oplus\fs'_{Y_0}$, we have
$m+m'=\frac{1}{2}\left(4n^2-r-r'\right)$. Thus we get
$m=\frac{1}{4}\left(4n^2-2r\right)$ and $r+m=n^2+\frac{r}{2}$. The
inequality $r\geq 2n$ implies the lemma.
\end{proof}

\section{Smooth transfer}\label{section. smooth transfer}

In this section, we introduce the main object of this article: the
smooth transfer between Schwartz functions on different symmetric
spaces. By several reduction steps, we explain why Theorem \ref{thm.
fourier preserve transfer} implies Theorem \ref{thm. partial smooth
transfer 2} in details.

\subsection{Definitions}
\paragraph{Matching of orbits}
We first recall the matching between semisimple orbits in symmetric
spaces $S$ and $S'$, and then give the definition of matching
between semisimple orbits in Lie algebras $\fs(F)$ and $\fs'(F)$.
These definitions of matching orbits also hold when $F=k$ is a
number field.

\begin{prop}\label{prop. matching of orbits}
\begin{enumerate}
\item For each semisimple element $y$ of $S'$, there exists $h\in
\FH(E)$ such that $hyh^{-1}$ belongs to $S$. This establishes an
injection of the $H'$-semisimple orbits in $S'$ into the
$H$-semisimple orbits in $S$, which carries the orbit of
$y(A,n_1,n_2)$ in $S'$ to the orbit of $x(A,n_1,n_2)$ in $S$.
\item For each semisimple element $Y$ of $\fs'(F)$, there exists $h\in
\FH(E)$ such that $hYh^{-1}$ belongs to $\fs(F)$. This establishes
an injection of the $H'$-semisimple orbits in $\fs'(F)$ into the
$H$-semisimple orbits in $\fs(F)$, which carries the orbit of $Y(A)$
in $\fs'(F)$ to the orbit of $X(A)$ in $\fs(F)$.
\end{enumerate}
\end{prop}
\begin{proof}
See \cite[Proposition 1.3]{guo} for the first assertion. The second
assertion can be proved in the same way.
\end{proof}

\begin{defn}\label{defn. matching of orbits}
\begin{enumerate}
\item We say that $y\in S'_\ss$ (resp. $Y\in\fs'_\ss(F)$) matches $x\in
S_\ss$ (resp. $X\in\fs_\ss(F)$) and write $x\leftrightarrow y$
(resp. $X\leftrightarrow Y$) if the above map sends the orbit of $y$
(resp. $Y$) to the orbit of $x$ (resp. $X$).
\item We say that $x\in S_\ss$ (resp. $X\in\fs_\ss(F)$)
comes from $S'_\ss$ (resp. $\fs'_\ss(F)$) if there exists $y\in
S'_\ss$ (resp. $Y\in\fs'_\ss(F)$) such that $x\leftrightarrow y$
(resp. $X\leftrightarrow Y$). We denote by $S_{\ss,0}$ (resp.
$\fs_\ss(F)_0$) the subset of elements in $S_\ss$ (resp.
$\fs_\ss(F)$) coming from $S'_\ss$ (resp. $\fs'_\ss(F)$).
\end{enumerate}
\end{defn}

\begin{remark}\label{rem. matching categorical quotient}
Denote by $\CQ$ (resp. $\CQ'$) the categorical quotient $\CS/\FH$
(resp. $\CS'/\FH'$), and by $\fq$ (resp. $\fq'$) the categorical
quotient $\fs/\FH$ (resp. $\fs'/\FH'$). The maps in Proposition
\ref{prop. matching of orbits} induce natural maps
$$\CQ'\incl\CQ,\quad\mathrm{and}\quad\fq'\incl\fq.$$
Actually, $\CQ$ is isomorphic to the affine space $\FA^n$, and the
quotient map $\pi:\CS\ra\CQ$ is given by
$$\begin{pmatrix}A&B\\ C&D\end{pmatrix}\mapsto
\left(\tr\wedge^iBC\right),\quad i=1,2,...,n.$$ The natural map
$\CQ'\incl\CQ$ is induced by
$$\CS'\lra\CQ,\quad
\begin{pmatrix}A&\gamma B\\ \bar{B}&\bar{A}\end{pmatrix}\mapsto
\left(\tr\wedge^i\gamma B\bar{B}\right),\quad i=1,...,n.$$
Similarly, $\fq$ is isomorphic to the affine space $\FA^n$, and the
quotient map $\pi:\fs\ra\fq$ is given by
$$\begin{pmatrix}0&A\\ B&0\end{pmatrix}\mapsto
\left(\tr\wedge^iAB\right),\quad i=1,2,...,n.$$ The natural map
$\fq'\incl\fq$ is induced by
$$\fs'\lra\fq,\quad
\begin{pmatrix}0&\gamma B\\ \bar{B}&\bar{0}\end{pmatrix}\mapsto
\left(\tr\wedge^i\gamma B\bar{B}\right),\quad i=1,...,n.$$
\end{remark}

\begin{remark}\label{rem. matching orbit}
A semisimple element $x=x(A,n_1,n_2)$ in $S_\ss$ comes from $S'_\ss$
if and only if $A^2-{\bf1}_m\in\gamma\N(\GL_m(E))$ where
$m=n-n_1-n_2$. A semisimple element $X=X(A)$ in $\fs_\ss(F)$ comes
from $\fs'_\ss(F)$ if and only if $A\in\gamma\N(\GL_m(E))$.
\end{remark}

\begin{remark}\label{rem. matching descent}
Suppose that $x\in S_\ss$ and $y\in S'_\ss$ match. We want to
compare $(H_x,\fs_x)$ with $(H'_y,\fs'_y)$. It suffices to assume
that $x=x(A,n_1,n_2)$ and $y=y(A,n_1,n_2)$. Thus, by Propositions
\ref{prop. descendant 1 group} and \ref{prop. descendant 2 group},
we have
$$(H_x,\fs_x)\simeq\left(\GL_m(F)_A,
\fg\fl_m(F)_A\right)\times(H_{n_1},\fs_{n_1})\times
(H_{n_2},\fs_{n_2}),$$ and
$$(H'_y,\fs'_y)\simeq\left(\GL_m(E)_A\cap\GL_{\sigma,m}(E)_B,
\fg\fl_m(E)_A\cap\fg\fl_m^\sigma(E)_B\right)\times(H'_{n_1},\fs'_{n_1})\times
(H'_{n_2},\fs'_{n_2})$$ with $A^2-{\bf 1}_m=\gamma B\bar{B}$ and
$AB=BA$. By the proof of Lemma \ref{lem. compare descendants} below,
we see that $\left(\GL_m(E)_A\cap\GL_{\sigma,m}(E)_B,
\fg\fl_m(E)_A\cap\fg\fl_m^\sigma(E)_B\right)$ essentially is an
inner form of $\left(\GL_m(F)_A, \fg\fl_m(F)_A\right)$. The other
factors in the descendants are related in a similar manner as
$(H,\fs)$ and $(H',\fs')$ are. For $X\in\fs_\ss(F)$ and
$Y\in\fs'_\ss(F)$ such that $X\leftrightarrow Y$, by Propositions
\ref{prop. descendant 1 lie} and \ref{prop. descendant 2 lie} and
Lemma \ref{lem. compare descendants}, the factors of the descendants
$(H_X,\fs_X)$ and $(H'_Y,\fs'_Y)$ have the similar relations as
above.
\end{remark}

\begin{remark}\label{rem. matching stabilizer}
It is obvious that the maps in Proposition \ref{prop. matching of
orbits} send regular semisimple orbits to regular semisimple ones.
We denote by $S_{\rs,0}$ (resp. $\fs_\rs(F)_0$) the subset of
elements in $S_\rs$ (resp. $\fs_\rs(F)$) coming from $S'_\rs$ (resp.
$\fs'_\rs(F)$). Suppose that $x\in S_\rs$ (resp. $x\in\fs_\rs(F)$)
and $y\in S'_\rs$ (resp. $y\in\fs'_\rs(F)$) match. Then by the above
remark, we see that $H_x$ is an inner form of $H'_y$. Since they are
torus, we have
$$H_x\simeq H'_y.$$
\end{remark}

\paragraph{Transfer factors}
To state our results on smooth transfer, we need to define transfer
factors for the symmetric pair $(\FG,\FH,\theta)$ and its
descendants. In general, the transfer factor is defined as follows
(cf. \cite[Definition 3.2]{zh}).

\begin{defn}\label{defn. transfer factor}
Let a reductive group $\FH$ act on an affine variety $\FX$, both
defined over $F$. Let $\eta$ be a quadratic character of $H$.
Suppose that for all regular semisimple $x\in X=\FX(F)$, the
character $\eta$ is trivial on the stabilizer $H_x$. Then a transfer
factor is a smooth function $\kappa:X_\rs\ra\BC^\times$ such that
$\kappa(x^h)=\eta(h)\kappa(x)$ for any $h\in H$.
\end{defn}

\begin{defn}\label{defn. explicit transfer factor}
For convenience, we give an explicit definition of various transfer
factors in our situation as follows:
\begin{itemize}
\item type $(H,S)$: for $x=\begin{pmatrix}A&B\\C&D\end{pmatrix}\in
S$ regular semisimple, define $\kappa(x):=\eta(\det(B))$;
\item type $(H_m,\fs_m)$: for $X=\begin{pmatrix}0&A\\B&0
\end{pmatrix}\in\fs_m(F)$ regular semisimple, define
$\kappa(X):=\eta(\det(A))$;
\item type $(\GL_m(F)_A,\fg\fl_m(F)_A)$:
we define $\kappa$ to be the constant function with value $1$.
\end{itemize}
In the cases (1) and (2), $\eta$ is the non-trivial quadratic
character on $F^\times$ associated to $E$, while in the case (3)
$\eta$ is the trivial character. In all the cases, it is easy to see
that $\eta$ is trivial on the stabilizers $H_x$.
\end{defn}

\paragraph{Smooth transfer}
Now we give the definition of smooth transfer. First, we fix Haar
measures on $H$ and $H'$. Notice that, for $x\in S_\rs$ (resp.
$x\in\fs_\rs(F)$) and $y\in S'_\rs$ (resp. $y\in\fs'_\rs(F)$) such
that $x\leftrightarrow y$, their stabilizers $H_x$ and $H'_y$ are
isomorphic to each other (see Remark \ref{rem. matching
stabilizer}), and we fix such an isomorphism. Fix a Haar measure on
$H_x$ for each $x\in S_\rs$ (resp. $x\in\fs_\rs(F)$). We fix a Haar
measure on $H'_y$ for each $y\in S'_\rs$ (resp. $y\in\fs'_\rs(F)$)
which is compatible with that of $H_x$ if $x\leftrightarrow y$.

\begin{defn}\label{defn. orbital integral}
For $x\in S_\rs$ (resp. $x\in\fs_\rs(F)$) and $f\in\CC_c^\infty(S)$
(resp. $f\in\CC_c^\infty(\fs(F))$), define the orbital integral of
$f$ at $x$ to be
$$O^\eta(x,f):=\int_{H_x\bs H}f(x^h)\eta(h)\ \d h.$$
For $y\in S'_\rs$ (resp. $y\in\fs_\rs'(F)$) and
$f'\in\CC_c^\infty(S')$ (resp. $f'\in\CC_c^\infty(\fs'(F))$), define
the orbital integral of $f'$ at $y$ to be
$$O(y,f):=\int_{H'_y\bs H'}f(x^h)\ \d h.$$
\end{defn}

\begin{defn}\label{defn. transfer}
\begin{enumerate}
\item For $f\in\CC_c^\infty(S)$ and $f'\in\CC_c^\infty(S')$, we say that
$f$ and $f'$ are smooth transfer of each other if for each $x\in
S_\rs$
$$\begin{array}{lll}\kappa(x)O^\eta(x,f)=\left\{\begin{array}{ll}
\begin{aligned}O(y,f'),&\quad \textrm{if there exists }
y\in S'_\rs\textrm{ such that }x\leftrightarrow y,\\
0,&\quad \textrm{otherwise}.
\end{aligned}
\end{array}\right.
\end{array}$$
We denote by $\CC_c^\infty(S)_0$ the subspace of elements $f$ in
$\CC_c^\infty(S)$ satisfying that $O^\eta(x,f)=0$ for any $x$ in
$S_\rs$ but not in $S_{\rs,0}$.
\item For $f\in\CC_c^\infty(\fs(F))$
and $f'\in\CC_c^\infty(\fs'(F))$, we say that $f$ and $f'$ are
smooth transfer of each other if for each $X\in \fs_\rs(F)$
$$\begin{array}{lll}\kappa(X)O^\eta(X,f)=\left\{\begin{array}{ll}
\begin{aligned}O(Y,f'),&\quad \textrm{if there exists }
Y\in \fs'_\rs(F)\textrm{ such that }X\leftrightarrow Y,\\
0,&\quad \textrm{otherwise}.
\end{aligned}
\end{array}\right.
\end{array}$$
We denote by $\CC_c^\infty(\fs(F))_0$ the subspace of elements $f$
in $\CC_c^\infty(\fs(F))$ satisfying that $O^\eta(X,f)=0$ for any
$X$ in $\fs_\rs(F)$ but not in $\fs_{\rs}(F)_0$.
\end{enumerate}
\end{defn}

\begin{remark}\label{rem. transfer and measure}
The definition of smooth transfer depends on the Haar measures on
$H$ and $H'$, but the existence of smooth transfer does not depend
on them. Sometimes, we will write transfer in place of smooth
transfer for short.
\end{remark}

\begin{remark}\label{rem. transfer on descendant}
For semisimple $x\in S$ and semisimple $y\in S'$ such that
$x\leftrightarrow y$, by Remark \ref{rem. matching descent}, we can
define the notion of smooth transfer between elements in
$\CC_c^\infty(\fs_x(F))$ and those in $\CC_c^\infty(\fs'_y(F))$,
determined by the orbital integrals with respect to the action of
$H_x$ on $\fs_x(F)$, the action of $H'_y$ on $\fs'_y(F)$, and the
transfer factor $\kappa$ defined as above. Similarly, for semisimple
$X\in\fs(F)$ and semisimple $Y\in\fs'(F)$ such that
$X\leftrightarrow Y$, we can also define the notion of smooth
transfer between elements in $\CC_c^\infty(\fs_X(F))$ and those in
$\CC_c^\infty(\fs'_Y(F))$.
\end{remark}

Our main theorems are as follows.

\begin{thm}\label{thm. partial smooth transfer 2}
For each $f'\in\CC_c^\infty(S')$, there exists $f\in\CC_c^\infty(S)$
that is a smooth transfer of $f'$. Conversely, for each
$f\in\CC_c^\infty(S)_0$, there exists $f'\in\CC_c^\infty(S')$ that
is a smooth transfer of $f$.
\end{thm}

\begin{thm}\label{thm. partial smooth transfer 3}
For each $f'\in\CC_c^\infty(\fs'(F))$, there exists
$f\in\CC_c^\infty(\fs(F))$ that is a smooth transfer of $f'$.
Conversely, for each $f\in\CC_c^\infty(\fs(F))_0$, there exists
$f'\in\CC_c^\infty(\fs'(F))$ that is a smooth transfer of $f$.
\end{thm}

In the later subsections, we will show that Theorem \ref{thm.
partial smooth transfer 3} implies Theorem \ref{thm. partial smooth
transfer 2}.

\begin{lem}\label{lem. case epsilon=1}
To prove Theorem \ref{thm. partial smooth transfer 3}, it suffices
to prove it for the case $\fs=\fs_\epsilon$ when $\epsilon=1$.
\end{lem}
\begin{proof}
Let $$\fs'(F)=\left\{Y(B)=\begin{pmatrix}   0&B\\
\bar{B}&0
\end{pmatrix}:B\in\fg\fl_n(E)\right\}.$$ Choose a
representative $\gamma\in F^\times$ of the nontrivial element in
$F^\times/\N E^\times$. Let
$$\fs_\gamma'(F)=\left\{Y_\gamma(B)=\begin{pmatrix}   0&\gamma B\\
\bar{B}&0
\end{pmatrix}:B\in\fg\fl_n(E)\right\}.$$ Identify $H'$ with $\GL_n(E)$.
Then there is a natural $H'$-equivariant isomorphism
$$j:\fs'(F)\stackrel{\sim}{\ra}\fs'_\gamma(F),\ Y(B)\mapsto
Y_\gamma(B),$$ which implies the lemma.
\end{proof}

\paragraph{Fourier transform}
Define the Fourier transform $f\mapsto\wh{f}$ on
$\CC_c^\infty(\fs(F))$ (resp. $\CC_c^\infty(\fs'(F))$) with respect
to the fixed bilinear form $\pair{\ ,\ }$ and the additive character
$\psi$. The following theorem is the key point in proving the
existence of smooth transfer.

\begin{thm}\label{thm. fourier preserve transfer}
There exists a nonzero constant $c\in\BC$ such that if
$f\in\CC_c^\infty(\fs(F))$ and $f'\in\CC_c^\infty(\fs'(F))$ are
smooth transfer of each other then $\wh{f}$ and $c\wh{f}'$ are also
smooth transfer of each other.
\end{thm}

In the later subsections, we will prove the following main result of
this section.

\begin{prop}\label{prop. conjecture 2 implies conjecture 1}
Theorem \ref{thm. fourier preserve transfer} implies Theorem
\ref{thm. partial smooth transfer 3}.
\end{prop}

\subsection{Fundamental lemma}
In this subsection, we prove the following fundamental lemma (Lemma
\ref{lem. fund lem}). This is an important example of the smooth
transfer and also a crucial lemma for us to prove Theorem \ref{thm.
fourier preserve transfer} by using global method.

Now assume that $\gamma=1$. Thus, $\FG'$ is isomorphic to $\FG$.
Suppose that $F$ is of odd residual characteristic and $E$ is
unramified over $F$. We choose the Haar measures on $H$ and $H'$ so
that $\vol(\FH(\CO_F))=1$ and $\vol(\FH'(\CO_F))=1$ respectively.

Let $f_0\in\CC_c^\infty(\fs(F))$ and $f'_0\in\CC_c^\infty(\fs'(F))$
be the characteristic functions of the standard lattices
$$L=\fg\fl_n(\CO_F)\oplus\fg\fl_n(\CO_F),\quad  L'=\fg\fl_n(\CO_E)$$
respectively.

\begin{lem}\label{lem. fund lem}
$f_0$ and $f_0'$ are smooth transfer of each other.
\end{lem}

\begin{remark}
The group version of the above fundamental lemma was proved in
\cite{guo} (cf. \cite[Theorem]{guo}).
\end{remark}

\begin{proof}
Let $X\in\fs_\rs(F)$. It suffices to consider $X$ of
the form $\begin{pmatrix}0&{\bf1}_n\\
A&0 \end{pmatrix}$ with $A\in\GL_n(F)$ being regular semisimple.
Then we have
$$\begin{aligned}
\kappa(X)O^\eta(X,f_0)&=\int_{H_X(F)\bs
H(F)}f_0(h_1^{-1}h_2,h_2^{-1}Ah_1)\eta(h_1h_2)\ \d h_1\ \d h_2\\
&=\int_{\left(\GL_n(F)_A\bs\GL_n(F)\right)\times\GL_n(F)}
f_0(h_2,h_2^{-1}h_1^{-1}Ah_1)\eta(h_2)\ \d h_2\ \d h_1.\\
\end{aligned}$$
Let $K=\GL_n(\CO_F)$ and $K'=\GL_n(\CO_E)$. For
$r=(r_{i,j})\in\fg\fl_n(\bar{F})$, put
$\abs{r}=\max_{i,j}\abs{r_{i,j}}_F$. Then for $r,t\in\fg\fl_n(F)$,
the value $f(r,t)\neq0$ if and only if $\abs{r}\leq1,\abs{t}\leq1$.
Let $\Phi_A$ be the characteristic function of the set of
$(r,t)\in\GL_n(F)\times\GL_n(F)$ satisfying $\abs{r}\leq
1,\abs{t}\leq1$ and $\abs{\det(rt)}_F=\abs{\det A}_F$. Then $\Phi_A$
belongs to $\CC_c^\infty(\GL_n(F)\times\GL_n(F))$ and is
bi-$K$-invariant both for the variables $r$ and $t$. Let $\Psi_A$ be
the function on $\GL_n(F)$ defined by
$$\Psi_A(g)=\int_{\GL_n(F)}\Phi_A(h,h^{-1}g)\eta(h)\ \d h.$$ Then
$\Psi_A$ belongs to $\CC_c^\infty(\GL_n(F))$, and is
bi-$K$-invariant (that is, $\Psi_A$ is a Hecke function). We have
$$\kappa(X)O^\eta(X,f_0)=\int_{\GL_n(F)_A\bs\GL_n(F)}\Psi_A(g^{-1}Ag)
\ \d g.$$

If $Y=\begin{pmatrix} 0&B \\ \bar{B}&0 \end{pmatrix}\in\fs'_\rs(F)$,
we have
$$O(Y,f_0')=\int_{\GL_{\sigma,n}(E)_B\bs\GL_n(E)}f_0'(h^{-1}B\bar{h})
\ \d h.$$ Let $\Psi_B$ be the characteristic function of the set of
$r\in\GL_n(E)$ satisfying $\abs{r}\leq1$ and $\abs{\det
r}_F=\abs{\det B}_F$. Then $\Psi_B$ belongs to
$\CC_c^\infty(\GL_n(E))$, and is bi-$K'$-invariant. We have
$$O(Y,f_0')=\int_{\GL_{\sigma,n}(E)_B\bs\GL_n(E)}\Psi_B(h^{-1}B\bar{h})
\ \d h.$$

Denote by $$\bc:\CH(\GL_n(E),K')\lra\CH(\GL_n(F),K)$$ the base
change map between the two spaces of Hecke functions. Then, in fact,
it was shown in \cite[Corollary 3.7]{guo} (can be read off from the
proof of Proposition 3.7 loc. cit.) that $\Psi_A=0$ if
$A\notin\N(\GL_n(E))$, and $\Psi_A=\bc(\Psi_B)$ if $A=B\bar{B}$.
Recall that $\Psi_A=\bc(\Psi_B)$ implies that
$$\kappa(X)O^\eta(X,f_0)=O(Y,f'_0),\quad \mathrm{if}
\ X\leftrightarrow Y.$$ Hence the lemma follows.
\end{proof}

\subsection{Reduction steps}
The main aim of this subsection is to reduce Theorem \ref{thm.
partial smooth transfer 2} to Theorem \ref{thm. partial smooth
transfer 3}. The reduction steps here are almost the same as those
in \cite[Section 3]{zh}.

\paragraph{Descent of orbital integrals}
The following proposition essentially is \cite[Proposition
3.11]{zh}, whose proof is also valid here.

\begin{prop}\label{prop. descent of orbital integral}
Let $X$ be any one of $S,S',\fs(F)$ or $\fs'(F)$. Let $x\in X$ be
semisimple and $(U_x,p,\psi,Z_x,N_x)$ an analytic Luna slice at $x$.
Then there exists a neighborhood $\xi\subset\psi(p^{-1}(x))$ of $0$
in $N_x$ satisfying the following properties:
\begin{itemize}
\item for each $f\in\CC_c^\infty(X)$, there exists
$f_x\in\CC_c^\infty(N_x)$ such that for all regular semisimple
$z\in\xi$ with $z=\psi(y)$ we have
$$\int_{H_y\bs H}f(y^h)\eta(h)\d h=\int_{H_y\bs H_x}f_x(z^h)\eta(h)\d
h;$$
\item and conversely, for each $f_x\in\CC_c^\infty(N_x)$, there exists
$f$ in $\CC_c^\infty(X)$ such that above equality holds for any
regular semisimple $z\in\xi$.
\end{itemize}
Here $H=H'$ and $\eta=\bf1$ when $X$ is $S'$ or $\fs'(F)$.
\end{prop}

\paragraph{Reduction to local transfer}
Recall that we denote by $\CQ$ (resp. $\CQ'$) the categorical
quotient $\CS/\FH$ (resp. $\CS'/\FH'$), and by $\fq$ (resp. $\fq'$)
the categorical quotient $\fs/\FH$ (resp. $\fs'/\FH'$). By Remark
\ref{rem. matching categorical quotient}, we always view $\CQ'$ and
$\fq'$ as closed subsets of $\CQ$ and $\fq$ respectively. Let $\FX$
be any one of $\CS,\CS',\fs$ or $\fs'$, and $\FQ$ the quotient
$\CQ,\CQ',\fq$ or $\fq'$ of $\FX$. Let $\FQ(F)_\rs$ be the regular
semisimple locus in $\FQ(F)$. Since $H^1(F,\FH)=H^1(F,\FH')=1$, the
natural map $\pi:\FX(F)\ra\FQ(F)$ is a surjection. For
$x\in\FQ(F)_\rs$, the fiber $\pi^{-1}(x)$ consists of precisely one
orbit.

\begin{defn}\label{defn. local orbital integral}
Let $\FX$ and $\FQ$ be as above. Write $X=\FX(F)$ and $Q=\FQ(F)$.
\begin{enumerate}
\item Let $\Phi$ be a function on
$Q_\rs$ which vanishes outside a compact set of $Q_\rs$. For $x\in
Q$, we say that $\Phi$ is a local orbital integral around $x$, if
there exists a neighborhood $U$ of $x$ and a function
$f\in\CC_c^\infty(X)$ such that for all $y\in U_\rs$ and $z$ with
$\pi(z)=y$ we have
$$\Phi(y)=\kappa(z)O^\eta(z,f).$$
\item For $f\in\CC_c^\infty(X)$,
define a function $\pi_*(f)$ on $Q_\rs$ to be:
$$\pi_*(f)(x)=\kappa(y)O^\eta(y,f),\quad \textrm{for } x\in Q_\rs,
\ y\in\pi^{-1}(x).
$$
\end{enumerate}
Here $\kappa=\bf1$ and $\eta=\bf1$ when $\FX$ is $\CS'$ or $\fs'$.
\end{defn}

The following result is \cite[Proposition 3.8]{zh}.

\begin{prop}\label{prop. local orbital integral}
Let $\Phi$ be a function on $Q_\rs$ which vanishes outside a compact
set $\Xi$ of $Q$. If $\Phi$ is a local orbital integral at each
$x\in\Xi$, it is an orbital integral. Namely there exists
$f\in\CC_c^\infty(X)$ such that for all $y\in Q_\rs$, and $z$ with
$\pi(z)=y$ we have
$$\Phi(y)=\kappa(z)O^\eta(z,f).$$
\end{prop}

\begin{defn}\label{defn. local transfer}
For $x\in\CQ(F)$ (resp. $x\in\fq(F)$), we say that local transfer
around $x$ exists, if for each $f'\in\CC_c^\infty(S')$ (resp.
$f'\in\CC_c^\infty(\fs'(F))$), there exists $f\in\CC_c^\infty(S)_0$
(resp. $f\in\CC_c^\infty(\fs(F))_0$) such that in a neighborhood $U$
of $x$, the following equality holds:
$$\pi_*(f)=\pi_*(f')\ \textrm{on $U\cap\CQ(F)_\rs$
(resp. $U\cap\fq(F)_\rs$)},$$ and conversely for each
$f\in\CC_c^\infty(S)_0$ (resp. $f\in\CC_c^\infty(\fs(F))_0$), there
exists $f'\in\CC_c^\infty(S')$ (resp. $f'\in\CC_c^\infty(\fs'(F))$)
satisfying the above equality.
\end{defn}

\begin{cor}\label{cor. local transfer}
To prove Theorems \ref{thm. partial smooth transfer 2} and \ref{thm.
partial smooth transfer 3}, it suffices to prove the existence of
local transfer around all elements of $\CQ(F)$ and $\fq(F)$.
\end{cor}

\begin{proof}
This is a direct consequence of Proposition \ref{prop. local orbital
integral}.
\end{proof}

\paragraph{Reduction to local transfer around zero}

\begin{lem}\label{lem. reduction to sliced}
To prove the existence of local transfer around an element $z$ in
$\CQ(F)$ (resp. $\fq(F)$), it suffices to prove the existence of
smooth transfer for the sliced representations
$\left(H_x,\fs_x\right)$ and $\left(H'_y,\fs'_y\right)$ where $x$ in
$S_\ss$ (resp. $\fs_\ss(F)$) and $y$ in $S'_\ss$ (resp.
$\fs'_\ss(F)$) are such that $x\leftrightarrow y$ and
$\pi(x)=\pi(y)=z$.
\end{lem}
\begin{proof}
This result partially follows from Proposition \ref{prop. descent of
orbital integral} and the fact that for $f'\in\CC_c^\infty(S')$
(resp. $\CC_c^\infty(\fs'(F))$) and $f\in\CC_c^\infty(S)$ (resp.
$\CC_c^\infty(\fs(F))$) the functions $\pi_*(f')$ and $\pi_*(f)$ are
locally constant on $\CQ(F)_\rs$ (resp. $\fq(F)_\rs$). It remains to
prove Lemma \ref{lem. compatibility of transfer factor} ahead, which
shows the compatibility of the transfer factors under the semisimple
descent.
\end{proof}

\begin{lem}\label{lem. compare descendants}
\begin{enumerate}
\item Given semisimple $A\in\fg\fl_m(F)$ such that $A^2-{\bf1}_m
=\gamma B\bar{B}$, $AB=BA$ with $B\in\GL_m(E)$, the smooth transfer
exists for the sliced representations
$$\left(\GL_m(F)_A,\fg\fl_m(F)_A\right)\ \mathrm{and}\
\left(\GL_m(E)_A\cap\GL_{\sigma,m}(E)_B,
\fg\fl_m(E)_A\cap\fg\fl_m^\sigma(E)_B\right).$$
\item Given semisimple $A\in\GL_m(F)$ such that $A=\gamma B\bar{B}$
with $B\in\GL_m(E)$, the smooth transfer exists for the sliced
representations
$$\left(\GL_m(F)_A,\fg\fl_m(F)_A\right)\ \mathrm{and}\
\left(\GL_{\sigma,m}(E)_B,\fg\fl_m^\sigma(E)_B\right)$$

\end{enumerate}
\end{lem}
\begin{proof}
Firstly, we prove the second assertion. We can assume that
$\gamma=1$ and $A$ is of the form $\diag(A_1,A_2,...,A_k)$ such that
$$\GL_m(F)_A=\prod_{i=1}^k\GL_{m_i}(F_i),$$
where $F_i=F[A_i]$ is a field and $A_i$ is in the center of
$\GL_{m_i}(F_i)$. For each $1\leq i\leq k$, let $L_i=E\otimes_FF_i$.
Since $A\in\N(\GL_m(E))$, there exists $B_i\in\GL_{m_i}(L_i)$ such
that $A_i=\N(B_i)$ for each $i$. We can choose $B$ to be
$\diag(B_1,B_2,...,B_k)$. Then $\GL_{\sigma,m_i}(L_i)_{B_i}$ is an
inner form of $\GL_{m_i}(F_i)$, and
$\GL_{\sigma,m}(E)_B=\prod_{i=1}^k\GL_{\sigma,m_i}(L_i)_{B_i}$. For
$X\in\fg\fl_m^\sigma(E)_B$, it is easy to see that
$X\bar{B}\in\fg\fl_{\sigma,m}(E)_B$, where
$$\fg\fl_{\sigma,m}(E)_B=\left\{Y\in\fg\fl_m(E):YB=B\bar{Y}\right\},$$
which is the Lie algebra of $\GL_{\sigma,m}(E)_B$. For
$X\in\fg\fl_m^\sigma(E)_B$ and $h\in\GL_{\sigma,m}(E)_B$, we have
$h^{-1}X\bar{h}\bar{B}=h^{-1}X\bar{B}h$. Therefore, timing $\bar{B}$
on right, we get a $\GL_{\sigma,m}(E)_B$-equivariant isomorphism
$$\fg\fl_m^\sigma(E)_B\lra \fg\fl_{\sigma,m}(E)_B,$$ where
$\GL_{\sigma,m}(E)_B$ acts on $\fg\fl_{\sigma,m}(E)_B$ by
conjugation. Since the existence of smooth transfer between
$\GL_m(F_i)$ and its inner forms is known, we completes the proof.

The first assertion is proved in the same way. By the above
discussion, we know that the smooth transfer holds for
$$\left(\GL_m(F)_{A^2-{\bf1}_m},\fg\fl_m(F)_{A^2-{\bf1}_m}\right)
\ \mathrm{and}\
\left(\GL_{\sigma,m}(E)_B,\fg\fl_m^\sigma(E)_B\right).$$ We can
choose some scalar $\lambda\in F$ so that $A+\lambda\in\GL_m(F)$.
Then $A+\lambda\in\GL_m(F)_{A^2-{\bf1}_m}$ and
$A+\lambda\in\GL_{\sigma,m}(E)_B$. Hence
$$\left(\GL_{\sigma,m}(E)_B\right)_{A+\lambda}
=\GL_m(E)_A\cap\GL_{\sigma,m}(E)_B$$ is an inner form of
$$\left(\GL_m(F)_{A^2-{\bf 1}_m}\right)_{A+\lambda}=\GL_m(F)_A.$$
The rest of the proof is the same as that of the first assertion.
\end{proof}

\begin{prop}\label{prop. local transfer around 0}
To prove the existence of local transfer around all elements of
$\CQ(F)$ or $\fq(F)$, it suffices to prove the existence of local
transfer around zero of $\fq(F)$.
\end{prop}
\begin{proof}
By Lemma \ref{lem. reduction to sliced}, it suffices to prove the
existence of smooth transfer for the sliced representations
$(H_x,\fs_x)$ and $(H'_y,\fs'_y)$ where $x\leftrightarrow y$. By
Remark \ref{rem. matching descent} and Lemma \ref{lem. compare
descendants}, it suffices to prove the existence of smooth transfer
for $(H_m,\fs_m)$ and $(H'_m,\fs'_m)$, that is, the existence of
local transfer around zero of $\fq(F)$.
\end{proof}

\begin{cor}\label{cor. lie algebra implies symmetric space}
Theorem \ref{thm. partial smooth transfer 3} implies Theorem
\ref{thm. partial smooth transfer 2}.
\end{cor}

\paragraph{Explicit analytic Luna slices}
We now describe explicit analytic Luna slices at semisimple elements
of $S$ or $\fs(F)$. We refer the reader to \cite[page 76]{jr} for
the discussions on $\fs$, and to \cite[\S5.2]{jr} for the
discussions on $S$.

First let $X\in\fs(F)$ be semisimple. Write
$\fs(F)=\fs_X\oplus\fs_X^\bot$, where $\fs_X^\bot$ is the orthogonal
complement of $\fs_X$ in $\fs(F)$ with respect to $\pair{\ ,\ }$.
Set
$$Z=\left\{\ \xi\in\fs_X:\
\det\left([\ad(X+\xi)^2]|_{\fs_X^\bot}\right)\neq0\right\},$$ which
is a non-empty open set of $\fs_X$ and invariant under $H_X$. Let
$Z_X=\{X+\xi:\ \xi\in Z\}$. Consider the map $$\phi:H\times
Z_X\lra\fs(F),\quad (h,X+\xi)\mapsto \Ad h(X+\xi),$$ which is
everywhere submersive. Let $U_X$ be the image of $\phi$, which is an
open $H$-invariant set in $\fs(F)$. Then $Z_X$ and $U_X$ are what we
want, and $\psi$ is the natural map:
$$\psi:Z_X\lra\fs_X,\quad X+\xi\mapsto\xi.$$

Next let $x\in S$ be semisimple. Write $x=s(g)$ for some $g\in G$,
where $s$ is the symmetrization map. Consider the map $$\phi:H\times
G_x\times H\lra G,\quad (h,\xi,h')\mapsto h\xi gh'.$$ Let $\CZ'$ be
the set of $\xi$ such that $\phi$ is submersive at $(1,\xi,1)$,
which is also the set of $\xi$ in $G_x$ such that
$$\det\left([{\bf1}-\Ad s(\xi g)]|_{\fg_x^\bot}\right)\neq0.$$ Let
$$W'=\{X\in\fs_x:\ \det({\bf1}+X)\det({\bf1}-X)\neq0\},$$
which is an open neighborhood of $0$ in $\fs_x$. Consider the Cayley
transform
$$\lambda: W'\lra G_x,\quad X\mapsto({\bf1}-X)({\bf1}+X)^{-1},$$
and denote by $\CV$ the image of $W'$ under $\lambda$. Put
$\CZ=\CZ'\cap\CV$ and $W=\lambda^{-1}(\CZ)$. Let $U_x$ be the image
of $\phi(H\times\CZ\times H)$ under the symmetrization map $s$, and
$Z_x$ the image of $\phi({\bf1}\times\CZ\times{\bf1})$ under $s$.
Then $Z_x$ and $U_x$ are what we want.

The lemma below follows from the above construction and a direct
computation by choosing $x=x(A,n_1,n_2)$ and $X=X(A)$ in a standard
form. We omit the proof here.

\begin{lem}\label{lem. compatibility of transfer factor}
Let $x\in S$ (resp. $X\in\fs(F)$) be semisimple. Then we may choose
an $H_x$-invariant (resp. $H_X$-invariant) neighborhood of $x$
(resp. $X$) such that for any regular semisimple $y$ in this
neighborhood, $\kappa(y)$ is equal to a non-zero constant times
$\kappa(\psi(y))$.
\end{lem}

\subsection{Proof of Proposition \ref{prop. conjecture 2 implies
conjecture 1}} Now we can prove Proposition \ref{prop. conjecture 2
implies conjecture 1} with the help of the following results.

\begin{thm}\label{thm. support of distribution}
Denote by $\CN$ the null-cone of $\fs(F)$, by $\CN'$ the null-cone
of $\fs'(F)$.
\begin{enumerate}
\item Let $T\in\CD(\fs(F))^{H,\eta}$ be
such that $\Supp(T)\subset\CN$ and $\Supp(\wh{T})\subset\CN$. Then
$T=0$.
\item Let $T\in\CD(\fs'(F))^{H'}$ be
such that $\Supp(T)\subset\CN'$ and $\Supp(\wh{T})\subset\CN'$. Then
$T=0$.
\end{enumerate}
\end{thm}

\begin{proof}
The first assertion is proved in \cite[Proposition 3.1]{jr} when
$\eta$ is the trivial character. The same proof goes through for the
quadratic character $\eta$. The same proof is also valid for the
second assertion, noting the relation $m'<n^2$ in Proposition
\ref{prop. inequality for nilpotent 2}.
\end{proof}

The following corollary is a direct consequence of the above theorem
(cf. \cite[Corollary 4.20]{zh}).

\begin{cor}\label{cor. factorization of functions}
\begin{enumerate}
\item Let $\CC_0=\bigcap\limits_T\ker(T)$ where $T$ runs over all
$(H,\eta)$-invariant distributions on $\fs(F)$. Then each
$f\in\CC_c^\infty(\fs(F))$ can be written as
$$f=f_0+f_1+\wh{f_2},$$ with $f_0\in\CC_0$ and
$f_i\in\CC_c^\infty(\fs(F)-\CN),\ i=1,2$.
\item Let $\CC_0=\bigcap\limits_T\ker(T)$ where $T$ runs over all
$H'$-invariant distributions on $\fs'(F)$. Then each
$f\in\CC_c^\infty(\fs'(F))$ can be written as
$$f=f_0+f_1+\wh{f_2},$$ with $f_0\in\CC_0$ and
$f_i\in\CC_c^\infty(\fs'(F)-\CN'),\ i=1,2$.
\end{enumerate}
\end{cor}

\begin{proof}[Proof of Proposition \ref{prop. conjecture 2 implies
conjecture 1}] Now we assume that Theorem \ref{thm. fourier preserve
transfer} is true. First we consider the converse direction: given
$f\in\CC_c^\infty(\fs(F))_0$, we want to show its smooth transfer
exists in $\CC_c^\infty(\fs'(F))$. For a general element $f$ in
$\CC_c^\infty(\fs(F))$, we say that $f'\in\CC_c^\infty(\fs'(F))$ is
a smooth transfer of $f$ if
$$O(y,f')=\kappa(x)O^\eta(x,f),\quad x\leftrightarrow y,$$ for each
$y\in\fs'_\rs(F)$. We can and do assume that: there exists a nonzero
$c\in\BC$ such that if $f'$ is a smooth transfer of
$f\in\CC_c^\infty(\fs(F))$ then $c\wh{f}'$ is a smooth transfer of
$\wh{f}$. This assumption is proved in Theorem \ref{thm. half
fourier 2}. Basing on this assumption, we will show the following
stronger form of Theorem \ref{thm. partial smooth transfer 3}: for
each $f\in\CC_c^\infty(\fs(F))$, there exists
$f'\in\CC_c^\infty(\fs'(F))$ that is a smooth transfer of $f$. We
use induction argument to show this result. Suppose that the
stronger form of Theorem \ref{thm. partial smooth transfer 3} holds
for $\CC_c^\infty(\fs_m(F))$ and $\CC_c^\infty(\fs'_m(F))$ for every
$m<n$. Thus, by Corollary \ref{cor. local transfer} and Lemma
\ref{lem. compare descendants}, for each
$f\in\CC_c^\infty(\fs(F)-\CN)$, its smooth transfer exists.
Therefore, by Corollary \ref{cor. factorization of functions}, it
suffices to show the existence of smooth transfer for $\wh{f}$ with
$f\in\CC_c^\infty(\fs(F)-\CN)$, which is guaranteed by the
assumption.

For the other direction in Theorem \ref{thm. partial smooth transfer
3} the proof is the same.
\end{proof}

\section{Representability}\label{sec: representability}
For $X\in\fs_\rs(F)$, it is more convenient to consider the
normalized orbital integral
$$I^\eta(X,f):=\abs{D^\fs(X)}_F^{\frac{1}{2}}O^\eta(X,f),\quad
f\in\CC_c^\infty(\fs(F)).$$ Similarly, for $Y\in\fs'_\rs(F)$, we
consider the normalized orbital integral
$$I(Y,f'):=\abs{D^{\fs'}(Y)}_F^{\frac{1}{2}}O(Y,f'),\quad
f'\in\CC_c^\infty(\fs'(F)).$$ If $X\leftrightarrow Y$ it is not hard
to see that $\abs{D^\fs(X)}_F=\abs{D^{\fs'}(Y)}_F$. Hence it does
not matter if we consider the smooth transfer with respect to the
normalized orbital integrals instead of the orbital integrals
introduced before. The Fourier transform of the normalized orbital
integral $I^\eta_X$ is defined to be
$$\wh{I}^\eta(X,f)=I^\eta(X,\wh{f}).$$ For $Y\in\fs'_\rs(F)$,
we define $\wh{I}_Y$ similarly.

To prove Theorem \ref{thm. fourier preserve transfer}, we first need
to study the Fourier transform of orbital integrals. In this
section, we prove the following fundamental theorem on the
representability of $\wh{I}^\eta_X$ and $\wh{I}_Y$.

\begin{thm}\label{thm. representability}
\begin{enumerate}
\item For each $X\in\fs_\rs(F)$, there exists a locally constant
$H$-invariant function $\wh{i}^\eta_X$ defined on $\fs_\rs(F)$ which
is locally integrable on $\fs(F)$, such that for any
$f\in\CC_c^\infty(\fs(F))$ we have
$$\wh{I}^\eta(X,f)=\int_{\fs(F)}\wh{i}^\eta_X(Y)\kappa(Y)f(Y)
|D^\fs(Y)|_F^{-1/2}\ \d Y.$$
\item For each $X\in\fs'_\rs(F)$, there exits a locally constant
$H'$-invariant function $\wh{i}_X$ defined on $\fs'_\rs(F)$ which is
locally integrable on $\fs'(F)$, such that for any
$f\in\CC_c^\infty(\fs'(F))$ we have
$$\wh{I}(X,f)=\int_{\fs'(F)}\wh{i}_X(Y)f(Y)
|D^{\fs'}(Y)|_F^{-1/2}\ \d Y.$$
\end{enumerate}
\end{thm}
We also write $\wh{i}^\eta(X,Y)$ (resp. $\wh{i}(X,Y)$) instead of
$\wh{i}^\eta_X(Y)$ (resp. $\wh{i}_X(Y)$), which is viewed as a
function on $\fs_\rs(F)\times\fs_\rs(F)$ (resp.
$\fs'_\rs(F)\times\fs'_\rs(F)$). Then it is not hard to see that
$\wh{i}^\eta(X,Y)$ (resp. $\wh{i}(X,Y)$) is locally constant on
$\fs_\rs(F)\times\fs_\rs(F)$ (resp. $\fs'_\rs(F)\times\fs'_\rs(F)$),
$(H,\eta)$-invariant (resp. $H'$-invariant) on the first variable
and $H$-invariant (resp. $H'$-invariant) on the second variable. Our
method to prove Theorem \ref{thm. representability} follows that of
of \cite{hc1} and \cite{hc}. Some of our treatment also follows that
of \cite{ko}. We only prove the assertion for $\wh{I}^\eta_X$. The
assertion for $\wh{I}_X$ can be proved in the same way and is left
to the reader.

\subsection{Reduction to elliptic case}\label{sec. parabolic induction}
In this subsection we reduce the question of the representability of
$\wh{I}^\eta_X$ to that for elliptic elements $X\in\fs_\rs(F)$. For
$X\in\fs_\rs(F)$, we say that $X$ is elliptic if its stabilizer
$H_X$ is an elliptic torus.
Thus, if $X=\begin{pmatrix}0&A\\
B&0\end{pmatrix}$, $X$ is elliptic if and only $AB$ is elliptic in
$\GL_n(F)$ in the usual sense.

For convenience, we suppose that $X\in\fs_\rs(F)$ is of the form
$\begin{pmatrix}0&{\bf1}_n\\A&0\end{pmatrix}$. From now on, we also
suppose that $X$ is not elliptic, or equivalently, $A$ is not
elliptic. Then there exists a proper Levi subgroup $\FM_0$ of
$\GL_n$ such that $A\in\FM_0$. Let $\FP_0$ be a proper parabolic
subgroup of $\GL_n$ such that $\FM_0$ is a Levi component of
$\FP_0$. Let $\FU_0$ be the unipotent subgroup of $\FP_0$. Set
$\fm_0=\Lie(\FM_0)$, $\fp_0=\Lie(\FP_0)$ and $\fu_0=\Lie(\FU_0)$.
Then $\fp_0=\fm_0\oplus\fu_0$, and $\gl_n=\fp_0\oplus\bar{\fu}_0$
where $\bar{\fu}_0$ is the Lie algebra of the unipotent subgroup
$\bar{\FU}_0$ opposite to $\FU_0$.

Write $\fs=\fs^+\oplus\fs^-$, where
$$\fs^+=\left\{\begin{pmatrix}0&B\\0&0\end{pmatrix}:B\in\gl_n\right\},
\quad
\fs^-=\left\{\begin{pmatrix}0&0\\C&0\end{pmatrix}:C\in\gl_n\right\}.$$
Identify $\fs^+$ (resp. $\fs^-$) with $\gl_n$. Under this
identification, let $\fr^+\subset\fs^+$ (resp. $\fr^-\subset\fs^-$)
be the subspace that corresponds to $\fm_0$, $\fn^+\subset\fs^+$
(resp. $\fn^-\subset\fs^-$)  the subspace that corresponds to
$\fu_0$, $\bar{\fn}^+\subset\fs^+$ (resp. $\bar{\fn}^-\subset\fs^-$)
the subspace that corresponds to $\bar{\fu}_0$. Set
$\fr=\fr^+\oplus\fr^-$, $\fn=\fn^+\oplus\fn^-$ and
$\bar{\fn}=\bar{\fn}^+\oplus\bar{\fn}^-$. Then
$\fs=\fr\oplus\fn\oplus\bar{\fn}$ and $X\in\fr(F)$. Notice that
$\fr$ is isomorphic to a product of $\fs_{n_i}$ with $\sum n_i=n$.
Also notice that $\fn^\bot=\fr\oplus\fn$ and
$(\fr\oplus\fn)^\bot=\fn$ under the fixed pairing
$\pair{\cdot,\cdot}$ on $\fs$.

We call a subspace $\ff$ of $\fs$ a proper Levi subspace if $\ff$ is
of the form $\fr$ as above for some $\fr$.

Let $\FP=\FP_0\times \FP_0$, which is a parabolic subgroup of
$\FH=\GL_n\times\GL_n$. There is a Levi decomposition $\FP=\FM\FU$
and $\fp=\fm\oplus\fu$, with $\FM=\FM_0\times \FM_0$,
$\FU=\FU_0\times \FU_0$, $\fm=\fm_0\oplus\fm_0$ and
$\fu=\fu_0\oplus\fu_0$. Notice that
$(\FM,\fr)\simeq\prod(\FH_{n_i},\fs_{n_i})$ for some
$(\FH_{n_i},\fs_{n_i})$. We fix an open compact subgroup $K$ of $H$
such that $H=MUK$ and $\eta|_K$ is trivial. Recall that we write
$M=\FM(F)$ and $U=\FU(F)$. Here we choose the Haar measure on $H$ so
that $\vol(K)=1$, and choose Haar measures on $M$ and $U$ so that
for any $f\in\CC_c^\infty(H)$, $$\int_H f(h)\ \d
h=\int_M\int_U\int_K f(muk)\ \d m\ \d u\ \d k.$$ We choose the Haar
measure on Lie algebra $\fu(F)$ compatible with that on $U$ under
the exponential map, and choose Haar measures on
$\fr(F),\fn(F),\bar{\fn}(F)$ according to the above identifications.

For $f\in\CC_c^\infty(\fs(F))$, we define
$f^\fr\in\CC_c^\infty(\fr(F))$ to be
$$f^\fr(Y):=\int_{\fn(F)}f(Y+Z)\ \d Z,$$ define
$\wt{f}\in\CC_c^\infty(\fs(F))$ to be $$\wt{f}(Y)=\int_K f(Y^k)\ \d
k,$$ and define $f^{(\fr)}\in\CC_c^\infty(\fr(F))$ to be
$$f^{(\fr)}:=\left(\wt{f}\right)^{\fr}.$$

\begin{defn} Let $T_\fr$ be a distribution on $\fr(F)$.
We define the distribution $i^\fs_\fr(T_\fr)$ on $\fs(F)$ to be:
$$i^\fs_\fr(T_\fr)(f):=T_\fr(f^{(\fr)}),\quad\textrm{for }
f\in\CC_c^\infty(\fs(F)).$$
\end{defn}

The above process is an analogue of parabolic induction in the usual
sense (cf. \cite[\S1]{hc} or \cite[\S13]{ko}), and has the following
similar properties. Notice that $\FM$ acts on $\fr$ by the adjoint
action, which is induced from the action of $\FH$ on $\fs$. Denote
by $\fr_\rs$ the regular semisimple locus of $\fr$ with respect to
the action of $\FM$. If $Y$ is in $\fs_\rs(F)$, then it is also in
$\fr_\rs(F)$. If $Y\in\fr_\rs(F)$, put
$$|D^\fr(Y)|_F=|\det(\ad(Y);\fm/\ft\oplus\fr/\fc)|^{\frac{1}{2}},$$
where $\fc$ is the Cartan space of $\fr$ containing $Y$ and $\ft$ is
the Lie algebra of the centralizer of $Y$ in $\FM$. The normalized
orbital integral $I_X^{\eta,M}(f')$, for
$f'\in\CC_c^\infty(\fr(F))$, is defined to be
$$|D^\fr(X)|^{\frac{1}{2}}_F\int_{H_X\bs M}f'(X^m)\eta(m)\ \d m.$$
Then $I_X^{\eta,M}$ is a distribution on $\fr(F)$. In the
proposition below, we write $I_X^{\eta,H}$ instead of $I_X^\eta$ to
distinguish it from $I_X^{\eta,M}$.

\begin{prop}\label{prop. parabolic induction}
\begin{enumerate}
\item Suppose that $T_\fr$ is an $(M,\eta)$-invariant distribution
on $\fr(F)$, then $i_\fr^\fs(T_\fr)$ is an $(H,\eta)$-invariant
distribution on $\fs(F)$.
\item We have $i_\fr^\fs(I^{\eta,M}_X)
=I^{\eta,H}_X.$
\item Suppose that $T_\fr$ is an $(M,\eta)$-invariant distribution on
$\fr(F)$, which is represented by a function $\Theta_\fr$ which is
locally constant on $\fr_\rs(F)$ and locally integrable on $\fr(F)$.
In other words, for any $f\in\CC_c^\infty(\fr(F))$,
$$T_\fr(f)=\int_{\fr_\rs(F)}\Theta_\fr(Y)\kappa(Y)f(Y)|D^\fr(Y)|_F^{-1/2}\
\d Y.$$ Then the distribution $i_\fr^\fs(T_\fr)$ is represented by
the function $$\Theta_\fs(Y)=\sum_{Y'}\Theta_\fr(Y'),$$ where $Y'$
runs over a finite set of representatives for the $M$-conjugacy
classes of elements in $\fr(F)$ which are $H$-conjugate to $Y$. The
function $\Theta_\fs$ is locally constant on $\fs_\rs(F)$ and
locally integrable on $\fs(F)$, and, for any
$f\in\CC_c^\infty(\fs(F))$,
$$i_\fr^\fs(T_\fr)(f)=\int_{\fs_\rs(F)}\Theta_\fs(Y)
\kappa(Y)f(Y)|D^\fs(Y)|_F^{-1/2}\ \d Y.$$
\item The map $f\mapsto f^{(\fr)}$ commutes with the Fourier
transform, and therefore
$i_\fr^\fs(\wh{T}_\fr)=\wh{i^\fs_\fr(T_\fr)}$.
\end{enumerate}
\end{prop}

\begin{proof}
(1) For $f\in\CC_c^\infty(\fs(F))$ and $h\in H$, define
$^hf\in\CC_c^\infty(\fs(F))$ by $^hf(Y)=f(Y^h)$. To prove (i), it
suffices to observe the following relation: for $p=mu\in P$ and
$Y\in\fr(F)$, we have
$$\begin{aligned}
(^pf)^\fr(Y)&=\int_{\fn(F)} {^pf}(Y+Z)\ \d Z
=\int_{\fn(F)} f(Y^p+Z^p)\ \d Z\\
&=\int_{\fn(F)} f(Y^m+Z^p)\ \d Z\\
&=|\det(\Ad(p);\fn)|_F\int_{\fn(F)} f(Y^m+Z)\ \d Z.
\end{aligned}$$
It is easy to verify that $$|\det(\Ad(p);\fn)|_F
=|\det(\Ad(p);\fu)|_F=\delta_P(p),$$ where $\delta_P$ is the modulus
character of $P$. Therefore $(^pf)^\fr(Y)=\delta_P(p)f^\fr(Y^m)$.
The rest arguments are routine.

(2). Write $T=H_X$ for simplicity. For $f\in\CC_c^\infty(\fs(F))$,
$$\begin{aligned}
\int_{T\bs H}f(X^h)\eta(h)\ \d h&=\int_{T\bs
M}\int_U\int_Kf(X^{muk})
\eta(m)\ \d k\ \d u\ \d m\\
&=\int_{T\bs M}\int_U\wt{f}(X^{mu})\eta(m)\ \d u\ \d m.
\end{aligned}$$
Write $Y=X^m$. Notice that the map
$$\alpha: \FU\lra \fn,\quad u\mapsto u^{-1}Yu-Y$$
is an isomorphism of algebraic varieties, whose Jacobian is
$$|\det(\varrho\circ\ad(Y);\fu)|_F.$$ Also note that
$$\begin{aligned}
|\det(\varrho\circ\ad(Y);\fu)|_F
&=|\det(\varrho\circ\ad(Y);\fu\oplus\bar{\fu})|_F^{1/2}\\
&=\frac{|\det(\varrho\circ\ad(Y);\fh/\ft)|_F^{1/2}}
{|\det(\varrho\circ\ad(Y);\fm/\ft)|_F^{1/2}}\\
&=\frac{|D^\fs(Y)|_F^{1/2}}{|D^\fr(Y)|_F^{1/2}}.
\end{aligned}$$
Therefore $$\begin{aligned}I^{\eta,H}(X,f)&=\int_{T\bs M}
\int_{\fn(F)}|D^\fr(X^m)|_F
^{1/2}\wt{f}(X^m+Z)\eta(m)\ \d Z\ \d m\\
&=|D^\fr(X)|_F^{1/2}\int_{T\bs M}f^{(\fr)}(X)\eta(m)\ \d m\\
&=I^{\eta,M}(X,f^{(\fr)}).\end{aligned}$$

The assertion (iii) is a consequence of Weyl integration formula,
and the assertion (iv) is obvious.
\end{proof}

Let $\fs_\el$ be the open subset of elliptic regular semisimple
elements in $\fs(F)$.

\begin{lem}\label{lem. elliptic cusp form}
Suppose that $\phi\in\CC_c^\infty(\fs_\el)$. Then $\phi^{(\fr)}$ and
$(\wh{\phi})^{(\fr)}$ are identically zero for every proper Levi
subspace $\fr$ of $\fs$. Moreover, for any regular semisimple
element $X$ of $\fs(F)$ lying in $\fr(F)$, we have
$$I^\eta(X,\phi)=\wh{I}^\eta(X,\phi)=0.$$
\end{lem}
\begin{proof}
The vanishing of $\phi^{(\fr)}$ is obvious. The vanishing of
$(\wh{\phi})^{(\fr)}$ is a consequence of Proposition \ref{prop.
parabolic induction} (iv). The vanishing of the orbital integrals
follows from the first assertion and Proposition \ref{prop.
parabolic induction} (ii).
\end{proof}

Let $\fc$ be an elliptic Cartan subspace of $\fs$, which means that
one (any) element of $\fc_\reg(F)$ is elliptic.  Let $T$ be the
centralizer of $\fc$ in $H$, and $Z$ the center of $G$ which is also
contained in $T$. Then $Z\bs T$ is compact since $\fc$ is elliptic.
Now we require that $\vol(Z\bs T)=1$ here, which does not matter.

Let $\fs_\rs^\fc=(\fc_\reg(F))^H$ and
$\phi\in\CC_c^\infty(\fs_\rs^\fc)$. We define the distribution
$I_\phi\in\CD(\fs(F))^{H,\eta}$ to be
$$I_\phi(f)=\int_{Z\bs H}\int_{\fs(F)} f(Y)\phi(Y^h)\eta(h)\ \d Y\ \d h.$$
This distribution is well defined:
$$\begin{aligned}&\int_{Z\bs H}\int_{\fs(F)} |f(Y)\phi(Y^h)|
\ \d Y\ \d h\\
=&\int_{Z\bs H}\d h\left(\int_{\fc(F)}|D^\fs(Y)|_F\ \d Y\int_{Z\bs
H}|f(Y^{h'})|\cdot|\phi(Y^{h'h})|\ \d h'\right)\\
=&\int_{\fc(F)}|D^\fs(Y)|_F\ \d Y\left(\int_{(Z\bs H)\times (Z\bs
H)}|f(Y^{h'})|\cdot|\phi(Y^h)|\ \d h'\ \d h\right)\\
=&\int_{\fc_\reg(F)}I\left(Y,|f|\right)\cdot I(Y,|\phi|)\ \d
Y<\infty,
\end{aligned}$$
since $I(Y,|\phi|)\in\CC_c^\infty(\fc_\reg(F))$. Here $I(\cdot,f)$
is the normalized orbital integral without twisting $\eta$. We also
define the distribution $I_{\wh{\phi}}\in\CD(\fs(F))^{H,\eta}$ to be
$$I_{\wh{\phi}}(f)=\int_{Z\bs H}\int_{\fs(F)} f(Y)
\wh{\phi}(Y^h)\eta(h)\ \d Y\ \d h.$$ We have the relation
$$\int_{\fs(F)} f(Y)\wh{\phi}(Y^h)\ \d Y=\int_{\fs(F)}\wh{f}(Y)\phi(Y^h)\ \d Y.$$
Thus $$\begin{aligned}&\int_{Z\bs H}\eta(h)\ \d h\left(\int_{\fs(F)}
f(Y)\wh{\phi}(Y^h)\ \d Y\right)\\
=&\int_{Z\bs H}\eta(h)\ \d h\left(\int_{\fs(F)}\wh{f}(Y)\phi(Y^h)\
\d Y\right),\end{aligned}$$ by the absolute convergence of the
latter one, which shows that $I_{\wh{\phi}}$ is well defined and
$I_{\wh{\phi}}=\wh{I}_\phi$. In summary, we have the following
lemma.

\begin{lem}\label{lem. Iphi}
Let $\fc$ be an elliptic Cartan subspace of $\fs$ and
$\phi\in\CC_c^\infty(\fs_\rs^\fc)$. Then $\wh{I}_\phi=I_{\wh{\phi}}$
\end{lem}

In the next subsection, we will reduce Theorem \ref{thm. Iphi} to
the following theorem whose proof will be given in \S\S\ref{subsec.
majorize of orb}--\ref{sec. proof of thm Iphi}.

\begin{thm}\label{thm. Iphi}
Let $\fc$ be an elliptic Cartan subspace of $\fs$ and
$\phi\in\CC_c^\infty(\fs_\rs^\fc)$. Then $\wh{I}_\phi$ is
represented by a locally integrable function on $\fs(F)$ which is
locally constant on $\fs_\rs(F)$.
\end{thm}

\subsection{Proof of Theorem \ref{thm. representability}}
To show the representability of the Fourier transform of orbital
integrals, we need the following relative version of Howe's
finiteness theorem (Theorem \ref{thm Howe's finiteness theorem}).
Let us introduce some notation. If $\omega$ is a compact set in
$\fs(F)$, put
$$\CJ(\omega)^\eta=\{T\in\CD(\fs(F))^{H,\eta}:\ \Supp(T)\subset\cl
(\omega^H)\}.$$ Let $L\subset\fs(F)$ be a lattice (a compact open
$\CO_F$-submodule). Denote by $\CC_c(\fs(F)/L)$ the space of
$f\in\CC_c^\infty(\fs(F))$ which is invariant under translation by
$L$. Let $j_L:\CJ(\omega)^\eta\ra\CC_c(\fs(F)/L)^*$ be the
composition of the maps:
$$j_L:\CJ(\omega)^\eta\incl\CD(\fs(F))
\stackrel{\res}{\lra}\CC_c(\fs(F)/L)^*,$$ where $\CC_c(\fs(F)/L)^*$
is the vector space dual to $\CC_c(\fs(F)/L)$ and $\res$ is the
restriction map. Then Howe's finiteness theorem is the following.
\begin{thm}\label{thm Howe's finiteness theorem}
For any lattice $L$ and any compact set $\omega$ in $\fs(F)$, we
have $$\dim j_L(\CJ(\omega)^\eta)<+\infty.$$
\end{thm}
\begin{proof}
It was shown in \cite[Theorem 6.1]{rr} that Howe's finiteness
theorem holds in a more general setting when $\eta=\bf1$. It is not
hard to check that it still holds when $\eta$ is our quadratic
character.
\end{proof}

The following variant of Howe's theorem is often used, and we refer
the reader to \cite[\S26]{ko} for more details. Let
$\wh{j}_L:\CJ(\omega)^\eta\ra\CD(L)$ be the composition of the maps
$$\wh{j}_L:\CJ(\omega)^\eta\incl\CD(\fs(F))\stackrel{\CF}{\lra}\CD(\fs(F))
\stackrel{\res}{\lra}\CD(L),$$ where $\CF$ denotes the Fourier
transform.
\begin{thm}
For any lattice $L$ and any compact set $\omega$ in $\fs(F)$,
$$\dim\wh{j}_L(\CJ(\omega)^\eta)<+\infty.$$
\end{thm}
\begin{proof}
See \cite[Theorem 26.3]{ko}.
\end{proof}

\begin{cor}\label{cor. cor of Howe's finiteness}
Let $\omega$ be compact, and let $V$ be a subspace of
$\CJ(\omega)^\eta$. Let $L$ be any lattice in $\fs(F)$. Then
$\wh{j}_L(V)=\wh{j}_L(\cl(V))$.
\end{cor}
\begin{proof}
See \cite[Proposition 26.1]{ko}.
\end{proof}

\begin{proof}[Proof of Theorem \ref{thm. representability}]
By Proposition \ref{prop. parabolic induction}, it suffices to show
that $\wh{I}^\eta_X$ can be represented when $X$ lies in
$\fc_\reg(F)$ for some elliptic Cartan subspace $\fc$ of $\fs$. Then
Theorem \ref{thm. representability} follows from Theorem \ref{thm.
Iphi}, Lemma \ref{lem. restriction} and the fact that
$\fs(F)=\bigcup_{\textrm{lattice}}L$.
\end{proof}

\begin{lem}\label{lem. restriction}
Let $X\in\fc_\reg(F)$ be an elliptic element and $\omega$ a compact
open neighborhood of $X$ in $\fc_\reg(F)$. Then given a lattice $L$
in $\fs(F)$, there exists $\phi\in\CC_c^\infty(\omega^H)$ such that
$\wh{I}^\eta_X$ and $\wh{I}_\phi$ have the same restriction to $L$.
\end{lem}
\begin{proof}
The proof is similar as that of \cite[Lemma 26.5]{ko}. We first show
that $I^\eta_X$ lies in the closure of the linear space
$$I_\omega:=\{\ I_\phi:\ \phi\in\CC_c^\infty(\omega^H)\ \},$$ which
is a subspace of $\CJ(\omega)^\eta$. It suffices to show that: if
$I_\phi(f)=0$ for all $\phi\in\CC_c^\infty(\omega^H)$ then
$I_X^\eta(f)=0$. Note that
$$I_\phi(f)=\int_\omega I^\eta_Y(f)\cdot I_Y^\eta(\phi)\ \d Y.$$ We
may shrink $\omega$ so that every function
$\varphi\in\CC_c^\infty(\omega)$ arises as $Y\mapsto I^\eta_Y(\phi)$
for some $\phi\in\CC_c^\infty(\omega)$. Thus $I^\eta_X(f)=0$ if
$I_\phi(f)=0$ for all $\phi\in\CC_c^\infty(\omega^H)$. By Corollary
\ref{cor. cor of Howe's finiteness}, we see that
$\wh{j}_L(I^\eta_X)\in\wh{j}_L(I_\omega)$ for any lattice $L$. In
other words, given a lattice $L$, there exists a
$\phi\in\CC_c^\infty(\omega^H)$ such that $\wh{I}^\eta_X$ and
$\wh{I}_{\phi}$ have the same restriction to $L$.
\end{proof}

\subsection{Bounding the orbital integrals }
\label{subsec. majorize of orb} In this subsection, we will show the
boundness of the normalized orbital integrals along a Cartan
subspace (Theorem \ref{thm. upper bound}), which is crucial for
proving Theorem \ref{thm. Iphi}. We follow the same line as the
proof of \cite[Theorem 14]{hc1}, where there are no Shalika germs
involved.

\begin{thm}\label{thm. upper bound}\begin{enumerate}
\item Let $\fc$ be a Cartan subspace of $\fs$ and
$f\in\CC_c^\infty(\fs(F))$. Then
$$\sup_{X\in\fc_\reg(F)}|I^\eta(X,f)|<+\infty.$$
\item Let $\fc'$ be a Cartan subspace of $\fs'$ and
$f'\in\CC_c^\infty(\fs'(F))$. Then
$$\sup_{X\in\fc'_\reg(F)}|I(X,f')|<+\infty.$$
\end{enumerate}
\end{thm}

We will prove only the first assertion with respect to $\fs$. The
second assertion can be proved in the same way. We use inductive
method to prove this theorem. In the case $n=1$, our case
essentially is the Gan-Gross-Prasad conjecture for unitary groups of
rank 1. Thus Theorem \ref{thm. upper bound} follows from the
discussions in \cite[\S4.1]{zh} (in particular, Lemma 4.1 in loc.
cit.). Now we assume that Theorem \ref{thm. upper bound} holds for
$\CC_c^\infty(\fs_m(F))$ for every $m<n$.

\begin{lem}\label{lem. a compact lem1}
Fix a compact set $\omega$ of $\fs(F)$ and a Cartan subspace $\fc$.
Then the set of all $X\in\fc(F)$ such that $X\in\cl(\omega^H)$ is
relative compact in $\fc(F)$.
\end{lem}
\begin{proof}
It suffices to assume $\omega$ is closed. Consider the closed
inclusion $i:(\fc/W)(F)\ra(\fs/\FH)(F)$ where $W$ is the Weyl group
of $\fc$, and the natural map $\pi:\fs(F)\ra(\fs/\FH)(F)$. Then
$\pi(\omega)$ and thus $i^{-1}(\pi(\omega))$ is compact. The lemma
follows from the fact that the map $\fc(F)\ra(\fc/W)(F)$ is a proper
map between locally compact Hausdorff spaces.
\end{proof}

\begin{cor}\label{cor. a compact cor}
For $f\in\CC_c^\infty(\fs(F))$, $I^\eta(X,f)=0$ for
$X\in\fc_\reg(F)$ lying outside a compact subset of $\fc(F)$.
\end{cor}

We first prove Theorem \ref{thm. upper bound} in the following
situation.
\begin{lem}\label{lem. lem 29}
Let $f$ be in $\CC_c^\infty(\fs(F)-\CN)$. Then $I^\eta(\cdot,f)$ is
bounded on $\fc_\reg(F)$.
\end{lem}
\begin{proof}
By Lemma \ref{lem. a compact lem1} and Corollary \ref{cor. a compact
cor}, it suffices to prove: given $X_0\in\fc(F)$, we can choose a
neighborhood $V$ of $X_0$ in $\fc(F)$ such that  $$\sup_{X\in
V'}|I^\eta(X,f)|<+\infty,\quad\textrm{where }V'=V\cap\fs_\rs(F).$$
When $X_0\neq0$, using the descent of orbital integrals (Proposition
\ref{prop. descent of orbital integral}), we reduce to considering
the orbital integrals for $\CC_c^\infty(\fs_{X_0})$ with respect to
the action of $H_{X_0}$. Since $X_0\neq0$, $(H_{X_0},\fs_{X_0})$ is
of the form
$$(\GL_m(F)_A,\fg\fl_m(F)_A)\times(H_{n-m},\fs_{n-m}(F))$$ for some
semisimple $A$ in $\GL_n(F)$ and some integer $0<m\leq n$. Then the
result follows from the inductive hypothesis on $n-m$ and the bound
of the usual orbital integrals for $\CC_c^\infty(\gl_m(F)_A)$ by
Harish-Chandra. When $X_0=0$, since $\Supp(f)\cap\CN=\emptyset$, we
can find a neighborhood $V$ of $X_0$ such that $I^\eta(X,f)=0$ on
$V'$.
\end{proof}

Now let $\fs_0$ be the set of $Y\in\fs(F)$ such that: there exists
an open neighborhood $\omega$ of $Y$ in $\fs(F)$ so that
$\sup\limits_{X\in\fc_\reg(F)}|I^\eta(X,f)|<\infty$ for all
$f\in\CC_c^\infty(\fs(F))$ with $\Supp(f)\subset\omega$. Since $\CN$
is closed in $\fs(F)$, Lemma \ref{lem. lem 29} implies that
$\fs(F)-\CN\subset\fs_0$. To prove Theorem \ref{thm. upper bound},
it remains to show that $\CN\subset\fs_0$. We need some preparation
below.

Fix $X_0\neq0$ in $\CN$. Let $(X_0,\d(X_0),Y_0)$ be an
$\fs\fl_2$-triple as in Lemma \ref{lem. sl2-triple}. Consider the
map
$$\psi:H\times\fs_{Y_0}\lra\fs(F),\quad (h,U)\mapsto(X_0+U)^h.$$
By the same discussion as that of \cite[Part VI, \S4]{hc1}, we see
that $\psi$ is everywhere submersive. Set
$\omega=\psi(H\times\fs_{Y_0})$, which is an open and $H$-invariant
subset of $\fs(F)$. Since $\psi$ is everywhere submersive, we have a
surjective linear map
$$\CC_c^\infty(H\times\fs_{Y_0})\lra\CC_c^\infty(\omega),
\quad\alpha\mapsto f_\alpha$$ such that
$$\int_\omega f_\alpha(X)p(X)\ \d X=\int_{H\times\fs_{Y_0}}\alpha(h,u)
p\left((X_0+U)^h\right)\ \d h\ \d U$$ for every locally integrable
function $p$ on $\omega$.

Let $\Gamma$ be the Cartan subgroup of $H$ with the Lie algebra
$F\cdot\d(X_0)$. Please refer to Lemma \ref{lem. eigen for nilp.}
and Proposition \ref{prop. inequality for nilpotent} for the
notations below. Put $t=\xi(\gamma)$ and write
$U_\gamma=\xi(\gamma)U^{\gamma^{-1}}$ for
$U\in\fs_{Y_0},\gamma\in\Gamma$. We have
$$(X_0+U_\gamma)^{\gamma h}=(X_0+tU^{\gamma^{-1}})^{\gamma h}
=t(X_0+U)^h.$$ For $\gamma\in\Gamma$ and
$\alpha\in\CC_c^\infty(H\times\fs_{Y_0})$, define
$\alpha'\in\CC_c^\infty(H\times\fs_{Y_0})$ to be
$$\alpha'(h,U)=\alpha(\gamma^{-1}h,U_{\gamma^{-1}}).$$

\begin{lem}\label{lem. 36}
Fix $\gamma\in\Gamma$ and $\alpha\in\CC_c^\infty(H\times\fs_{Y_0})$.
Then
$$f_\alpha(t^{-1}X)=|t|_F^{2n^2-r-m}f_{\alpha'}(X),\quad X\in\omega.$$
\end{lem}

\begin{proof}
Choose any function $\alpha$ in $\CC_c^\infty(\omega)$. We have
$$\begin{aligned}
&\int_{\fs(F)}f_\alpha(t^{-1}X)p(X)\ \d X\\
=&|t|_F^{2n^2}\int_{\fs(F)}f_\alpha(X)p(tX)\ \d X\\
=&|t|_F^{2n^2}\int_{H\times\fs_{Y_0}}\alpha(h,U)p\left(t(X_0+U)^h\right)\
\d h\ \d U\\
=&|t|_F^{2n^2}\int_{H\times\fs_{Y_0}}\alpha(h,U)p\left((X_0+U_\gamma)^
{\gamma h}\right)\ \d h\ \d U\\
=&|t|_F^{2n^2}\int_{H\times\fs_{Y_0}}\alpha(\gamma^{-1}h,U_{\gamma^{-1}})
p\left((X_0+U)^h\right)\left|\frac{\d U_{\gamma^{-1}}}{\d
U}\right|_F\
\d h\ \d U.\\
\end{aligned}$$
It remains to compute the Jacobian $\left|\frac{\d
U_{\gamma^{-1}}}{\d U}\right|_F$. Choose a basis $U_1,...,U_r$ of
$\fs_{Y_0}$ as in Lemma \ref{lem. eigen for nilp.}. Write
$U=\sum_{1\leq i\leq r}a_iU_i$. Then
$$\begin{aligned}U_{\gamma^{-1}}&=t^{-1}U^\gamma=t^{-1}\sum_ia_iU_i^\gamma\\
&=t^{-1}\sum_ia_i \xi_i(\gamma^{-1})U_i.\end{aligned}.$$ Hence
$$\left|\frac{\d U_{\gamma^{-1}}}{\d U}\right|_F
=|t|_F^{-r}\prod_{1\leq i\leq r} |t|_F^{\frac{-\lambda_i}{2}}
=|t|_F^{-r-m},$$ which implies the lemma.
\end{proof}

For $X\in\fc_\reg(F)$, there is a unique distribution $\tau^\eta_X$
on $\fs_{Y_0}$ such that
$I^\eta(X,f_\alpha)=\tau_X^\eta(\beta_\alpha)$ where
$$\beta_\alpha(U)=\int_H\alpha(h,U)\eta(h)\ \d h,
\quad \alpha\in\CC_c^\infty(H\times\fs_{Y_0}).$$ For
$f\in\CC_c^\infty(\omega)$, define $f'\in\CC_c^\infty(\omega)$ to be
$f'(X)=f(t^{-1}X)$. It is easy to see that
$$I^\eta(X,f')=|t|_F^{\frac{1}{2}(2n^2-n)}I^\eta(t^{-1}X,f).$$ Now fix
$\alpha\in\CC_c^\infty(H\times\fs_{Y_0})$, and set $f=f_\alpha$,
$f'=f'_\alpha$, $\beta=\beta_\alpha$ and $\beta'=\beta_{\alpha'}$.
Note that
$$f'=|t|_F^{2n^2-r-m}f_{\alpha'}.$$
We have
$$\beta'(U)=\int_H\alpha(\gamma^{-1}h,U_{\gamma^{-1}})\eta(h)\ \d h
=\eta(\gamma)\beta(t^{-1}U^\gamma),\quad U\in\fs_{Y_0}.$$ So we
obtain $$|t|_F^{\frac{1}{2}(2n^2-n)}I^\eta(t^{-1}X,f)=
|t|_F^{2n^2-r-m}I^\eta(X,f_{\alpha'})=|t|_F^{2n^2-r-m}\tau^\eta_X(\beta'),$$
or
\begin{equation}\label{equ. for iteration }I^\eta(t^{-1}X,f)
=|t|_F^{n^2+\frac{n}{2}-r-m}\left(I^\eta(X,f)
+\tau_X^\eta(\beta'-\beta)\right).\end{equation} By Proposition
\ref{prop. inequality for nilpotent}, we know
$n^2+\frac{n}{2}-r-m<0$.

Now we continue to prove Theorem \ref{thm. upper bound}. Let
$X_0\in\CN$ and suppose $X_0\neq 0$. We want to construct an open
neighborhood $\omega_0$ of $X_0$ such that $I^\eta(\cdot,f)$ is
bounded on $\fc_\reg(F)$ as soon as $\Supp(f)\subset\omega_0$.
Recall that we denote by $\CN_q$ the union of all $H$-orbits in
$\CN$ of dimension $\leq q$, and notice that $X_0\in\CN_{2n^2-r}$
and $\CN_{2n^2-n}=\CN$. So we can choose an open neighborhood
$\omega_1$ of $X_0$ in $\omega$ such that
$\omega_1\cap\CN_{2n^2-r}\subset X_0^H$, and can assume
$\omega_1=\omega_1^H$. By \cite[Lemma 37]{hc1}, we can choose an
open neighborhood $\CU$ of zero in $\fs_{Y_0}$ such that
$X_0+\CU\subset\omega_1$ and $\left(X_0+\CU\right)\cap
X_0^H=\{X_0\}$.

Fix $\gamma\in\Gamma$ such that $\eta(\gamma)=1$ and
$|t|_F=|\xi(\gamma)|_F>1$. Choose an open neighborhood $\CU_0$ of
zero in $\CU$ such that $t^{-1}\CU_0^\gamma\cup
t\CU_0^{\gamma^{-1}}\subset\CU$. Put $\CN^*=\CN-\{0\}$.

\begin{lem}\label{lem. 38}
$\CN^*\subset\fs_0$.
\end{lem}
\begin{proof}
We induct on $r=\dim\fs_{Y_0}$ for $X_0\in\CN^*$. Put
$\omega_0=(X_0+\CU_0)^H$, which is an open invariant neighborhood of
$X_0$. Consider the surjective map
$$H\times\CU_0\lra\omega_0,\quad (h,U)\mapsto(X_0+U)^h,$$ which is
everywhere submersive. Consider the surjective linear map
$$\CC_c^\infty(H\times\CU_0)\lra\CC_c^\infty(\omega_0),
\quad \alpha\mapsto f_\alpha,$$ which is the restriction of the map
$\CC_c^\infty(H\times\fs_{Y_0})\ra\CC_c^\infty(\omega)$ as before.
Let $f\in\CC_c^\infty(\omega_0)$ and choose
$\alpha\in\CC_c^\infty(H\times\CU_0)$ such that $f=f_\alpha$. Set
$\beta=\beta_\alpha,\beta'=\beta_{\alpha'}$ and $f'=f'_\alpha$ as
before. Then $\beta-\beta'\in\CC_c^\infty(\CU)$, and
$0\notin\Supp(\beta-\beta')$. Define
$\alpha_0(h,U)=\alpha_1(h)\left(\beta(U)-\beta'(U)\right)$ for $h\in
H,U\in\CU$, where $\alpha_1\in\CC_c^\infty(H)$ and
$\int_{H}\alpha_1(h)\eta(h)\ \d h=1$. For $X\in\fc_\reg(F)$, we have
$$I^\eta(X,f_{\alpha_0})=\tau^\eta_X(\beta_{\alpha_0})
=\tau^\eta_X(\beta-\beta'),$$ and $\Supp
(f_{\alpha_0})\cap\CN_{2n^2-r}=\emptyset$. Now we start the
induction on $r=\dim\fs_{Y_0}$.

First assume that $r=n$. Note that $r=n$ is the initial step.  In
such case we have $n^2+\frac{n}{2}-r-m=-\frac{n}{2}$ and
$$I^\eta(t^{-1}X,f)=|t|_F^{-\frac{n}{2}}\left(I^\eta(X,f)
+\tau^\eta_X(\beta'-\beta)\right),$$ by (\ref{equ. for iteration }).
Put $c=|t|_F^{-\frac{n}{2}}<1$. Since $\Supp
(f_{\alpha_0})\cap\CN=\emptyset$ ($\CN=\CN_{2n^2-n}$), by Lemma
\ref{lem. lem 29}, we have
$$a=\sup_{X\in\fc_\reg(F)}|\tau^\eta_X(\beta'-\beta)|<+\infty.$$
Iteration gives
$$I^\eta(t^{-d}X,f)=|t|_F^{-\frac{dn}{2}}I^\eta(X,f)+\sum_{1\leq k\leq d}
|t|_F^{-\frac{kd}{2}}\tau^\eta_{t^{k-d}X}(\beta'-\beta),\quad
(d\geq1),$$ or
$$I^\eta(X,f)=c^dI^\eta(t^dX,f)+\sum_{1\leq k\leq d}c^d
\tau^\eta_{t^kX}(\beta'-\beta),\quad (d\geq1).$$ Since
$\lim\limits_{d\ra+\infty}I^\eta(t^dX,f)=0$, we get
$$|I^\eta(X,f)|\leq a\sum_{1\leq k<\infty}c^k\leq a\frac{c}{1-c}.$$

Now assume $r>n$. Since
$\Supp(f_{\alpha_0})\cap\CN_{2n^2-r}=\emptyset$, by the inductive
hypothesis and Lemma \ref{lem. lem 29}, $I^\eta(X,f_{\alpha_0})$ is
bounded on $\fc_\reg(F)$ and so is $\tau^\eta_X(\beta-\beta')$.
Applying the same argument as the case $r=n$, we complete the proof
of the lemma.
\end{proof}

Applying the same arguments as those of \cite[Part VI \S7]{hc1}, we
have the following lemma.
\begin{lem}\label{lem. bound for 0}
$0\in\fs_0$.
\end{lem}

At last, Theorem \ref{thm. upper bound} follows from Lemma \ref{lem.
lem 29}, Lemma \ref{lem. 38} and Lemma \ref{lem. bound for 0}.

\subsection{Proof of Theorem \ref{thm. Iphi}}\label{sec. proof of thm
Iphi} Now we continue to prove Theorem \ref{thm. Iphi}. Let $\fc_0$
be an elliptic Cartan subspace of $\fs$ and
$\phi_0\in\CC_c^\infty(\fs_\rs^{\fc_0})$. For simplicity, we write
$\phi=\wh{\phi_0}$, and denote by $\Theta$ the distribution
$I_\phi$, that is, for $f\in\CC_c^\infty(\fs(F))$,
$$\Theta(f):=\int_{Z\bs H}\int_{\fs(F)}f(Y)\phi(Y^h)\eta(h)\ \d Y\ \d h.$$
Our goal is to prove that the distribution $\Theta$ can be
represented by a locally integrable function on $\fs(F)$ which is
locally constant on $\fs_\rs(F)$. We follow the strategy of the
proof of \cite[Theorem 16]{hc1}.

For $t\geq1$, let $\Omega_t$ denote the set of all $h\in H$ such
that $1+\log\|h\|_{Z\bs H}\leq t$. Then $\Omega_t$ is a compact set
modulo $Z$. Let $\Phi_t$ denote the characteristic function of
$\Omega_t$. Then we have
$$\begin{aligned}
\Theta(f)&=\lim_{t\ra+\infty}\int_{Z\bs
H}\Phi_t(h)\int_{\fs(F)}f(Y)\phi(Y^h)\eta(h)\ \d Y\ \d h\\
&=\lim_{t\ra+\infty}\int_{\fs(F)}f(Y)\Theta_t(Y)\ \d Y,
\end{aligned}$$
where $$\Theta_t(Y)=\int_{Z\bs H}\Phi_t(h)\phi(Y^h)\eta(h)\ \d h.$$
We will first show that $\lim\limits_{t\ra+\infty}\Theta_t(Y)$
exists for all $Y\in\fs_\rs(F)$, and then will give an estimation on
$\Theta_t$ to apply Lebesgue's Theorem.

\begin{lem}\label{lem. norm}
Given a compact subset $\omega$ of $\fs(F)$, we can choose
$c_0\geq0$ such that
$$1+\log\|h\|_{T\bs H}\leq c_0
\left(1+\log(\max\{1,|D^\fs(X)|_F^{-1}\})\right)$$ for $h\in
H,X\in\fc_\reg(F)$ such that $X^h\in\omega$.
\end{lem}
\begin{proof}
The proof is the same as that of \cite[Lemma 20.3]{ko}.
\end{proof}

We choose a compact set $\omega\subset\fs(F)$ such that
$\Supp(\phi)\subset\omega,\Supp(f)\subset\omega$. Fix a Cartan
subspace $\fc\subset\fs$. Let $T$ be the centralizer of $\fc$ in
$H$, and $A$ the maximal split torus in $T$. Notice that $A$
consists of elements of the form $\diag(a,a)$ where $a\in A_0$ for
some split torus $A_0$ contained in $\GL_n(F)$. Let $\omega_\fc$ be
the set of $X\in\fc(F)$ such that $X^h\in\omega$ for some $h\in H$.
Then $\omega_\fc$ is compact. For $X\in\omega_\fc,h\in H$, set
$$\phi_X(h)=\phi(X^h)\eta(h).$$
Note that $\phi_X$ has the following properties:\\
(i) $\Supp(\phi_X)\subset C_X$ for some subset $C_X\subset H$ which
is compact modulo $A$ and
$\phi_X(ah)=\phi_X(h)$ for $h\in H,a\in A$;\\
(ii) if $P'_0$ is a proper parabolic subgroup in $\GL_n(F)$ with
Levi decomposition $P'_0=M'_0U'_0$, and $A'_0\subset A_0$ where
$A'_0$ is the center of $M'_0$, then $$\int_{U'}\phi_X(uh)\ \d
u=0,\quad\textrm{for each}\ h\in H,$$ where $U'=U'_0\times U'_0$ is
a unipotent subgroup of $H$.

Let $K'_1$ be an open subgroup of $K'=\GL_n(\CO_F)$ such that
$\|k\|=1,\eta(k)=1$ for all $k\in K'_1$. Here we choose the Haar
measure on $H$ so that $\vol(K'_1\times K'_1)=1$. Fix an open
compact subgroup $K'_0$ of $\GL_n(F)$ such that
$$K'_0\subset (\bar{U}\cap K'_1)(M\cap K'_1)(U\cap K'_1)$$
for any parabolic subgroup $P'=M'U'$ in $\CP(A_0)$, where we denote
by $\CP(A_0)$ the set of all parabolic subgroups $P'=M'U'$ of
$\GL_n(F)$ such that $A_0$ is the center of $M'$. Set
$K_0=K'_0\times K'_0\subset H$. For an element $y\in H$, put
$K_0(y)=K_0\cap K_0^y$. Set
$$\|C_X\|_{T\bs H}=\sup_{h\in C_X}\|h\|_{T\bs H}.$$ The following
lemma is an analogue of \cite[Theorem 20]{hc1}, and we omit the
details of the proof since it is the same as that of \cite[Theorem
20]{hc1}

\begin{lem}\label{lem. thm20}
There exists a number $c\geq1$ with the following property. Let
$y\in H$, and $\Omega=\Omega(C_X,y)$ be the set of $h\in H$ such
that
$$1+\log\|h\|_{Z\bs H}\leq
c(1+\log\|C_X\|_{T\bs H})(1+\log\|y\|_{T\bs H}).$$ Then
$$\int_{K_0(y)}\phi_X(ykh)\ \d k=0$$ unless $h\in\Omega$.
\end{lem}

Now suppose that $X\in\fc_\reg(F)$ and $y\in H$ are such that
$X^y\in\omega$. Then $X\in\omega_\fc$. By Lemma \ref{lem. norm},
there is a positive constant $c_0$, only depending on $\omega$ and
$\fc$, such that
$$1+\log\|y\|_{T\bs H}\leq
c_0\left(1+\log(\max\{1,|D^\fs(X)|_F^{-1}\})\right).$$ Set
$\omega'_\fc=\omega_\fc\cap\fc_\reg(F)$. Then for any
$X\in\omega'_\fc$
we can choose a subset $C_X$ of $H$ such that\\
(1) $\Supp(\phi_X)\subset C_X$ and $C_X$ is compact modulo $A$;\\
(2) $1+\log\|C_X\|_{T\bs H}\leq
c_0\left(1+\log(\max\{1,|D^\fs(X)|_F^{-1}\})\right)$.\\
Let $\Omega_X$ ($X\in\omega'_\fc$) be the set of $h\in H$ such that
$$1+\log\|h\|_{Z\bs H}\leq
c_1\left(1+\log(\max\{1,|D^\fs(X)|_F^{-1}\})\right)^2,$$ where
$c_1=c\cdot c_0^2$ with $c$ as in Lemma \ref{lem. thm20}. Let
$\Phi_X$ denote the characteristic function of $\Omega_X$. Then we
have
$$\begin{aligned}
\Theta_t(X^y)&=\int_{Z\bs H}\Phi_t(h)\phi(X^{yh})\eta(h)\ \d h\\
&=\int_{Z\bs H}\Phi_t(h)\int_{K_1}\phi(X^{ykh})\eta(h)\ \d k\ \d h.
\end{aligned}$$
Note that $\|kh\|=\|h\|$ for $k\in K_1$. By Lemma \ref{lem. thm20}
we have
$$\int_{K_1}\phi(X^{ykh})\ \d k=\int_{K_1}\phi_X(ykh)\ \d k=0,$$
unless: $$1+\log\|k_1h\|_{Z\bs H}\leq c(1+\log\|C_X\|_{T\bs
H})(1+\log\|y\|_{T\bs H}),$$ where $k_1$ runs over a set of
representatives of $K_1/K_0(y_0)$ in $K_1$. Since $\|kh\|=\|h\|$ and
$$c(1+\log\|C_X\|_{T\bs H})(1+\log\|y\|_{T\bs H})\leq
cc_0^2\left(1+\log(\max\{1,|D^\fs(X)|_F^{-1}\})\right)^2,$$ the
integral $\int_{K_1}\phi(X^{ykh})=0$ unless $h\in\Omega_X$. Thus, if
$$t\geq c_1\left(1+\log(\max\{1,|D^\fs(X)|_F^{-1}\})\right)^2,$$ we get
$$\begin{aligned}
\Theta_t(X^y)&=\int_{Z\bs
H}\Phi_t(h)\Phi_X(h)\eta(h)\int_{K_1}\phi(X^{ykh})
\ \d k\ \d h\\
&=\int_{Z\bs H}\Phi_X(h)\int_{K_1}\phi(X^{ykh})\ \d k\ \d h\\
&=\int_{Z\bs H}\int_{K_1}\phi(X^{ykh})\eta(h)\ \d k\ \d h.
\end{aligned}$$
Therefore $\lim\limits_{t\ra+\infty}\Theta_t(X^y)$ exists for
$X^y\in\omega\cap\fs_\rs(F)$. By enlarging $\omega$,
$\lim\limits_{t\ra+\infty}\Theta_t(X)$ exists for all
$X\in\fs_\rs(F)$.

Now we estimate $\Theta_t(X)$. All the notations are the same as
above. We have $$\begin{aligned}
|\Theta_t(X^y)|&\leq\int_{Z\bs H}\Phi_X(y^{-1}h)|\phi(X^h)|\ \d h\\
&=\int_{A\bs H}|\phi(X^h)|\ \d h\int_{Z\bs A}\Phi_X(y^{-1}ah)\ \d a.
\end{aligned}$$
Recall that $\phi(X^h)\eta(h)=\phi_X(h)=0$ unless $h\in C_X$.
Suppose $h\in C_X$. We can assume $\log\|h\|\leq\log\|C_X\|$ and
$\log\|y\|=\log\|y\|_{T\bs H}$. Then $\Phi_X(y^{-1}ah)=0$ unless
$y^{-1}ah\in\Omega_X$. Since
$$1+\log\|a\|_{Z\bs H}\leq\left(1+\log\|h\|\right)
\left(1+\log\|y^{-1}ah\|_{Z\bs H}\right)\left(1+\log\|y\|\right),$$
we have $\Phi_X(y^{-1}ah)=0$ unless
$$1+\log\|a\|_{Z\bs H}\leq c_2
\left(1+\log(\max\{1,|D^\fs(X)|_F^{-1}\})\right)^4,$$ where
$c_2=c_1c_0^2$. Therefore
$$\begin{aligned}\int_{Z\bs A}\Phi_X(y^{-1}ah)\ \d a&\leq
\int_{1+\log\|a\|_{Z\bs H}\leq c_2
\left(1+\log(\max\{1,|D^\fs(X)|_F^{-1}\})\right)^4}\d a\\
&\leq
c_3\left(1+\log(\max\{1,|D^\fs(X)|_F^{-1}\})\right)^{4\ell}\end{aligned}$$
where $c_3$ is a positive constant, independent of the choice of
$X\in\omega_\fc'$, and $\ell=\dim Z\bs A$. This shows that
$$|\Theta_t(X^y)|\leq c_3\left(1+\log(\max\{1,|D^\fs(X)|_F^{-1}\})\right)
^{4\ell}\int_{A\bs H}|\phi(X^h)|\ \d h.$$ Notice that Theorem
\ref{thm. upper bound} also holds when $\eta={\bf1}$. Then we have:
$$\sup_{X\in\omega'_\fc}|D^\fs(X)|_F^{\frac{1}{2}}
\int_{A\bs H}|\phi(X^h)| \ \d h<+\infty.$$ Hence
$$|\Theta_t(X^y)|\leq c_4|D^\fs(X)|_F^{-\frac{1}{2}}
\left(1+\log(\max\{1,|D^\fs(X)|_F^{-1}\})\right)^{4\ell}$$ for all
$X\in\fc(F)$ and $y\in H$ such that
$X^h\in\omega'=\omega\cap\fs_\rs(F)$. Since there are only finitely
many non-conjugate Cartan subspaces in $\fs$, there exists a
constant $c_5$ such that
$$|\Theta_t(X)|\leq c_5|D^\fs(X)|_F^{-\frac{1}{2}}
\left(1+\log(\max\{1,|D^\fs(X)|_F^{-1}\})\right)^{4\ell}$$ for all
$X\in\omega'$ and all $t\geq1$.

It follows from the lemma below that the function
$$X\mapsto |D^\fs(X)|_F^{-\frac{1}{2}}
\left(1+\log(\max\{1,|D^\fs(X)|_F^{-1}\})\right)^{4\ell}$$ is
locally integrable on $\fs(F)$. Then Theorem \ref{thm. Iphi} follows
from Lebesgue's Theorem.

\begin{lem}\label{lem. Igusa}
There exists $\epsilon>0$ such that the function
$|D^\fs(X)|_F^{-\epsilon}$ is locally integrable on $\fc(F)$ for any
Cartan subspace $\fc$ of $\fs$.
\end{lem}

\begin{proof}
See \cite[Lemma 4.3]{zh}.
\end{proof}

\section{Local calculations}\label{sec: local calculations}
\subsection{Limit formulae}
In this subsection, we obtain formulae for $\wh{i}^\eta(X,Y)$
($X,Y\in\fs_\rs(F)$) and $\wh{i}(X,Y)$ ($X,Y\in\fs'_\rs(F)$) at
``infinity", which are analogues of \cite[VIII.1 Proposition]{wa95}.
The proof of \cite[VIII.1 Proposition]{wa95} is very technical. Here
we modify Waldspurger's proof a little to make it available in our
situation.

\paragraph{Statement}
Let $\fc$ be a Cartan subspace of $\fs$, and $\FT^-$ the maximal
$\theta$-split torus in $\FG$ whose Lie algebra is $\fc$. Let $\FT$
be the centralizer of $\fc$ in $\FH$, and write $\ft=\Lie(\FT)$. For
$X,Y\in\fc_\reg(F)$, define a bilinear form $q_{X,Y}$ on
$\fh(F)/\ft(F)$ by
$$q_{X,Y}(Z,Z')=\pair{[Z,X],[Y,Z']},$$ where the pairing
$\pair{\cdot,\cdot}$ is the one as before. One can check that the
form $q_{X,Y}$ is nondegenerate and symmetric. One can also verify
that $q_{X,Y}=q_{Y,X}$. We write
$\gamma_\psi(X,Y)=\gamma_\psi(q_{X,Y})$ for simplicity. Recall that,
by conventions, $T=\FT(F)$, $H=\FH(F)$.

Let $\fc'$ be a Cartan subspace of $\fs'$. Similarly, we denote by
$\FT'^-$ the maximal $\theta$-split torus in $\FG'$ whose Lie
algebra is $\fc'$, by $\FT'$ the centralizer of $\fc'$ in $\FH'$,
and by $\ft'$ the Lie algebra of $\FT'$. For $X,Y\in\fc'_\reg(F)$,
we also define a nondegenerate, bilinear and symmetric form
$q_{X,Y}$ on $\fh'(F)/\ft'(F)$ in the same way as above.

The following formulae depend on the choices of the Haar measures on
$T$ and $H$ (also on $T'$ and $H'$). Here we equip $H$ or $T$ with
the Haar measure so that the exponential map preserve the measure in
a neighborhood of 0 in $\fh(F)$ or $\ft(F)$. We make the similar
choices for the Haar measures on $T'$ and $H'$.

\begin{prop}\label{prop. i(X,Y)}
Let the notations be as above.
\begin{enumerate}
\item Let $X\in\fs_\rs(F)$ and $Y\in\fc_\reg(F)$. Then there exists
$N\in\BN$ such that if $\mu\in F^\times$ satisfying $v_F(\mu)<-N$,
we have the equality
$$\wh{i}^\eta(\mu X,Y)=\kappa(Y)\sum_{h\in T\bs H,\ h\cdot X\in\fc}\eta(h)
\gamma_\psi \left(\mu h\cdot X,Y\right)\psi\left(\pair{\mu h\cdot
X,Y}\right),$$ and
$$\wh{i}^\eta(X,\mu Y)=\kappa(\mu Y)\sum_{h\in T\bs H,\ h\cdot X\in\fc}\eta(h)
\gamma_\psi\left(\mu h\cdot X,Y\right)\psi\left(\pair{\mu h\cdot X,
Y}\right).$$
\item Let $X\in\fs'_\rs(F)$ and
$Y\in\fc'_\reg(F)$. Then there exists $N\in\BN$ such that if $\mu\in
F^\times$ satisfying $v_F(\mu)<-N$, we have the equality
$$\wh{i}(\mu X,Y)=\wh{i}(X,\mu Y)=\sum_{h\in T'\bs H',\ h\cdot X\in\fc'}
\gamma_\psi \left(\mu h\cdot X,Y\right)\psi\left(\pair{(\mu h\cdot
X, Y}\right).$$
\end{enumerate}
\end{prop}
In particular, the above expression is zero if $X$ is not conjugate
to any element of $\fc(F)$ (or $\fc'(F)$).

\paragraph{Proof of Proposition \ref{prop. i(X,Y)}}
We now prove the formula for $\wh{i}^\eta(\mu X,Y)$. The formula for
$\wh{i}^\eta(X,\mu Y)$ can be deduced from it. We leave the proof of
the formulae for $\wh{i}(\mu X,Y)$ and $\wh{i}(X,\mu Y)$ to the
reader. They can be proved in the same way.

Firstly, we introduce some notations. Let $\fq$ (resp. $\fp$) be the
unique complement of $\ft$ (resp. $\fc$) in $\fh$ (resp. $\fs$)
which is stable under the adjoint action of $\FT$. Denote by $S_\fc$
the set of roots of $\FT^-$ in $\fg(\bar{F})$. For each subspace
$\ff\subset\fg(F)$ such that the restriction of $\pair{\cdot,\cdot}$
to $\ff$ is nondegenerate and for each $\CO_F$-lattice
$L\subset\ff$, set $\wt{L}=\{\ell\in \ff:\ \forall\ell'\in
L,\psi(\pair{\ell',\ell})=1 \}$. We denote by $L_\fc$ the
$\CO_F$-lattice of $\fc(F)$ such that
$$\wt{L}_\fc=\{Z\in\fc(F):\
\forall\alpha\in S_\fc,v_F(\alpha(Z))\geq0\}.$$ Fix $\CO_F$-lattices
$L_\fp\subset\fp(F)$, $L_\ft\subset\ft(F)$ and $L_\fq\subset\fq(F)$.
Set $L_\fs=L_\fc\oplus L_\fp$, $L_\fh=L_{\ft}\oplus L_\fq$,
$L=L_\fs\oplus L_\fh$.

For simplicity, write $d=\dim_F(\fg(F))=4n^2$. Denote by $F[U]_d$
the set of monic polynomials of degree $d$ with coefficients in $F$.
For $P\in F[U]_d$, write
$$P(U)=\sum_{i=0}^d s_i(P)U^{d-i}.$$ For $a\in\BZ$ and
$P_1,P_2\in F[U]_d$, we write $P_1\equiv P_2\mod\varpi^a\CO_F$ if
$v_F(s_i(P_1)-s_i(P_2))\geq a$ for each $i=0,1,...,d$. For each
$Z\in\fg(F)$, denote by $P_Z$ the characteristic polynomial of
$\ad(Z)$ acting on $\fg(F)$. Then $P_Z\in F[U]_d$.

Fix an integer $c\in\BN$ satisfying the following conditions.
\begin{enumerate}
\item
For each $a\in\BN$, $a\geq c$, we have
\begin{itemize}
\item $\varpi^a L_\fh\subset V_\fh$ and $\varpi^a L\subset V_\fg$;
\item $K_a:=\exp(\varpi^aL_\fh)$ is a subgroup of $K=\GL_n(\CO_F)
\times\GL_n(\CO_F)$, and $\eta|_{K_a}=1$;
\item the action of $K_a$ stabilizes $L_\fs$ (hence stabilizes
$\wt{L}_\fs$).
\end{itemize}
\item
For each $a\in\BN,\ a\geq c$, and each $Z\in\varpi^aL_\fh$, we have
\begin{itemize}
\item $(\exp Z)\cdot Y-Y-[Z,Y]\in\varpi^{2a-c}L_\fs$;
\item $(\exp Z)\cdot Y-Y-[Z,Y]-\frac{1}{2}[Z,[Z,Y]]
\in\varpi^{3a-c}L_\fs$.
\end{itemize}
\item
Denote by $C(X)$ the set of $X'\in\fc(F)$ satisfying that there
exists $h\in H$ such that $h\cdot X'=X$, which is a finite set. We
require that:
\begin{itemize}
\item if $a\in\BN,\ a\geq c$, $X',X''\in C(X)$, and $\gamma\in
K_a$ satisfying $\gamma\cdot X'=X''$, then $X'=X''$;
\item for each $X'\in C(X)$, denote by $\wt{L}_\fq^{X'}$ the
dual of $L_\fq$ in $\fq(F)$ with respect to the form $q_{X',Y}$;
then require $\varpi^c\wt{L}_\fq^{X'} \subset 2\varpi^{-c}L_\fq$.
\end{itemize}
\item
If $Z\in\fp(F)$ satisfying $[Y,Z]\in\wt{L}_\fh$, then
$Z\in\varpi^{-c}\wt{L}_\fp$.
\item
For each $h\in H$, denote by $c(h)$ the unique element of $\BZ$ such
that
$$X^h\in\varpi^{-c(h)}\wt{L}_\fs-\varpi^{-c(h)+1}\wt{L}_\fs.$$
Since $X\in\fs_\rs(F)$, the set $\{c(h),h\in H\}$ has a lower bound.
We require that
\begin{itemize} \item for each $h\in H$, $c(h)\geq-c$.\end{itemize}
\item
Fix a basis $\CB$ of $\fg(\bar{F})$ formed of basis of
$\fc(\bar{F})$ and $\ft(\bar{F})$, and root vectors associated to
$S_\fc$. We require that
\begin{itemize}
\item for each $Z\in\wt{L}_\fp$, the coefficients of the
matrix representation of $\ad(Z)$ with respect to the basis $\CB$
are of valuation $\geq -c$;
\item for each $i=0,1,...,d$, $v_F(s_i(P_X))\geq-c$.
\end{itemize}
\item
There exists an open compact set $\Omega\subset\fc_\reg(F)$ such
that if $Z\in\fc_\reg(F)$ satisfying $P_Z\equiv
P_X\mod\varpi^c\CO_F$, then $Z\in\Omega$.
\end{enumerate}

The integer $c$ is fixed from now on. We also fix an open compact
$\Omega$ satisfying condition (vii). The following lemma actually is
\cite[VIII.3 Lemme]{wa95}, and whose proof can be applied in our
situation.

\begin{lem}\label{lem technique 8.3.}
There exists $c'\in\BN$, $c'\geq c$, such that if $a\in\BN,\ a\geq
c'$, and $Z\in\Omega+\varpi^{a+c'}\wt{L}_\fp$, then there exists
$\gamma\in K_a$ such that $\gamma\cdot Z\in\Omega$.
\end{lem}

From now on, we fix an integer $c'$ as in the above lemma. Set
\begin{equation}\label{equ. i.1}
N=2(d+8)c+6c'+12.
\end{equation}
Let $\mu\in F^\times$ be such that $v_F(\mu)<-N$. Choose $m\in\BN$
such that
\begin{itemize}
\item the functions $Y'\mapsto\wh{i}^\eta(\mu X,Y')$,
$Y'\mapsto|D^\fs(Y')|_F$ and $\kappa(Y)$ are constant on
$Y+\varpi^mL_\fs$;
\item for each $X'\in C(X)$, $\mu X'\in\varpi^{-m}\wt{L}_\fs$.
\end{itemize}

Let $f$ be the characteristic function of $Y+\varpi^mL_\fs$, and
$f'$ be the characteristic function of $\varpi^{-m}\wt{L}_\fs$. Then
we have
\begin{equation}\label{equ. i.3}
\begin{aligned}
\wh{I}^\eta(\mu X,f)&=\int_{\fs(F)}\wh{i}^\eta(\mu
X,Y')\kappa(Y')f(Y')
|D^\fs(Y')|_F^{-1/2}\ \d Y'\\
&=\vol(\varpi^mL_\fs)|D^\fs(Y)|_F^{-1/2}\kappa(Y)\wh{i}^\eta(\mu
X,Y).
\end{aligned}
\end{equation}
On the other hand, it is easy to verify that
$$\wh{f}(Y')=\vol(\varpi^mL_\fs)\psi(\pair{Y,Y'})f'(Y').$$
Hence
$$\wh{I}^\eta(\mu X,f)=|D^\fs(\mu X)|_F^{1/2}\vol(\varpi^mL_\fs)
\int_{T\bs H}f'(\mu X^h)\psi \left(\pair{Y,\mu X^h}\right)\eta(h)\
\d h.$$ Set
\begin{equation}\label{equ. i.4}
a=[-v_F(\mu)/2]-2c-c'-1.
\end{equation}
By (\ref{equ. i.1}), $a\geq c$. Fix a set of representatives
$\Gamma$ in $H$ for the double coset $T\bs H/K_a$. By condition
(iii), we can suppose that if there exist $h\in\Gamma$ and $h'\in
ThK_a$ such that $X^{h'}\in\fc(F)$, then $X^h\in\fc(F)$. Then we
have
$$\wh{i}^\eta(\mu X,Y)=|D^\fs(\mu X)D^\fs(Y)|_F^{1/2}\kappa(Y)
\sum_{h\in\Gamma}\vol(T\bs ThK_a)f'(\mu X^h)\eta(h)i(h),$$ where
$$i(h)=\int_{K_a}\psi(\pair{Y,\mu X^{h\gamma}})\ \d\gamma.$$
Fix $h\in\Gamma$. Choose $b\in\BN$ such that
\begin{equation}\label{equ. i.5}
\begin{aligned}&\bullet \left(c+c(h)-v_F(\mu)\right)/2\leq b
\leq c(h)-v_F(\mu)-1-2c;\\
&\bullet \textrm{ if }c(h)\leq c,
\begin{array}{lll}b\leq\left\{\begin{array}{ll}
-(d+2)c-1-v_F(\mu),\\
-c-c'-1-a-v_F(\mu),
\end{array}\right.
\end{array}
\end{aligned}
\end{equation}
which implies $b\geq a$. Fix a set of representatives $\Gamma'$ of
$K_a/K_b$. Then we have
$$i(h)=\sum_{g\in\Gamma'}i(h,g),$$ where
$$i(h,g)=\int_{K_b}\psi\left(\pair{Y,\mu X^{hg\gamma}}\right)\ \d\gamma.$$
Fix $g\in\Gamma'$, and set $X'=X^{hg}$. Then
$$i(h,g)=\int_{\varpi^bL_\fh}\psi\left(\pair{\exp Z\cdot Y,\mu X'}\right)\
\d Z.$$ Notice that since $K_b$ stabilizes $L_\fs$ and $\wt{L}_\fs$,
then $c(hg)=c(h)$. In particular, $X'\in\varpi^{-c(h)}\wt{L}_\fs$.
By (\ref{equ. i.5}), we have
$$\psi\left(\pair{Z,\mu X'}\right)=1$$ for each $Z\in\varpi^{2b-c}L_\fs.$
Notice that $b\geq c$. For $Z\in\varpi^bL_\fh$, by condition (ii),
we have
$$\begin{aligned}
\psi\left(\pair{\exp Z\cdot Y,\mu
X'}\right)&=\psi\left(\pair{Y+[Z,Y],\mu X'}\right)\\
&=\psi\left(\pair{Y,\mu X'}\right)\psi\left(\pair{Z,[Y,\mu
X']}\right).
\end{aligned}$$
Therefore we see that $i(h,g)=0$ if $[Y,\mu
X']\notin\varpi^{-b}\wt{L}_\fh$. We make the following claim:
$$(*)\quad \textrm{ if }[Y,\mu X']\in\varpi^{-b}\wt{L}_\fh,
\textrm{ then }X^h\in\fc(F).$$ Now we prove this claim. Suppose
$[Y,\mu X']\in\varpi^{-b}\wt{L}_\fh$, in other words,
$[Y,X'_\fp]\in\mu^{-1}\varpi^{-b}\wt{L}_\fq$, where
$X'=X'_\fc+X'_\fp$ is the decomposition of $X'$ with respect to
$\fs=\fc\oplus\fp$. Thus, by condition (iv),
\begin{equation}\label{equ. i.6}
X'_\fp\in\mu^{-1}\varpi^{-b-c}\wt{L}_\fp.
\end{equation}
Moreover, by (\ref{equ. i.5}),
$X'_\fp\in\varpi^{-c(h)+1}\wt{L}_\fp$. By the definition of $c(h)$
and that $c(hg)=c(h)$, we deduce that
$$X'_\fc\in\varpi^{-c(h)}\wt{L}_\fc-\varpi^{-c(h)+1}\wt{L}_\fc.$$
Set $R=\{\alpha\in S_\fc:\ v_F(\alpha(X'_\fc))<-c(h)+1\}$. The above
relation and the definition of $\wt{L}_\fc$ imply that
$R\neq\emptyset$. Set $r=\#R$, we calculate the coefficient
$s_r(P_{X'})$. This is a sum of products of the coefficients of the
matrix representations of $\ad X'_\fc$ and $\ad X'_\fp$ with respect
to the basis $\CB$. By (\ref{equ. i.5}), (\ref{equ. i.6}) and
condition (vi), the coefficients of $\ad X'_\fp$ are of valuation
$\geq -c(h)+1$. The same relation holds for the coefficients of $\ad
X'_\fc$ other than that of $\alpha(X'_\fc)$ for $\alpha\in R$. The
term $\prod_{\alpha\in R}\alpha(X'_\fc)$, which occurs in
$s_r(P_{X'})$, is of the valuation strictly less than that of any
other term. Thus
$$v_F(s_r(P_{X'}))=v_F(\prod_{\alpha\in
R}\alpha(X'_\fc))<r(-c(h)+1).$$ Since $X'$ is conjugate to $X$ by
the action of $H$, then $P_{X'}=P_X$. By condition (vi), we have
$$-c<r(-c(h)+1),$$ therefore
\begin{equation}\label{equ. i.7}
c(h)\leq c.
\end{equation}
Let $i\in\{1,2,...,d\}$. We now compare the coefficients
$s_i(P_{X'})$ and $s_i(P_{X'_\fc})$. Their difference is a sum of
products of coefficients of the matrix representations of $\ad
X'_\fc$ and $\ad X'_\fp$ with respect to the basis $\CB$, and at
least one coefficient of $\ad X'_\fp$ is involved in these products.
By (\ref{equ. i.7}), the coefficients of $\ad X'_\fc$ are of
valuation $\geq-c(h)\geq-c$. By (\ref{equ. i.5}), (\ref{equ. i.6}),
(\ref{equ. i.7}) and condition (vi), the coefficients of $\ad
X'_\fp$ are of valuation $\geq dc$. Therefore
$$v_F\left(s_i(P_{X'})-s_i(P_{X'_\fc})\right)\geq-(i-1)c+dc\geq c.$$
In other words, $P_{X'_\fc}\equiv P_{X'}\mod\varpi^c\CO_F$. Thus, by
condition (vii), $X'_\fc\in\Omega$. By (\ref{equ. i.5}), (\ref{equ.
i.6}) and (\ref{equ. i.7}), $X'_\fp\in\varpi^{a+c'}\wt{L}_\fp$. By
(\ref{equ. i.1}) and (\ref{equ. i.4}), $a\geq c'$. By Lemma \ref{lem
technique 8.3.}, there exists $\gamma\in K_a$ such that $\gamma\cdot
X'\in\fc(F)$. By the choice of $\Gamma$, we have $X^h\in\fc(F)$. Now
we have finished the proof the claim.

From now on, we suppose that $X^h\in\fc(F)$. Thus $f'(\mu X^h)=1$ by
the condition on $f'$. Notice that the multiplication by $h^{-1}$
induces an isomorphism from $T\bs ThK_a$ to $T\bs TK_a$. Now we have
\begin{equation}\label{equ. i.8}\begin{aligned}
\wh{i}^\eta(\mu X,Y)=&\kappa(Y)|D^\fs(\mu
X)D^\fs(Y)|_F^{1/2}\vol(K_a)^{-1}\vol(T\bs TK_a)\\
&\times \sum_{X'=X^h\in C(X)}\eta(h)j(X'),\end{aligned}
\end{equation}
where
$$\begin{aligned}j(X')&=\int_{K_a}\psi\left(\pair{Y,\mu X'^\gamma}
\right)\ \d\gamma\\
&=\int_{\varpi^aL_\fh}\psi\left(\pair{\exp Z\cdot Y,\mu X'}\right)\
\d Z.
\end{aligned}$$
Fix $X'\in C(X)$. By (\ref{equ. i.1}) and (\ref{equ. i.4}),
$\psi\left(\pair{Y',\mu X'}\right)=1$ for $Y'\in\varpi^{3a-c}L_\fs$.
Since $Y,X'\in\fc(F)$, then for any $Z\in\fg(F)$,
$\pair{[Z,Y],X'}=\pair{Z,[Y,X']}=0$. By condition (ii), we have
$$\begin{aligned}
j(X')&=\psi\left(\pair{Y,\mu
X'}\right)\int_{\varpi^aL_\fh}\psi\left(\frac{1}{2}\pair{[Z,[Z,Y]],\mu
X'}\right)\ \d Z\\
&=\psi\left(\pair{Y,\mu
X'}\right)\int_{\varpi^aL_\fh}\psi\left(\frac{1}{2}\pair{[Z,Y],[\mu
X',Z]}\right)\ \d Z\\
&=\psi\left(\pair{Y,\mu
X'}\right)\vol(\varpi^aL_\ft)\int_{\varpi^aL_\fq}\psi\left(\frac{1}{2}
q_{\mu X',Y}(Z)\right)\ \d Z.
\end{aligned}$$
Since $a\leq-c-v_F(\mu)/2$ and by condition (iii), we obtain
\begin{equation}\label{equ. i.9}
j(X')=\vol(\varpi^aL_\ft)\vol(\varpi^aL_\fq)^{1/2}\vol(\varpi^{-a}\check{L}_\fq)
^{1/2}\gamma_\psi(q_{\mu X',Y})\psi\left(\pair{Y,\mu X'}\right),
\end{equation}
where $\check{L}_\fq$ is the dual lattice of $L_\fq$ with respect to
the form $q_{\mu X',Y}$. There is a relation:
\begin{equation}\label{equ. i.10}\vol(K_a)=\vol(T\bs TK_a)\vol(T\cap K_a)
=\vol(T\bs TK_a)\vol(\varpi^aL_\ft).
\end{equation}
By definition
$$\begin{aligned}
\check{L}_\fq&=\left\{Z\in\fq(F):\ \forall Z'\in L_\fq,\
\psi\left(\pair{[Z,\mu X'],[Y,Z']}\right)=1\right\}\\
&=\left\{Z\in\fq(F):\ \forall Z'\in L_\fq,\ \psi\left(\pair{[[Z,\mu
X'],Y],Z']}\right)=1\right\}\\
&=\left\{Z\in\fq(F):\ [[Z,\mu X'],Y]\in\wt{L}_\fq\right\}.
\end{aligned}$$
In other words, $$(\ad Y)\circ(\ad\mu
X')(\check{L}_\fq)=\wt{L}_\fq,$$ and
\begin{equation}\label{equ. i.11}
\vol(\check{L}_\fq)=|D^\fs(Y)D^\fs(\mu X')|_F^{-1}\vol(\wt{L}_\fq).
\end{equation}
On the other hand, we have the relation
\begin{equation}\label{equ. i.12}
\vol(L_\fq)\vol(\wt{L}_\fq)=1.
\end{equation}
Then Proposition \ref{prop. i(X,Y)} follows.

\subsection{Formulae for $\gamma_\psi(X,Y)$}
For $X,Y\in\fc_\reg(F)$ or $\fc'_\reg(F)$, since $\gamma_\psi(X,Y)$
appears in the expression of $\wh{i}^\eta(X,Y)$ or $\wh{i}(X,Y)$ as
in Proposition \ref{prop. i(X,Y)}, we need to know an explicit
formula of $\gamma_\psi(X,Y)$. In this subsection, we show a formula
(see Proposition \ref{prop. gamma1}) of $\gamma_\psi(X,Y)$ for $X,Y$
lying in a Cartan subspace of the Lie algebra associated to a
general symmetric pair. This result is an analogue of \cite[VIII.5
Lemme]{wa95}.

Now we introduce some notations. Assume that $(\FG,\FH,\theta)$ is a
general symmetric pair, as introduced in \S\ref{sec: symmetric pairs
1}. Let $\fs$ be the Lie algebra associated to $(\FG,\FH,\theta)$,
and $\fc$ a Cartan subspace of $\fs$. Let $\FT$ be the centralizer
of $\fc$ in $\FH$ and write $\ft=\Lie(\FT)$. Fix a $\FG$-invariant
and $\theta$-invariant nondegenerate symmetric bilinear form
$\pair{\ ,\ }$ on $\fg(F)$. Then, for $X,Y\in\fc_\reg(F)$, the
bilinear form $q_{X,Y}$ on $\fh(F)/\ft(F)$ defined by
$$q_{X,Y}(Z,Z')=\pair{[Z,X],[Y,Z']}$$ is nondegenerate and
symmetric. Write $\gamma_\psi(X,Y)=\gamma_\psi(q_{X,Y})$. For any
subspace $\ff$ of $\fg(F)$ such that the restriction of $\pair{\ ,\
}$ on $\ff$ is nondegenerate, we write $\gamma_\psi(\ff)$ for the
Weil index associated to $\psi$ and the form $\pair{\ ,\ }$ on
$\ff$.

Let $\FT^-$ be the maximal $\theta$-split torus in $\FG$ whose Lie
algebra is $\fc$. Denote by $S_\fc$ the set of roots of $\FT^-$ in
$\fg(\bar{F})$. Write $\Gamma_F$ for the absolute Galois group
$\Gal(\bar{F}/F)$. Then $\Gamma_F$ acts on $S_\fc$. For $\alpha\in
S_\fc$, denote by $m_\alpha$ its multiplicity in $\fg(\bar{F})$.
Since $\fc\subset\fs$, for $\alpha\in S_\fc$, we have
$\theta(\alpha)=-\alpha$ and $m_\alpha=m_{-\alpha}$. For $\alpha\in
S_\fc$, denote by $\Gamma_{\pm\alpha}$ the stabilizer of
$\{\alpha,-\alpha\}$ in $\Gamma_F$, by $F_{\pm\alpha}$ the fixed
field of $\Gamma_{\pm\alpha}$ in $\bar{F}$, and by $S_\fc^*$ a fixed
set of representatives of orbits $\{\alpha,-\alpha\}$. Notice that,
if $X,Y\in\fc_\reg(F)$, $\alpha(X)\alpha(Y)\in F_{\pm\alpha}$.

For $\alpha\in S_\fc^*$, denote by $\psi'$ the character
$\psi\circ\RTr_{F_{\pm\alpha}/F}$ of $F_{\pm\alpha}$. Set
$$\gamma_{F_{\pm\alpha}}(\alpha(X)\alpha(Y),\psi')=\frac{\gamma_{\psi'}
\left(\alpha(X)\alpha(Y)q\right)} {\gamma_{\psi'}(q)}$$ where $q$ is
the quadratic form on $F_{\pm\alpha}$ defined by
$q(\lambda)=\lambda^2$.

\begin{prop}\label{prop. gamma1}
Let the notations be as above. Then, for $X,Y\in\fc_\reg(F)$, we
have
$$\begin{aligned}\gamma_\psi(X,Y)=&\gamma_\psi(\ft(F))^{-1}\gamma_\psi
(\fh(F))\\
&\times\prod_{\alpha\in
S_\fc^*}\left((\alpha(X)\alpha(Y),2)_{F_{\pm\alpha}}
\gamma_{F_{\pm\alpha}}(\alpha(X)\alpha(Y),\psi')\right)^{m_\alpha}.
\end{aligned}$$
where $(\ ,\ )_{F_{\pm\alpha}}$ is the Hilbert symbol on
$F_{\pm\alpha}$.
\end{prop}

\begin{proof}
Notice that for $\alpha\in S_\fc$ we have
$m_{\sigma\alpha}=m_\alpha$ for every $\sigma\in\Gamma_F$. For each
root space $\fg_\alpha$ associated to $\alpha\in S_\fc$ we can
choose its basis $\{E_\alpha^1,...,E_\alpha^{m_\alpha}\}$ so that:
(1) $\sigma(E^i_\alpha)=E_{\sigma\alpha}^i$ for each
$\sigma\in\Gamma_F$; (2) $\theta(E_\alpha^i)=E_{-\alpha}^i$; (3)
$\pair{E^i_\alpha,E^j_{-\alpha}}=\delta_{ij}$.

Consider the homomorphism
$$\tau:\prod_{S^*_\fc}m_\alpha F_{\pm\alpha}\lra\fg(\bar{F}),\quad
(\lambda^i_\alpha)\mapsto\sum_{\alpha}\sum_{i=1}^{m_\alpha}
\sum_{\sigma\in\Gamma/\Gamma_{\pm\alpha}}\sigma(\lambda_\alpha^i)
\left(E^i_{\sigma\alpha}+E^i_{-\sigma\alpha}\right).$$ In fact the
image of $\tau$ lies in $\fg(F)$ and $\tau$ defines an isomorphism
$$\prod_{S_\fc^*}m_\alpha F_{\pm\alpha}\stackrel{\sim}{\lra}\fq(F),$$
where $\fq$ is the unique complement of $\ft$ in $\fh$ which is
stable under the adjoint action of $\FT$. For
$(\lambda_\alpha^i)\in\prod_{S_\fc^*}m_\alpha F_{\pm\alpha}$, we
have
$$\begin{aligned}
q_{X,Y}\left(\tau\left((\lambda_\alpha^i)\right)\right)&=
\sum_{\alpha\in
S_\fc^*}\sum_{i=1}^{m_\alpha}\sum_{\sigma\in\Gamma/\Gamma_\alpha}
\sigma(\lambda^i_\alpha)^2\pair{[E^i_{\sigma\alpha}+E^i_{-\sigma\alpha},X],
[Y,E^i_{\sigma\alpha}+E^i_{-\sigma\alpha}]}\\
&=\sum_{\alpha,i,\sigma}\sigma(\lambda_\alpha^i)^2
\left(-\sigma\alpha(X)\sigma\alpha(Y)\right)
\pair{E^i_{\sigma\alpha}-E^i_{-\sigma\alpha},
E^i_{\sigma\alpha}-E^i_{-\sigma\alpha}}\\
&=\sum_{\alpha,i,\sigma}\sigma(\lambda_\alpha^i)^2\sigma\alpha(X)\sigma\alpha(Y)
\pair{E^i_{\sigma\alpha}+E^i_{-\sigma\alpha},
E^i_{\sigma\alpha}+E^i_{-\sigma\alpha}}\\
&=\sum_{\alpha\in
S_\fc^*}\sum_{i=1}^{m_\alpha}q_{X,Y,\alpha}(\lambda_\alpha^i),
\end{aligned}$$
where $q_{X,Y,\alpha}(\lambda)$ is the quadratic form on
$F_{\pm\alpha}$ defined by
$$q_{X,Y,\alpha}(\lambda)=\RTr_{F_{\pm\alpha}/F}\left(2
\alpha(X)\alpha(Y)\lambda^2\right).$$ Therefore
$$\gamma_\psi(X,Y)=\prod_{\alpha\in S_\fc^*}
\gamma_\psi(q_{X,Y,\alpha})^{m_\alpha}.$$ For $\alpha\in S_\fc^*$,
let $q'_{X,Y,\alpha}$ be the quadratic form on $F_{\pm\alpha}$
defined by:
$$q'_{X,Y,\alpha}(\lambda)=2\alpha(X)\alpha(Y)\lambda^2.$$ Then
$\gamma_\psi(q_{X,Y,\alpha})=\gamma_{\psi'}(q'_{X,Y,\alpha})$, and
$$\gamma_{\psi'}(q'_{X,Y,\alpha})=\left(\alpha(X)\alpha(Y),2\right)_{F_{\pm\alpha}}
\gamma_{F_{\pm\alpha}}(\alpha(X)\alpha(Y),\psi')\gamma_{\psi'}(q'_\alpha),$$
where $q'_\alpha$ is the quadratic form on $F_{\pm\alpha}$ defined
by $q'_\alpha(\lambda)=2\lambda^2$. Therefore
$$\gamma_\psi(q_{X,Y,\alpha})=\left(\alpha(X)\alpha(Y),2\right)_{F_{\pm\alpha}}
\gamma_{F_{\pm\alpha}}(\alpha(X)\alpha(Y),\psi')\gamma_\psi(q_\alpha),$$
where
$$q_\alpha(\lambda)=\RTr_{F_{\pm\alpha}/F}(2\lambda^2)=\RTr_{F_{\pm\alpha}/F}
\left(\pair{E_\alpha^i+E_{-\alpha}^i,E_\alpha^i+E_{-\alpha}^i}\lambda^2\right).$$
In summary, we deduce that
$$\gamma_\psi(X,Y)=\prod_{\alpha\in S_\fc^*}\left((\alpha(X)\alpha(Y),2)_{F_{\pm\alpha}}
\gamma_{F_{\pm\alpha}}(\alpha(X)\alpha(Y),\psi')\gamma_\psi(q_\alpha)\right)^{m_\alpha}.$$
On the other hand, by the same argument as above, we can show that
$$\gamma_\psi(\fq(F))=\prod_{\alpha\in S_\fc^*}\gamma_\psi(q_\alpha)^{m_\alpha}.$$
Together with the obvious relation
$$\gamma_\psi(\fq(F))=\gamma_\psi(\ft(F))^{-1}\gamma_\psi(\fh(F)),$$
we complete the proof.
\end{proof}

\subsection{Comparison lemma}
To obtain the main result of this subsection, we need the following
lemma.

\begin{lem}\label{lem. inner form of t}
Let $X\in\fc_\reg(F)$ and $Y\in\fc'_\reg(F)$ be such that
$X\leftrightarrow Y$. Then there exists an element $x\in\GL_{2n}(E)$
such that $\Ad(x)Y=X$, and $\Ad(x)$ induces isomorphisms
$\Ad(x):\ft'\ra\ft$ and $\Ad(x):\fc'\ra\fc$ over $F$.
\end{lem}

\begin{proof}
It suffices to prove this for
$X=\begin{pmatrix}0&{\bf1}_n\\A&0\end{pmatrix}$ and
$Y=\begin{pmatrix}0&\gamma B\\ \bar{B}&0\end{pmatrix}$, where
$A\in\GL_n(F)$ is regular semisimple and $A=\gamma B\bar{B}$. Then
we have
$$\fc(F)=\left\{\begin{pmatrix}0&C\\AC&0\end{pmatrix}:\ C\in
\fg\fl_n(F), AC=CA\right\},$$
$$\ft(F)=\left\{\begin{pmatrix}D&0\\0&D\end{pmatrix}:\ D\in
\fg\fl_n(F),AD=DA\right\},$$
$$\fc'(F)=\left\{\begin{pmatrix}0&\gamma P\\ \bar{P}&0\end{pmatrix}:\
P\in \fg\fl_n(E), B\bar{P}=P\bar{B}\right\},$$ and
$$\ft'(F)=\left\{\begin{pmatrix}Q&0\\0&\bar{Q}\end{pmatrix}:\ Q\in
\fg\fl_n(E),B\bar{Q}=QB\right\}.$$ Take
$x=\begin{pmatrix}{\bf1}_n&0\\0&\gamma
B\end{pmatrix}\in\GL_{2n}(E)$. We claim that $\Ad(x)$ satisfies the
required condition. By the above relation, it is easy to see that:
\begin{enumerate}
\item $\Ad(x)\cdot\begin{pmatrix}0&\gamma P\\ \bar{P}&0\end{pmatrix}
=\begin{pmatrix}0&\gamma PB^{-1}\\ APB^{-1}&0\end{pmatrix}$,
$APB^{-1}=PB^{-1}A$;
\item $\Ad(x)\cdot\begin{pmatrix}Q&0\\0&\bar{Q}\end{pmatrix}
    =\begin{pmatrix}Q&0\\0&Q\end{pmatrix}$, $AQ=QA$.
\end{enumerate}
Therefore we have to show that $PB^{-1}\in\fg\fl_n(F)$,
$Q\in\fg\fl_n(F)$.

Note that since $A=B\bar{B}$, $A$ commutes with $B$. It is easy to
see that $P$ and $Q$ also commute with $A$. Hence $P$ and $Q$
commute with $B$, since $A$ is regular. Therefore the relation
$B\bar{P}=P\bar{B}$ implies that $PB^{-1}=\bar{P}\bar{B}^{-1}$; the
relation $B\bar{Q}=QB$ implies that $\bar{Q}=Q$, which concludes the
proof.
\end{proof}

Now let $X\in\fc_\reg(F)$ and $Y\in\fc'_\reg(F)$ be such that
$X\leftrightarrow Y$. Then we can take an $x\in\GL_{2n}(E)$ as in
the above lemma. For any $V\in\fc'_\reg(F)$, put $U=\Ad(x)V$.

\begin{lem}\label{lem. compare gamma}
Let $X,Y,U,V$ be as above. Then we have the following relations
$$\pair{X,U}=\pair{Y,V},$$ and
$$\gamma_\psi(X,U)=\gamma_\psi(\fh(F))\gamma_\psi(\fh'(F))^{-1}
\gamma_\psi(Y,V).$$
\end{lem}
\begin{proof}
The first relation follows directly from the above lemma. The second
relation follows from the above lemma, Proposition \ref{prop.
i(X,Y)} and the similar arguments of equation (6) in \cite[page
96]{wa1}.
\end{proof}

\subsection{Test functions}
This subsection is devoted to showing that we can construct specific
$\CC_c^\infty$-functions satisfying certain ``good" matching
conditions. Such functions will play an important role in proving
Theorem \ref{thm. fourier preserve transfer} by global method. The
result below is an analogue of \cite[Proposition in \S8.2]{wa97}.

\begin{prop}\label{prop. local prop}
Let $Y_0\in\fc'_\reg(F)\subset\fs'_\rs(F)$ and
$X_0\in\fc_\reg(F)\subset\fs_\rs(F)$ be such that
$X_0\leftrightarrow Y_0$. Then there exist functions
$f\in\CC_c^\infty(\fs(F))$ and $f'\in\CC_c^\infty(\fs'(F))$
satisfying the following conditions.
\begin{enumerate}
\item If $X\in\Supp(f)$, there exists $Y\in\fc'_\reg(F)$ such that
$X\leftrightarrow Y$.
\item If $Y\in\Supp(f')$, $Y$ is
$H'$-conjugate to an element in $\fc'_\reg(F)$.
\item $f$ and $f'$ are smooth transfer of each other.
\item There is an equality
$$\kappa(X_0)\wh{I}^\eta(X_0,f)=c\wh{I}(Y_0,f')\neq0,$$
where $c=\gamma_\psi(\fh(F))\gamma_\psi(\fh'(F))^{-1}$.
\end{enumerate}
\end{prop}

\begin{proof}
Let $W_\fc$ (resp. $W_{\fc'}$) be the Weyl group associated to $\fc$
(resp. $\fc'$), i.e. $W_\fc=N_H(\fc)/Z_H(\fc)$ (resp.
$W_{\fc'}=N_{H'}(\fc')/Z_{H'}(\fc')$). Set
$$C(X_0)=\{X\in\fc_\reg(F):\ X=i(X_0)\textrm{ for some }i\in W_\fc\},$$
and $$C(Y_0)=\{Y\in\fc'_\reg(F):\ Y=i(Y_0)\textrm{ for some }i\in
W_{\fc'}\}.$$

By Lemma \ref{lem. inner form of t}, we fix an isomorphism
$\varphi:\fc'(F)\ra\fc(F)$ such that $\varphi(Y_0)=X_0$. Fix
$V_0\in\fc'_\reg(F)$ and $U_0:=\varphi(V_0)\in\fc_\reg(F)$ so that
if $X\in C(X_0)-X_0$ (resp. $Y\in C(Y_0)-Y_0$), we have
$\pair{X-X_0,U_0}\neq0$ (resp. $\pair{Y-Y_0,V_0}\neq0$), and
moreover, $\kappa(U_0)=\kappa(X_0)$. We make the following choices.
\begin{enumerate}
\item Fix an integer $r\geq1$ such that:
\begin{itemize}
\item $1+\varpi^r\CO_F\subset F^{\times2}$;
\item the sets $i((1+\varpi^r\CO_F)U_0)$ (resp.
$i((1+\varpi^r\CO_F)V_0)$), for $i\in W_\fc$ (resp. $i\in
W_{\fc'}$), are mutually disjoint.
\end{itemize}
\item There exists an integer $N$ such that if $\mu\in F^\times$
satisfying $v_F(\mu)<-N$, we have
\begin{itemize}
\item for each $X\in C(X_0)-X_0$ (resp. $Y\in C(Y_0)-Y_0$), the
character
$\alpha\mapsto\psi\left(\varpi^r\mu\alpha\pair{X-X_0,U_0}\right)$
(resp.
$\alpha\mapsto\psi\left(\varpi^r\mu\alpha\pair{Y-Y_0,Y_0}\right)$)
is nontrivial on $\CO_F$.
\end{itemize}
\item Fix $N$ and $\mu\in F^\times$ with $v_F(\mu)<-N$ such that
\begin{itemize}
\item $\eta(\mu)$=1;
\item the condition (ii) above is satisfied;
\item the formulae of Proposition \ref{prop. i(X,Y)} hold for
$\wh{i}^\eta\left(X_0,i(\mu U_0)\right)$ and $\wh{i}(Y_0,i'(\mu
V_0))$ for all $i\in W_\fc$ and $i'\in W_{\fc'}$.
\end{itemize}
\item Set $\omega_0=\mu(1+\varpi^r\CO_F)U_0$ and
$\omega'_0=\mu(1+\varpi^r\CO_F)V_0$. Denote by $\fd$ (resp. $\fd'$)
the $F$-vector space generated by $U_0$ (resp. $V_0$), and fix a
complement $\fe$ (resp. $\fe'$) in $\fc(F)$ (resp. $\fc'(F)$), i.e.
$\fc(F)=\fd\oplus\fe$ (resp. $\fc'(F)=\fd'\oplus\fe'$). If
$U\in\fc(F)$ (resp. $V\in\fc'(F)$), denote by $U_\fd$ (resp.
$V_{\fd'}$) its projection on $\fd$ (resp. $\fd'$). We choose open
compact neighborhoods $\omega_\fe$ and $\omega'_{\fe'}$ of 0 in
$\fe$ and $\fe'$ small enough so that: if we set
$\omega=\omega_0\oplus\omega_\fe$ and
$\omega'=\omega'_0\oplus\omega'_{\fe'}$, then
\begin{itemize}
\item the sets $i(\omega)$ (resp. $i'(\omega')$), for $i\in
W_\fc$ (resp. $i'\in W_{\fc'}$), are mutually disjoint;
\item $\omega\subset\fc_\reg(F),\ \omega'\subset\fc'_\reg(F)$,
$\varphi(\omega')=\omega$, and therefore, for $U\in\omega,\
V\in\omega'$, they match with each other if and only if
$\varphi(V)=U$;
\item for each $X\in C(X_0)$ (resp. $Y\in C(Y_0)$), and each
$U\in\omega$ (resp. $V\in\omega'$),
$$\wh{i}^\eta(X,U)=\wh{i}^\eta(X,U_\fd),\quad
\wh{i}(Y,V)=\wh{i}(Y,V_{\fd'});$$
\item the function $\kappa$ is constant on $\omega$, which hence
equals to $\kappa(X_0)$.
\end{itemize}
\end{enumerate}

Define a function $f_\omega$ (resp. $f'_{\omega'}$) on $\omega$
(resp. $\omega'$) by
$$\begin{aligned}&f_\omega(U)=\psi\left(-\pair{X_0,U_\fd}\right),
\quad&\textrm{for }U\in\omega,\\
&f'_{\omega'}(V)=\psi\left(-\pair{Y_0,V_{\fd'}}\right),
\quad&\textrm{for }V\in\omega'. \end{aligned}$$ Now we fix a
function $f\in\CC_c^\infty(\fs(F))$ (resp.
$f'\in\CC_c^\infty(\fs'(F))$) such that
$$
\begin{aligned}
\Supp(f)\subset \omega^H,\quad &\textrm{and
}\kappa(U)I^\eta(U,f)=f_\omega(U) &\textrm{ for each }U\in\omega,\\
\Supp(f')\subset \omega'^{H'},\quad &\textrm{and
}I(V,f')=f'_{\omega'}(V) &\textrm{ for each }V\in\omega'.
\end{aligned}
$$
Then, we have, for $X\in\fs_\rs(F)$,
$$\begin{array}{lll}\kappa(X)I^\eta(X,f)=\left\{\begin{array}{ll}
 f_\omega(U)&\textrm{ if $X$ is $H$-conjugate to some $U\in\omega$} ,\\
 0&\textrm{ otherwise},\\
\end{array}\right.
\end{array}$$
and for $Y\in\fs'_\rs(F)$,
$$\begin{array}{lll}I(Y,f')=\left\{\begin{array}{ll}
 f'_{\omega'}(V)&\textrm{ if $Y$ is $H'$-conjugate to some $V\in\omega'$} ,\\
 0&\textrm{ otherwise}.\\
\end{array}\right.
\end{array}$$
Thus the assertions (i), (ii) and (iii) of the proposition follow
from the above construction and Lemma \ref{lem. compare gamma}.

To prove the assertion (iv), we observe that
$$\begin{aligned}
\kappa(X_0)\wh{I}^\eta(X_0,f)&=\kappa(X_0)\int_{\fs(F)}\wh{i}^\eta(X_0,U)
\kappa(U)f(U)|D^\fs(U)|^{-1/2}\ \d U\\
&=\abs{W_\fc}^{-1}\kappa(X_0)\int_{\fc(F)}\wh{i}^\eta(X_0,U)
\kappa(U)I^\eta(f,U)\ \d U\\
&=\kappa(X_0)\int_\omega\wh{i}^\eta(X_0,U)f_\omega(U)\ \d U\\
&=\sum_{i\in
W_\fc}\kappa(X_0)\int_\omega\eta(i)\kappa(U)\gamma_\psi(i(X_0),U)
\psi\left(\pair{i(X_0),U}\right)
f_\omega(U)\ \d U\\
&=\sum_{X\in
C(X_0)}\vol(\omega_\fe)\kappa(X)\kappa(U)\int_{\omega_0}\gamma_\psi(X,U_\fd)
\psi\left(\pair{X-X_0,U_\fd}\right)\ \d U_\fd\\
&=\sum_{X\in
C(X_0)}\vol(\omega)\kappa(X)\kappa(U)\\
&\times\int_{\CO_F}\gamma_\psi
\left(X,\mu(1+\varpi^r\alpha)U_0\right)
\psi\left(\pair{X-X_0,\mu(1+\varpi^r\alpha)U_0}\right)\ \d\alpha.
\end{aligned}$$
By condition (i),
$$\gamma_\psi\left(X,\mu(1+\varpi^r\alpha)U_0\right)=\gamma_\psi(X,\mu U_0),
\quad\textrm{for any }\alpha\in\CO_F.$$ If $X\neq X_0$, by condition
(ii),
$$\int_{\omega_0}
\psi\left(\pair{X-X_0,\mu(1+\varpi^r\alpha)U_0}\right)\
\d\alpha=0.$$ Therefore,
$$\kappa(X_0)\wh{I}^\eta(X_0,f)=\vol(\omega)\gamma_\psi(X_0,\mu U_0)\neq0.$$
The same computation goes for $\wh{I}(Y_0,f')$ and we get
$$\wh{I}(Y_0,f)=\vol(\omega')\gamma_\psi(Y_0,\mu V_0)\neq0.$$
Then the conclusion follows from Lemma \ref{lem. compare gamma} and
$\vol(\omega)=\vol(\omega')$.
\end{proof}

\section{Proof of Theorem \ref{thm. fourier preserve transfer}}
\label{sec: proof of the thm} In this section, we will prove Theorem
\ref{thm. fourier preserve transfer}.  We divide this theorem into
two parts, i.e., Theorems \ref{thm. half fourier 1} and \ref{thm.
half fourier 2} below.

\begin{thm}\label{thm. half fourier 1}
If $f$ is in $\CC_c^\infty(\fs(F))_0$, so is $\wh{f}$.
\end{thm}

\begin{thm}\label{thm. half fourier 2}
There exists a nonzero constant $c\in\BC$ satisfying that: if
$f\in\CC_c^\infty(\fs(F))$ and $f\in\CC_c^\infty(\fs'(F))$ are
smooth transfer of each other, then
$$\kappa(X)\wh{I}^\eta(X,f)=c\wh{I}(Y,f')$$ for any $X\in\fs_\rs(F)$
and $Y\in\fs'_\rs(F)$ such that $X\leftrightarrow Y$.
\end{thm}

We will use a local method to prove Theorem \ref{thm. half fourier
1}, and a global method to prove Theorem \ref{thm. half fourier 2},
as we have said before. The global method is a modification of that
of \cite{wa97}.

\subsection{Proof of Theorem \ref{thm. half fourier 1}}

By Lemma \ref{lem. case epsilon=1}, it suffices to only consider the
case $\fs'=\fs'_\epsilon$ when $\epsilon=1$. Throughout this
subsection, we assume that $\epsilon=1$.

Recall that $\fs_\rs(F)_0$ is the subset of elements in $\fs_\rs(F)$
coming from $\fs'_\rs(F)$. Let $\sC_0$ be the set of Cartan
subspaces of $\fs$ coming from those of $\fs'$, and $|\sC_0|$ a set
of representatives for $\FH$-conjugacy classes of Cartan subspaces
in $\sC_0$.

Let $f$ be in $\CC_c^\infty(\fs(F))_0$. Then, by the Weyl
integration formula, we have
$$\begin{aligned}
\wh{I}^\eta(X,f)&=\int_{\fs(F)}\wh{i}^\eta(X,Y)\kappa(Y)f(Y)
\abs{D^\fs(Y)}^{-1/2}\
\d Y\\
&=\sum_\fc\abs{W_\fc}^{-1}\int_{\fc_\reg(F)}\wh{i}^\eta(X,Y)
\kappa(Y)I^\eta(Y,f)\
\d Y\\
&=\sum_{\fc\in|\sC_0|}\abs{W_\fc}^{-1}\int_{\fc_\reg(F)}
\wh{i}^\eta(X,Y)\kappa(Y)I^\eta(Y,f)\ \d Y.\end{aligned}$$ Thus, to
show $\wh{I}^\eta(X,f)=0$ for any $X\notin\fs_\rs(F)_0$, it suffices
to show the following lemma.

\begin{lem}\label{lem. i=0}
For any $X\notin\fs_\rs(F)_0$ and any $Y\in\fs_\rs(F)_0$, we have
$\wh{i}^\eta(X,Y)=0$.
\end{lem}

\begin{proof}
First we need some preparation. Define an involution $\tau$ on
$\CC_c^\infty(\fs(F))$: $f^\tau(X):=f(X^t)$, where
$X=\begin{pmatrix}   0&A\\B&0
\end{pmatrix}\in\fs(F)$ and $X^t$ is its transpose. The following
two properties can be easily checked:
\begin{enumerate}
\item $\tau$ commutes with Fourier transform, i.e.,
$(\wh{f})^\tau=\wh{f^\tau}$;
\item for $X=\begin{pmatrix}   0&A\\B&0
\end{pmatrix}\in\fs_\rs(F)$, $I^\eta(X,f^\tau)=\eta(\det
AB)I^\eta(X,f)$ for any $f\in\CC_c^\infty(\fs(F))$.
\end{enumerate}
In particular, if $Y\in\fs_\rs(F)_0$, then
$I^\eta(Y,f^\tau)=I^\eta(Y,f)$; if an elliptic $X$ is not in
$\fs_\rs(F)_0$, then $I^\eta(X,f+f^\tau)=0$.

Now let $X\notin\fs_\rs(F)_0$ be an elliptic element. For any
$f\in\CC_c^\infty(\fs(F))_0$, by the above discussion, we see that
$$0=\wh{I}^\eta(X,f+f^\tau)=2\sum_{\fc\in|\sC_0|}
\abs{W_\fc}^{-1}\int_{\fc_\reg(F)}
\wh{i}^\eta(X,Y)\kappa(Y)I^\eta(Y,f)\ \d Y.$$ For any
$Y_0\in\fs_\rs(F)_0$ we may choose a specific
$f_0\in\CC_c^\infty(\fs(F))_0$ so that
$$\sum_{\fc\in|\sC_0|}\abs{W_\fc}^{-1}\int_{\fc_\reg(F)}
\wh{i}^\eta(X,Y)\kappa(Y)I^\eta(Y,f)\ \d Y=\wh{i}^\eta(X,Y_0).$$
Therefore $\wh{i}^\eta(X,Y)=0$ for any elliptic
$X\notin\fs_\rs(F)_0$ and any $Y\in\fs_\rs(F)_0$.

Now let $X\notin\fs_\rs(F)_0$ be a non-elliptic element. It suffices
to assume that $X$ is of the form $X(A)=\begin{pmatrix} 0&{\bf1}_n\\
A&0
\end{pmatrix}$ for some $A\in\GL_{n,\rs}(F)$. Since
$X\notin\fs_\rs(F)_0$ is non-elliptic, we can assume that $A$ is of
the form $\begin{pmatrix}   A_1&0\\ 0&A_2
\end{pmatrix}$ where $A_1\in\GL_{n_1,\rs}(F)$ is elliptic and not in
$\N(\GL_{n_1}(E))$, and $A_2$ is in $\GL_{n_2,\rs}(F)$. Recall the
discussions in \S\ref{sec. parabolic induction}. There is a subspace
$\fr\simeq\fs_{n_1}\times\fs_{n_2}$ of $\fs$ such that $X\in\fr$.
Moreover, under the natural isomorphism
$\iota:\fr\stackrel{\sim}{\ra}\fs_{n_1}\times\fs_{n_2}$, the image
of $X$ is $(X_1,X_2)$ where $X_i=X(A_i)$ for $i=1,2$. Write
$\fs_i=\fs_{n_i}$ for $i=1,2$. Let $M=H_1\times H_2$ where
$H_i=H_{n_i}$ for $i=1,2$. Then $M$ acts on $\fr$ naturally. For
$Z\in\fr_\rs(F)$, let $\wh{i}^{\eta,\fr}(Z,\cdot)$ be the kernel
function that represents the distribution
$f\mapsto\wh{I}^{\eta,M}(Z,f)$ for $f\in\CC_c^\infty(\fr(F))$. It is
obvious that
$$\wh{i}^{\eta,\fr}(Z,Y)=\wh{i}^{\eta,\fs_1}(Z_1,Y_1)
\wh{i}^{\eta,\fs_2}(Z_2,Y_2)$$ where $(Z_1,Z_2)$ and $(Y_1,Y_2)$ are
the images of $Z$ and $Y$ under $\iota$ in $\fs_1\times\fs_2$
respectively, and $\wh{i}^{\eta,\fs_i}(Z_i,\cdot)$ is the kernel
function that represents the distribution
$f\mapsto\wh{I}^{\eta,H_i}(Z_i,f)$ for $f\in\CC_c^\infty(\fs_i(F))$
for $i=1,2$.

By Proposition \ref{prop. parabolic induction}, we have
$$\wh{i}^\eta(X,Y)=\sum_{Y'}\wh{i}^{\eta,\fr}(X,Y'),$$ where $Y'$
runs over a set of representatives for the finitely many
$M$-conjugacy classes of elements of $\fr(F)$ which are
$H$-conjugate to $Y$. Therefore we can and do assume that
$Y\in\fs_\rs(F)_0$ is in $\fr(F)$ and of the form $Y=(Y_1,Y_2)$
under the natural map $\iota$ where $Y_i\in\fs_{i,\rs}(F)_0$. Then
$$\wh{i}^{\iota,\fr}(X,Y)=\wh{i}^{\eta,\fs_1}(X_1,Y_1)
\wh{i}^{\eta,\fs_2}(X_2,Y_2)=0,$$ since $X_1\notin\fs_{1,\rs}(F)_0$
is elliptic and $Y_1\in\fs_{1,\rs}(F)_0$. We complete the proof of
the lemma.
\end{proof}

\subsection{A result on convergence}
Now let $k$ be a number field, $\BA=\BA_\infty\times\BA_\f$ its ring
of adeles. Let $k'$ be a quadratic field extension of $k$, $\BD$ a
quaternion algebra over $k$ containing $k'$, and $\eta$ the
quadratic character of $\BA^\times/k^\times$ attached to $k'$ by the
class field theory. We define the global symmetric pairs $(\FG,\FH)$
and $(\FG',\FH')$ over $k$ with respect to $k'$ and $\BD$ similarly
as the local cases. Let $\fs,\fs'$ be the corresponding global ``Lie
algebras'' associated to $(\FG,\FH)$ and $(\FG',\FH')$ respectively,
which are defined over $k$. Denote by $\CS(\fs(\BA))$ (resp.
$\CS(\fs'(\BA))$) the space of Schwartz functions on $\fs(\BA)$
(resp. $\fs'(\BA)$). Denote by $\FH(\BA)^1$ the set of
$(h_1,h_2)\in\FH(\BA)$ such that $\abs{\det h_1}=\abs{\det h_2}=1$,
and by $\FH'(\BA)^1$ the set of $h\in\FH'(\BA)$ such that $\abs{\det
h}=1$. The groups $\FH(\BA)^1$ and $\FH'(\BA)^1$ are subgroups of
$\FH(\BA)$ and $\FH'(\BA)$ respectively. We have the following
theorem concerning the issue about convergence.

\begin{thm}\label{thm. convergence}
For each $\phi\in\CS(\fs(\BA))$,
$$\int_{\FH(k)\bs\FH(\BA)^1}\sum_{X\in\fs_\el(k)}|\phi(X^h)|\ \d
h<\infty.$$ Similarly, for each $\phi'\in\CS(\fs'(\BA))$,
$$\int_{\FH'(k)\bs\FH'(\BA)^1}\sum_{Y\in\fs'_\el(k)}|\phi'(Y^h)|
\ \d h<\infty.$$
\end{thm}

\begin{proof}
(1) Now we prove the assertion for $\phi\in\CS(\fs(\BA))$. Here we
still write $Z=(X,Y)\in\fs=\fg\fl_n\oplus\fg\fl_n$ and $h\cdot
Z=(\Ad h)Z$ where $h\in\FH$ for convenience. Recall that $Z=(X,Y)$
is in $\fs_\el(k)$ if and only if neither $XY$ nor $YX$ is contained
in a proper parabolic subgroup of $\GL_n(k)$. Let $\FP_0$ be the
minimal parabolic subgroup of $\GL_n$ consisting of the
upper-triangular one. Put $\FP=\FP_0\times\FP_0\subset\FH$. Identify
$\BR^\times_+$ with the subgroup of $\BA^\times_\infty$ consisting
of elements whose components at each place are the same and belong
to $\BR^\times_+$. For each real number $c>0$, put $A_c^0$ the set
of $a=\diag(a_1,...,a_n)\in\SL_n(\BR)$ such that
$\frac{a_i}{a_{i+1}}\geq c$ for all $1\leq i\leq n-1$ and
$a_i\in\BR_+^\times$ for all $1\leq i\leq n$, and set
$$A_c=A_c^0\times A_c^0\subset\FH(\BR)\subset\FH(\BA_\infty).$$ By
reduction theory, we know that there exists a maximal compact
subgroup $\FK$ of $\FH(\BA)$, a compact subset
$\omega\subset\P(\BA)\cap\FH(\BA)^1$ and a $c>0$ such that, if we
set $$\CG=\{\ pak;\ p\in\omega,a\in A_c,k\in\FK\},$$ we have the
equality $\FH(\BA)^1=\FH(k)\CG$, and thus, for each measurable
function $\phi$ on $\FH(k)\bs\FH(\BA)^1$ with valued $\geq0$, the
integral $$\int_{\FH(k)\bs\FH(\BA)^1}\phi(x)\ \d x$$ is convergent
if and only if the integral
$$\int_{\CG}\phi(x)\ \d x$$ is such so. Fix such $\FK,\omega,c$. Then
the integral is convergent if there exists $C\geq0$ such that for
each $p\in\omega,k\in\FK$,
$$\int_{A_c}\sum_{Z\in\fs_\el(k)}|\phi\left((pak)\cdot Z\right)|
\delta_\FP(a)^{-1} \ \d a\leq C,$$ where $\delta_\FP$ is the modulus
character of $\FP$. There exists a compact set
$\Omega\subset\FH(\BA)^1$ such that for all $ p\in\omega,a\in
A_c,k\in \FK$, $a^{-1}pak\in\Omega$. Then there exists
$\phi'\in\CS(\fs(\BA))$ such that for all $Z\in\fs(\BA)$ and
$h\in\Omega$, we have $|\phi(h\cdot Z)|\leq\phi'(Z)$. It suffices to
consider $\phi'$ of the form $\phi'=\phi'_\infty\otimes\phi'_\f$,
where
$\phi'_\infty\in\CS(\fs(\BA_\infty)),\phi'_\f\in\CS(\fs(\BA_\f))$
both with $\geq0$ valued, and suffices to consider the integral
$$\int_{A_c}\sum_{Z\in\fs_\el(k)}\phi'(a\cdot Z)\delta_\FP(a)^{-1}\ \d
a.$$ Choose an $\CO_k$-lattice $L$ in $\fs(k)$ such that
$\fs(k)\cap\Supp(\phi_\f)\subset L$. Denote $L_\el=L\cap\fs_\el(k)$.
Since $\phi'_\f(a\cdot Z)=\phi'_\f(Z)$, it suffices to consider the
integral $$\int_{A_c}\sum_{Z\in L_\el(k)}\phi'_\infty(a\cdot
Z)\delta_\FP(a)^{-1}\ \d a.$$ If $x_v\in k_v$ and $v$ is an infinite
place of $k$, write $\abs{x_v}$ for the usual absolute value of
$x_v$. For every $x=(x_v)\in\BA_\infty$, put
$\abs{x}=\max_v\abs{x_v}$. For $X=(x_{i,j})\in\fg\fl_n(\BA_\infty)$,
put $\abs{X}=\max_{i,j}\abs{x_{i,j}}$. For
$Z=(X,Y)\in\fs(\BA_\infty)$, write
$\abs{Z}=\max\{\abs{X},\abs{Y}\}$. Then the following lemma implies
the theorem.

\begin{lem}\label{lem. for convergence}
Assume that $n\geq2$. There is a positive valued polynomial function
$P$ on the real vector space $\fs(\BA_\infty)$, which depends on $L$
and $c$, such that
$$P(a\cdot Z)\geq\left(\prod_{i=1}^{n-1}\frac{a_i}{a_{i+1}}
\cdot\frac{b_i}{b_{i+1}}\right)\abs{Z},$$ for all
$a=\diag(a_1,...,a_n,b_1,...,b_n)\in A_c$ and all $Z\in L_\el$.
\end{lem}

\begin{proof} Take a positive valued polynomial
function $P_1$ on $\fs(\BA_\infty)$ such that
\[
  P_1(X,Y)\geq \max\{\abs{XY},\abs{YX}\},\quad \textrm{for all }
  (X,Y)\in \fs(\BA_\infty).
\] Take a positive number $c_L$ such that
\[
  (X,Y)\in L,\ d \textrm{ is a nonzero entry of $XY$ or $YX$}
  \Rightarrow\abs{d}\geq c_L.
\] Let $a=(a_1,a_2,\cdots, a_n, b_1, b_2,\cdots, b_n)$ in $A_c$ and let
$Z=(X,Y)=\left((x_{i,j}), (y_{i,j})\right)$ in $L_\el$. Write
$(u_{i,j})$ for $XY$. Fix $i_0=1,2,\cdots, n-1$. Since $XY$ is not
contained in a proper parabolic subalgebra of $\gl_n(k)$, there are
$i\geq i_0+1$ and $j\leq i_0$ such that
\[
  u_{i,j}\neq 0.
\]
Then
\[
   \abs{u_{i,j}}\geq c_L,
\]
and we have $$\begin{aligned}
  P_1(a\cdot Z)&\geq \abs{a_i u_{i,j} a_j^{-1}}\geq c_L a_i a_j^{-1}
  \geq c_L c^{i-i_0-1} c^{i_0-j} a_{i_0} a_{i_0+1}^{-1}\\
  &\geq c_L c^{n-2}a_{i_0} a_{i_0+1}^{-1}.
\end{aligned}$$ This implies that
\[
  a_n^{-1}=\prod_{i=1}^{n-1} (a_i/a_{i+1})^{\frac{i}{n}}\leq
  \left(\frac{1}{c_L c^{n-2}} P_1(a\cdot Z)\right)^{\frac{n-1}{2}},
\]
and
\[
  a_1=\prod_{i=1}^{n-1} (a_i/a_{i+1})^{\frac{n-i}{n}}\leq
  \left(\frac{1}{c_L c^{n-2}} P_1(a\cdot Z)\right)^{\frac{n-1}{2}}.
\]
Similarly,
\[
  P_1(a\cdot Z)\geq  c_L c^{n-2} b_{i_0} b_{i_0+1}^{-1},
\]
and
\[
  b_1, b_n^{-1}\leq \left(\frac{1}{c_L c^{n-2}}
  P_1(a\cdot Z)\right)^{\frac{n-1}{2}}.
\]
For all $i,j=1,2,\cdots, n$, we have
\[
  \abs{a\cdot Z}\geq a_i \abs{ a_{i,j}} b_j^{-1}.
\]
Therefore, $$\begin{aligned}
 \abs{a_{i,j}}&\leq a_i^{-1}b_j \abs{a\cdot Z}\leq c^{-(n-i)-(j-1)} a_n^{-1} b_1
 \abs{a\cdot Z}\\
 &\leq c^{-(2n-2)} \left(\frac{1}{c_L c^{n-2}}\right)^{n-1}P_1(a\cdot Z)^{n-1}
 \abs{a\cdot Z}.
\end{aligned}$$ Similarly,
\[
   \abs{b_{i,j}}\leq c^{-(2n-2)} \left(\frac{1}{c_L c^{n-2}}\right)^{n-1}
   P_1(a\cdot Z)^{n-1} \abs{a\cdot Z}.
\]
By timing $\frac{a_i}{a_{i+1}}$ and $\frac{b_i}{b_{i+1}}$ on both
sides of the above inequality, we get the lemma.
\end{proof}

(2) The convergence of the second integral (for
$\phi'\in\CS'(\fs(\BA))$) can be deduced easily from Lemma 10.8 of
\cite{wa97}, since the twisted conjugation by $A_c$ is the usual
conjugation.
\end{proof}

By the above theorem, we have a well-defined distribution $I^\eta$
on $\fs(\BA)$, defined by
$$I^\eta(\phi)=\int_{\FH(k)\bs\FH(\BA)^1}\sum_{X\in\fs_\el(k)}
\phi(X^h)\eta(h)\ \d h,\quad \phi\in\CS(\fs(\BA)),$$ and a
well-defined distribution $I$ on $\fs'(\BA)$, defined by
$$I(\phi')=\int_{\FH'(k)\bs\FH'(\BA)^1}\sum_{Y\in\fs'_\el(k)}\phi'(Y^h)
\ \d h,\quad \phi'\in\CS(\fs'(\BA)).$$ If $\phi=\prod_v \phi_v,
\phi'=\prod_v \phi'_v$, it is routine to see that
$$I^\eta(\phi)=\sum_{X\in[\fs_\el(k)]}\tau(\FH_{X})\prod_v\kappa_v(X)
I^\eta(X,\phi_v),$$
$$I(\phi')=\sum_{Y\in[\fs'_\el(k)]}\tau(\FH'_{Y})\prod_vI(Y,\phi'_v),$$
where
$$\tau(\FH_X)=\vol(\FH_X(k)\bs(\FH_X\cap\FH(\BA)^1)),\quad
\tau(\FH'_Y)=\vol(\FH_Y'(k)\bs(\FH_Y'\cap\FH'(\BA)^1),$$
$[\fs_\el(k)]$ denotes the set of $\FH(k)$-orbits in $\fs_\el(k)$,
and $[\fs'_\el(k)]$ denotes the set of $\FH'(k)$-orbits in
$\fs'_\el(k)$. If $X\in\fs_\rs(k)$ and $Y\in\fs'_\rs(k)$ so that
$X\leftrightarrow Y$, then $\FH_X\simeq\FH'_Y$ (same reason as the
local case). We choose Haar measures on $\FH_X(\BA)$ and
$\FH'_Y(\BA)$ so that they are compatible. Thus, if
$X\in\fs_\el(k),Y\in\fs'_\el(k)$ such that $X\leftrightarrow Y$, we
have
$$\tau(\FH_X)=\tau(\FH'_Y).$$

\subsection{Proof of Theorem \ref{thm. half fourier 2}}
Now, we fix $f\in\CC_c^\infty(\fs(F))$ and
$f'\in\CC_c^\infty(\fs'(F))$ so that they are smooth transfer of
each other. Here we allow that $f$ may not lie in
$\CC_c^\infty(\fs(F))_0$, as we have mentioned in the proof of
Proposition \ref{prop. conjecture 2 implies conjecture 1}. We also
refer the reader to the proof of Proposition \ref{prop. conjecture 2
implies conjecture 1} to see the definition of smooth transfer in
this more general situation.

Fix $X_0\in\fs_\rs(F),Y_0\in\fs'_\rs(F)$ such that
$X_0\leftrightarrow Y_0$. Our aim is to search for a nonzero
constant $c$ which is independent of $f,f',X_0$ and $Y_0$ such that
$$\kappa(X_0)\wh{I}^\eta(X_0,f)=c\wh{I}(Y_0,f').$$ In the following,
we choose some global data.

\paragraph{Fields}
We choose a number field $k$, a quadratic field extension $k'$ of
$k$, and a quaternion algebra $\BD$ over $k$ containing $k'$ so
that:
\begin{enumerate}
\item $k$ is totally imaginary;
\item there exists a finite place $w$ of $k$ such that $k_w\simeq F,\ k'_w\simeq
E$ and $\BD(k_w)\simeq \D$;
\item there exists another finite place $u$ of $k$ such that $u$ is inert
in $k'$.
\end{enumerate}
Such a number field $k$ and a quaternion algebra $\BD$ do exist (cf.
\cite[Proposition in \S11.1]{wa97}). From now on, we identify $k_w$
with $F$, $k'_w$ with $E$ and $\BD(k_w)$ with $\D$. Denote by $\BA$
the ring of adeles of $k$, by $\CO_k$ the ring of integers of $k$,
and $\CO_{k'}$ the ring of integers of $k'$. Fix a continuous
character $\BA/k$ whose local component at $w$ is our fixed
character $\psi$ of $k_w$. Denote by $\psi$ this global character,
when there is no confusion.

\paragraph{Groups} We define the global symmetric pairs $(\FG,\FH)$
and $(\FG',\FH')$ over $k$ with respect to $k'$ and $\BD$ similarly
as the local case. We still use $\fh$ and $\fh'$ to denote the Lie
algebras of $\FH$ and $\FH'$ respectively, use $\fs$ and $\fs'$ to
denote the global Lie algebras corresponding to $(\FG,\FH)$ and
$(\FG',\FH')$ respectively, if there is no confusion. Thus
$X_0\in\fs_\rs(k_w)$ and $Y_0\in\fs'_\rs(k_w)$.

\paragraph{Places}
Denote by $V$ (resp. $V_\infty,\ V_\f$) the set of all (resp.
archimedean, non-archimedean) places of $k$. Fix two
$\CO_k$-lattices:
$\FL=\fg\fl_n(\CO_k)\oplus\fg\fl_n(\CO_k)\subset\fs(k)$ and
$\FL'=\fg\fl_n(\CO_{k'})\subset\fs'(k)$. For each $v\in V_\f$, put
$\FL_v=\FL\otimes_{\CO_k}\CO_{k,v},
\FL'_v=\FL'\otimes_{\CO_k}\CO_{k,v}$. We fix a finite set $S\subset
V$ such that:
\begin{enumerate}
\item $S$ contains $u,w$ and $V_\infty$;
\item for each $v\in V-S$, everything is unramified, i.e. $\FG$ and
$\FG'$ are unramified over $k_v$, $\FL_v$ and $\FL'_v$ are self-dual
with respect to $\psi_v$ and $\pair{\ ,\ }$.
\end{enumerate}
We denote by $S'$ the subset $S-V_\infty-\{w\}$ of $S$.

\paragraph{Orbits}
For each $v\in V_\f$, we choose an open compact subset
$\Omega_v\subset\fs'(k_v)$ such that:
\begin{enumerate}
\item if $v=w$, we require that: $Y_0\in\Omega_w\subset\fs'_\rs(k_w)$,
$\wh{I}(\cdot,f')$ is constant on $\Omega_w$, and
$\kappa(\cdot)\wh{I}^\eta(\cdot,f)$ is constant and hence equal to
$\kappa(X_0)\wh{I}^\eta(X_0,f)$ on the set of $X\in\fs_\rs(k_w)$
which matches an element $Y$ in $\Omega_w$,
\item if $v=u$, we require $\Omega_u\subset\fs'_\el(k_u)$;
\item if $v\in S$ but $v\neq w,u$,
choose $\Omega_v$ to be any open compact subset;
\item if $v\in V_\f-S$, let $\Omega_v=\FL'_v$.
\end{enumerate}
Then by the strong approximation theorem, there exists
$Y^0\in\fs'(k)$ such that $Y^0\in\Omega_v$ for each $v\in V_\f$.
Furthermore, by the condition (ii) above, $Y^0\in\fs'_\el(k)$. Take
an element $X^0\in\fs_\el(k)$ such that $X^0\leftrightarrow Y^0$.

\paragraph{Functions}
For each $v\in V$, we choose functions $\phi_v\in\CS(\fs(k_v))$ and
$\phi'_v\in\CS(\fs'(k_v))$ as follows:
\begin{enumerate}
\item if $v=w$, let $\phi_v=f$ and $\phi'_v=f'$;
\item if $v\in S'$, by Proposition \ref{prop. local prop}, we
require that:
\begin{itemize}
\item if $X_v\in\Supp(\phi_v)$, there exists
$Y_v\in\fc'_{Y^0}(k_v)$ such that $X_v\leftrightarrow Y_v$, where we
denote by $\fc'_{Y^0}$ the Cartan subspace in $\fs'$ containing
$Y^0$;
\item if $Y_v\in\Supp(\phi'_v)$, there exists
$Y'_v\in\fc'_{Y^0}(k_v)$ such that $Y_v$ and $Y'_v$ are
$\FH'(k_v)$-conjugate;
\item $\phi_v$ is a transfer of $\phi'_v$;
\item $\kappa_v(X^0)\wh{I}^\eta(X^0,\phi_v)=c_v\wh{I}(Y^0,\phi'_v)\neq0$,
where $c_v=\gamma_\psi(\fh(k_v))\gamma_\psi(\fh'(k_v))^{-1}$;
\end{itemize}
\item for $v\in V-S$, set $\phi_v={\bf1}_{\FL_v},\phi_v'={\bf1}_{\FL'_v}$;
then $\phi_v=\wh{\phi}_v,\phi'_v=\wh{\phi'_v}$, and by Lemma
\ref{lem. fund lem} we have
$$\kappa_v(X^0)\wh{I}^\eta(X^0,\phi_v)=\kappa_v(X^0)I^\eta(X^0,\phi_v)
=I(Y^0,\phi'_v)=\wh{I}(Y^0,\phi'_v);$$
\item for $v\in V_\infty$, identifying $(\FH(k_v),\fs(k_v))$ with
$(\FH'(k_v),\fs'(k_v))$, we choose $\phi_v=\phi'_v\in\CS(\fs(k_v))$
such that:
\begin{itemize}
\item $\wh{I}^\eta(X^0,\phi_v)=\wh{I}(Y^0,\phi'_v)\neq0$;
\item if $X\in\fs(k)$ is $\FH(k_v)$-conjugate to an element in
the support of $\wh{\phi_v}$ at each place $v\in V$, then $X$ is
$\FH(k)$-conjugate to $X^0$;
\item if $Y\in\fs'(k)$ is $\FH'(k_v)$-conjugate to an element in
the support of $\wh{\phi'_v}$ at each place $v\in V$, then $Y$ is
$\FH'(k)$-conjugate to $Y^0$.
\end{itemize}
This is possible. The key point is that, by invariant theory, we
have natural maps (cf. Remark \ref{rem. matching categorical
quotient})
$$\fs'/\FH'\incl\fs/\FH\lra\FA_k^n,$$ where $\FA_k^n$ is the
$n$-dimensional affine space over $k$ so that
$\FA_k^n=\Spec(\CO(\fs)^\FH)$. We refer the reader to \cite[Lemme in
\S10.7]{wa97} for the proof in the endoscopic case, and a similar
argument is also valid here.
\end{enumerate}
Set $\phi\in\CS(\fs(\BA))$ and $\phi'\in\CS(\fs'(\BA))$ to be:
$$\phi=\prod_{v\in V}\phi_v,\quad \phi'=\prod_{v\in V}\phi'_v.$$

\paragraph{Final proof}
According to the conditions on $\phi_u$ (resp. $\phi'_u$), we know
that if $X\in\fs(k)$ (resp. $Y\in\fs'(k)$) is such that
$X\in\Supp(\phi)^{\FH(\BA)}$  (resp.
$Y\in\Supp(\phi')^{\FH'(\BA)}$), then $X\in\fs_\el(k)$ (resp.
$Y\in\fs'_\el(k)$). Here we use $\Supp(\phi)^{\FH(\BA)}$ to denote
the union of $\FH(\BA)$-orbits intersecting $\Supp(\phi)$, and
$\Supp(\phi')^{\FH'(\BA)}$ to denote the union of $\FH'(\BA)$-orbits
intersecting $\Supp(\phi')$. Suppose that $X\in\fs_\el(k)$ is such
that
$$I^\eta(X,\phi)=\prod_{v\in V} I^\eta(X,\phi_v)\neq0.$$ Then, by
the conditions on $\phi_v$, $X$ comes from $\fs'(k_v)$ at each place
$v$ not equal to $w$. We claim that $X$ must come from $\fs'(k)$. If
not, there exists at least two places $v_1$ and $v_2$ such that $X$
does not come from $\fs'(k_v)$, which is a contradiction.  Therefore
we have
$$I^\eta(\phi)=I(\phi'),$$
since $\phi_v$ is a transfer of $\phi'_v$ at each place $v$ not
equal to $w$ and is a partial transfer of $\phi'_v$ at the place
$v=w$ by the requirements we have imposed.

On the other hand, according to the conditions on $\wh{\phi_v}$ and
$\wh{\phi'_v}$, we know that if $X\in\fs(k)$ (resp. $Y\in\fs'(k)$)
is such that $X\in\Supp(\wh{\phi})^{\FH(\BA)}$ (resp.
$Y\in\Supp(\wh{\phi'})^{\FH'(\BA)}$) then $X$ is $\FH(k)$-conjugate
to $X^0$ (resp. $Y$ is $\FH'(k)$-conjugate to $Y^0$).

By Poisson summation formula, we have
$$\sum_{X\in\fs(k)}\phi(X^h)=\sum_{X\in\fs(k)}\wh{\phi}(X^h),\quad
\forall\ h\in\FH(\BA),$$ and
$$\sum_{Y\in\fs'(k)}\phi'(Y^h)=\sum_{Y\in\fs'(k)}\wh{\phi'}(Y^h),
\quad \forall\ h\in\FH'(\BA).$$ Therefore, by the conditions on
$\phi$ and $\phi'$, we have
$$I^\eta(\phi)=I^\eta(\wh{\phi}),\quad I(\phi')=I(\wh{\phi'}).$$
Hence we obtain
$$I^\eta(\wh{\phi})=I(\wh{\phi'}),$$
or equivalently,
$$\tau(\FH_{X^0})\prod_{v\in V}\kappa_v(X^0)\wh{I}^\eta(X^0,\phi_v)
=\tau(\FH'_{Y^0})\prod_{v\in V}\wh{I}(Y^0,\phi'_v).$$ Note that for
$v\in V-S$ , we have
$$\kappa_v(X^0)\wh{I}^\eta(X^0,\phi_v)=\wh{I}(Y^0,\phi'_v)\neq0,$$
and for almost all $v\in V-S$,
$$\kappa_v(X^0)\wh{I}^\eta(X^0,\phi_v)=\wh{I}(Y^0,\phi'_v)=1.$$
For $v\in S'$ and $v\in V_\infty$, we have
$$\kappa_v(X^0)\wh{I}^\eta(X^0,\phi_v)=c_v\wh{I}(Y^0,\phi'_v)\neq0.$$
Therefore
$$\kappa_w(X^0)\wh{I}^\eta(X^0,f)=c\wh{I}(Y^0,f'),$$ where
$$c=(\prod_{v\in S'}c_v)^{-1}=\prod_{v\in S'}\gamma_\psi(\fh(k_v))^{-1}
\gamma_\psi(\fh'(k_v)).$$ Notice that if $v\in V_\infty$ or $v\in
V-S$, $$\gamma_\psi(\fh(k_v))=\gamma_\psi(\fh'(k_v))=1.$$ Also
notice that $$\prod_{v\in V}\gamma(\fh(k_v))=\prod_{v\in
V}\gamma_\psi(\fh'(k_v))=1.$$ Therefore
$$c=\gamma_\psi(\fh(k_w))\gamma_\psi(\fh'(k_w))^{-1.}.$$ Since
$$\kappa_w(X_0)\wh{I}^\eta(X_0,f)=\kappa_w(X^0)\wh{I}^\eta(X^0,f),
\quad \wh{I}(Y_0,f')=\wh{I}(Y^0,f'),$$ we complete the proof of the
theorem.

\paragraph{Acknowledgements} This work was supported by the National Key
Basic Research Program of China (No. 2013CB834202). Needless to say,
the remarkable work of Jean-Loup Waldspurger on endoscopic transfer
has a huge influence on this article. The author is very grateful to
Wei Zhang for suggesting this problem and sharing the idea that
Waldspurger's method might apply to this situation, to Binyong Sun
for proving Theorem \ref{thm. convergence} which is crucial for our
method. He also thanks Kimball Martin for communicating their work
\cite{fmw} and sending their preprint, and thanks Wen-Wei Li for
helpful discussions. He expresses gratitude to Ye Tian and Linsheng
Yin for their constant encouragement and support. He would like to
thank the anonymous referee for explaining how to prove Theorem
\ref{thm. half fourier 1}, which improves the main result of the
article greatly, and many other useful comments.

\s{\small Chong Zhang\\
School of Mathematical Sciences, Beijing Normal University,\\
Beijing 100875, P. R. China.\\
E-mail address: \texttt{zhangchong@bnu.edu.cn}}

\end{document}